\documentclass[leqno,12pt]{article}

\usepackage[all]{xy}
\usepackage{amsthm}
\usepackage{amssymb} % \mathbb
\usepackage{amsmath,amscd} % \mathbb
\usepackage{mathrsfs}
\usepackage{color}
\usepackage{enumerate}
\usepackage{longtable}

\def\asid{\textsf{ASID}}

\def\GrMod{\operatorname{\mathsf{GrMod}}}

\def\turn!{\textup{!`}}

\def\op{\textup{op}}

\def\soc{\operatorname{soc}}

\def\pd{\mathop{\mathrm{pd}}\limits}

\def\injdim{\mathop{\mathrm{id}}\limits}

\def\ind{\mathop{\mathrm{ind}}\nolimits}

\def\ac{\operatorname{ac}}

\def\lb{\operatorname{\mathrm{lb}}} 
\def\ub{\operatorname{\mathrm{ub}}}

\def\mrlp{\mathrm{lp}}

\def\rank{\operatorname{\mathsf{rank}}}

\def\CM{\operatorname{\mathsf{CM}}}

\def\grCM{\operatorname{\mathsf{CM}}^{\Bbb{Z}}}

\def\grmod{\operatorname{mod}^{\Bbb{Z}}}

\def\mrb{\mathrm{b}}

\def\mre{\mathrm{e}}
\def\mrac{\mathrm{ac}}

\def\Sing{\operatorname{Sing}}

\def\grSing{\operatorname{Sing}^{\ZZ}}

\def\lpSing{\operatorname{Sing}^{\Bbb{Z}}_{\mrlp} }

\def\stabCM{\operatorname{\underline{\mathsf{CM}}}}
\def\stabgrCM{\operatorname{\underline{\mathsf{CM}}}^{\Bbb{Z}}}
\def\stablpCM{\operatorname{\underline{ \mathsf{CM} }^{\Bbb{Z}} _{\mrlp} }}

\def\stabmod{\operatorname{\underline{mod}}}

\def\stabgrmod{\operatorname{\underline{mod}}^{\Bbb{Z}}}

\def\thick{\mathop{\mathsf{thick}}\nolimits}

\def\Loc{\mathop{\mathsf{Loc}}\nolimits}

\def\leftvee{\vartriangleleft} %{\textup{¶}\vee}
\def\rightvee{\vartriangleright} %{\textup{‰E}\vee}

\def\kk{{\mathbf k}}

\def\NN{{\Bbb N}}

\def\ZZ{{\Bbb Z}}

\def\cA{{\cal A}}

\def\cC{{\cal C}}

\def\cH{{\cal H}}

\def\cT{{\cal T}}

\def\sfq{{\mathsf{q}}}

\def\sfA{{\mathsf{A}}}

\def\sfC{{\mathsf{C}}}
\def\sfD{{\mathsf{D}}}
\def\sfE{{\mathsf{E}}}
\def\sfF{{\mathsf{F}}}

\def\sfK{{\mathsf{K}}}

\def\sfO{{\mathsf{O}}}

\def\sfT{{\mathsf{T}}}

\def\sfq{{\mathsf{q}}}

\def\scrE{{\mathscr{E}}}
\def\scrF{{\mathscr{F}}}
\def\scrG{{\mathscr{G}}}
\def\scrH{{\mathscr{H}}}
\def\scrI{{\mathscr{I}}}
\def\scrJ{{\mathscr{J}}}
\def\scrK{{\mathscr{K}}}
\def\scrL{{\mathscr{L}}}

\def\scrS{{\mathscr{S}}}
\def\scrT{{\mathscr{T}}}

\def\tuD{{\textup{D}}}

\def\tuH{{\textup{H}}}

\def\tuT{{\textup{T}}}

\def\tuZ{{\textup{Z}}}

\def\frki{{\mathfrak{i}}}

\def\frkp{{\mathfrak{p}}}

\def\frkt{{\mathfrak{t}}}

\def\id{\operatorname{id}}
\def\op{\operatorname{op}}

\def\mod{\operatorname{mod}}
\def\Mod{\operatorname{Mod}}
\def\GrMod{\operatorname{Mod}^{\mathbb{Z}}}

\def\Ker{\mathop{\mathrm{Ker}}\nolimits}

\def\proj{\operatorname{proj}}

\def\grproj{\operatorname{proj}^{\mathbb{Z}}} 
 
\def\GrProj{\operatorname{Proj}^{\mathbb{Z}}}

\def\Proj{\operatorname{Proj}}

\def\grProj{\operatorname{Proj}^{\mathbb{Z}}}

\def\add{\operatorname{add}}

\def\image{\operatorname{Im}}

\def\Hom{\operatorname{Hom}}

\def\grHom{\operatorname{HOM}}
\def\End{\operatorname{End}}

\def\Ext{\operatorname{Ext}}
\def\grExt{\operatorname{EXT}}

\def\gldim{\operatorname{gldim}}

\newcommand{\RHom}{\operatorname{\Bbb{R}Hom}}
\def\grRHom{\operatorname{\Bbb{R}HOM}}
\newcommand{\lotimes}{\otimes^{\Bbb{L}}}

\newcommand{\cone}{\operatorname{\mathsf{cn}}}

\def\RHom{\operatorname{\mathbb{R}Hom}}
\def\grRHom{\operatorname{\mathbb{R}HOM}}

\newtheorem{lemma}{Lemma}[section]
\newtheorem{proposition}[lemma]{Proposition}
\newtheorem{theorem}[lemma]{Theorem}
\newtheorem{corollary}[lemma]{Corollary}

\newtheorem{claim}[lemma]{Claim}

\theoremstyle{definition}
\newtheorem{remark}[lemma]{Remark}
\newtheorem{example}[lemma]{Example}

\newtheorem{classification}[lemma]{Classification}

\newtheorem{definition}[lemma]{Definition}

\theoremstyle{remark}

\newcommand{\exnumber}[1]{$\langle$#1$\rangle$}

\title{On finitely graded Iwanaga-Gorenstein algebras and the stable categories 
of their (graded) Cohen-Macaulay modules  
{\small to the memory of  R.O. Buchweitz}}

\author{Hiroyuki Minamoto and Kota Yamaura}

\def\sfqv{{\mathsf{qv}}}

\begin{document}

\maketitle

\begin{abstract}
We discuss  finitely  graded Iwanaga-Gorenstein (IG) algebras $A$ 
and representation theory of their (graded) Cohen-Macaulay (CM) modules. 
By quasi-Veronese algebra construction, 
in principle, we may reduce our study to the case where 
$A$ is a trivial extension algebra $A = \Lambda \oplus C$ with the grading $ \deg \Lambda = 0, \ \deg C = 1$. 
In \cite{adasore} 
we gave a necessary and sufficient condition that $A$ is IG in terms of $\Lambda$ and $C$ 
by using derived tensor products and derived Homs.  
For simplicity, we assume that $\Lambda$ is of finite global dimension in the sequel. 
In this paper, we show that  the condition that $A$ is IG, 
has a triangulated categorical interpretation. 
We prove  that if $A$ is IG, 
then the graded stable category $\stabgrCM A$ of CM-modules is realized as 
an admissible subcategory  of the derived category $\sfD^{\mrb}(\mod \Lambda)$. 
As a corollary, we deduce that the Grothendieck group $K_{0}(\stabgrCM A)$ 
is free of finite rank. 
%We show that the stable category $\stabCM A$ of (non-graded) CM-modules is realized as the  orbit category of the derived category $\sfD^{\mrb}(\mod \Lambda)$ with respect to a certain autoequivalence. 

We give several applications.  
Among other things, for a path algebra $\Lambda= \kk Q$ of an $A_{2}$ or $A_{3}$ quiver Q, 
we give a complete list of $\Lambda$-$\Lambda$-bimodule $C$  such that $\Lambda \oplus C$ is IG (resp. 
of finite global dimension)  
by using the triangulated categorical interpretation mentioned above. 
\end{abstract}

\tableofcontents

\section{Introduction}\label{Introduction}

Representation theory of (graded) Iwanaga-Gorenstein (IG) 
algebra was initiated by Auslander-Reiten \cite{AR}, Happel \cite{Happel} 
 and Buchweitz \cite{Buchweitz}, 
has been studied by many researchers
 and is recently getting interest from other areas.  
One source of interest on representation theory IG-algebra is theory of cluster categories,   
since higher cluster category is often realized as the stable category $\stabCM A$ of ungraded CM-modules over a finitely  graded algebra $A$.  
Another source is in the study of Mirror symmetry, as is mentioned below.

 Recall that a graded algebra $A= \bigoplus_{i= 0}^{\infty} A_{i}$ is called IG if it is Noetherian on both sides and 
 of finite self-injective dimension on both sides. 
 The central object of representation theory of graded IG-algebra $A$ is 
 the category $\grCM A$ of graded Cohen-Macaulay (CM) modules and 
 its stable categories $\stabgrCM A$. 
 The latter has a canonical structure of a triangulated category. 
 Fundamental results are following  equivalences of triangulated categories 
 \begin{equation}\label{equivalences}
 \sfK^{\mrac}(\grproj A) \xrightarrow{\ \ \underline{\tuZ}^{0} \ \ } \stabgrCM A \xrightarrow{\ \ \beta \ \ } \grSing A 
 \end{equation}
where 
 $\sfK^{\mrac}(\grproj A)$ is the homotopy category of acyclic complexes of graded projective $A$-modules 
 and $\grSing A$ is the graded singularity category. 
 It is defined by the Verdier quotient $\grSing A := \sfD^{\mrb}(\grmod A)/\sfK^{\mrb}(\grproj A)$. 
 If we perform the same construction to an algebraic variety $X$, 
then we obtain a triangulated category $\Sing X$ which only depends  the singular locus of $X$. 
Hence, the name. Singularity categories play   important roles in theory of Mirror symmetry.
 This is another reason that  representation theory of IG-algebra have been becoming to get much attention.  
Story until now is the same with the ungraded situation. 
There is a special feature for the graded situation, the Orlov subcategory $\sfO$.  
Under the assumption that the degree $0$-subalgebra $A_{0}$ coincides with 
the base field $\kk$, 
Orlov \cite{Orlov} found another triangulated category $\sfO$ which is equivalent to 
$\stabgrCM A$ in study of Mirror symmetry.  
The category $\sfO$ is 
a triangulated subcategory $\sfO$ of $\sfD^{\mrb}(\grmod A)$ 
such that  the restriction functor $\pi|_{\sfO}$ of 
the canonical quotient functor $\pi: \sfD^{\mrb}(\grmod A) \to \grSing A$ gives an equivalence. 
\begin{equation}\label{Introduction: Orlov equivalence}
\pi|_{\sfO} : \sfO \xrightarrow{ \ \ \sim \ \  }  \grSing A.
\end{equation}
We postpone giving  the definition of $\sfO$ under more general assumption 
that $A_{0}$ is IG until Section \ref{Orlov's equivalence}.

The aim of this paper is to study finitely graded IG-algebras $A=\bigoplus_{i=0}^{\ell} A_{i}$ 
over a field and the stable categories of their graded Cohen-Macaulay (CM) modules 
by using results of \cite{adasore}.  
One of our achievement is the following result.

\begin{theorem}\label{Introduction: theorem 1} 
Let $A= \bigoplus_{i= 0}^{\ell} A_{i}$ be a finite dimensional  graded IG -algebra. 
If the degree $0$-part algebra $A_{0}$ is of finite global dimension, 
then  the Gorthendieck group $K_{0}(\stabgrCM A)$ is free and 
its rank is bounded by $\ell |A_{0}|$ from above: 
\[
\rank K_{0}(\stabgrCM A) \leq \ell |A_{0}| 
\]
where $|A_{0}|$ denotes the number of non-isomorphic simple $A_{0}$-modules. 
\end{theorem}

One important trick here is the \emph{quasi-Veronese} algebra construction: 
from a finitely graded algebra $A = \bigoplus_{i=0}^{\ell} A_{i}$, 
we can construct the $\ell$-th quasi-Veronese algebra $A^{[\ell]}$, 
whose important properties are  that it is graded Morita equivalent to $A$ and 
that $(A^{[\ell]})_{i} = 0$ for $i \neq 0,1$.  
We denote $\nabla A := (A^{[\ell]})_{0}, \Delta A := (A^{[\ell]})_{1}$ 
and call $\nabla A$ the \emph{Beilinson} algebra. 
\[
A^{[\ell]} = \nabla A \oplus \Delta A. 
\]
Since they are graded Morita equivalent, 
$A$ is IG if and only if so is $A^{[\ell]}$ and if this is the case, 
we have an equivalence $\stabgrCM A \simeq \stabgrCM A^{[\ell]}$. 
Therefore, we may and will concentrate in the case  $A= A_{0} \oplus A_{1}$. 
In this case $A$ is regarded as the trivial extension algebra of $A_{0}$ by the bimodule $A_{1}$ over it. 
Hence,  the basic set up of this paper is the followings: 
$\Lambda$ is an algebra, 
$C$ is a bimodule over it and $A = \Lambda \oplus C$ is the trivial extension 
with the canonical grading $\deg \Lambda = 0, \deg C= 1$.

%\color{blue}
An important class of trivial extension algebras which are IG is the class of cluster tilted algebras, which play a key role in cluster theory.  
Originally a cluster tilted algebra $A$ is defined to be the endomorphism algebra $A = \End_{\cC} T$ of a cluster tilting object $T$ of 
a cluster category $\cC$. 
It is shown in \cite{ABS} that $A$ is obtained as the trivial extension algebra $A = \Lambda \oplus \Ext_{\Lambda}^{2}(\tuD(\Lambda), \Lambda)$ 
where $\Lambda$ is an iterated tilted algebra. 
It is proved in \cite{Keller-Reiten} that $A$ is IG. 
We leave for  future work to apply our results to cluster tilted algebras.
%Thus we expect our general result can be applied to study cluster tilted algebras. 
%\color{black}

Let $\Lambda$ be a finite dimensional algebra again. 
In \cite{adasore}, we call a bimodule $C$ over $\Lambda$ \emph{asid} 
(as an abbreviation of ``attaching self-injective dimension")  
if the trivial extension algebra $A = \Lambda \oplus C$ is IG. 
We gave  a characterization of asid bimodule in terms of $- \lotimes_{\Lambda} C^{a}$ and $\RHom_{\Lambda}(C^{a},-)$ 
%as a consequence of injective dimension formula of trivial extensions algebras 
where $C^{a}$ denotes the iterated derived tensor product of $C$, namely, for a natural number  $a > 0$,  
\[
C^{a} := C \lotimes_{\Lambda} C \lotimes_{\Lambda} \cdots \lotimes_{\Lambda} C \ \ \ ( a\textup{-times}) 
\]
and $C^{0} := \Lambda$. 
One of our main theorem gives a categorical characterization of asid bimodules. 

\begin{theorem}\label{Introduction:Adaching theorem 0}
Let $\Lambda$ be a finite dimensional algebra of finite global dimension and $C$ a bimodule over $\Lambda$. 
Then 
the trivial extension algebra $A = \Lambda   \oplus C$ is IG 
if and only if there exists an admissible subcategory 
$\sfT \subset \sfD^{\mrb}(\mod \Lambda)$ 
which satisfies the following conditions. 
\begin{enumerate}[(1)]
\item The functor $ \cT = -\lotimes_{\Lambda} C$ acts on $\sfT$ as an equivalence, 
 i.e., 
$\cT({\sfT}) \subset {\sfT}$ and 
the restriction functor $\cT|_{\sfT}$ is an autoequivalence.

\item The functor $\cT = -\lotimes_{\Lambda} C$ nilpotently acts on $\sfT^{\perp}$, i.e., 
$\cT({\sfT}^{\perp}) \subset {\sfT}^{\perp}$ and 
there exists a natural number $a \in \NN$ such that $\cT^{a}({\sfT}^{\perp}) =0$. 
\end{enumerate} 
\end{theorem}

It might be looked that  the categorical characterization is only an abstract result. 
However, this can be a powerful tool to attack concrete problems. 
For instance, in Section \ref{Classification},    
we apply this characterization to solve the problem that classifies asid bimodules over a path algebra of $A_{2}$ or $A_{3}$ quiver. 

This theorem is proved in Theorem \ref{Adaching theorem 0} 
for the more general case where $\Lambda$ is IG with a little modification 
that replace $\sfD^{\mrb}(\mod \Lambda)$ with $\sfK^{\mrb}(\proj \Lambda)$. 
We note that in Theorem \ref{Adaching theorem 0} we give  another categorical characterization 
in terms of $C$-duality functors 
\[
(-)^{\star}:= \RHom_{\Lambda} ( -, C): \sfD^{\mrb}(\mod \Lambda) \rightleftarrows \sfD^{\mrb}(\mod \Lambda^{\op})^{\op} 
:  \RHom_{\Lambda^{\op}}(-,C) = : (-)^{\star}. 
\]

We call an admissible subcategory $\sfT \subset \sfD^{\mrb}(\mod \Lambda)$ 
an \emph{asid subcategory} of an asid module $C$  if it satisfies the conditions (1) and (2) of above theorem. 
The next main result,  Theorem \ref{Introduction: Adaching theorem 1}, 
 asserts a uniqueness of asid subcategory by giving  a description using an asid bimodule. 
For this, we use  two invariants of an asid bimodule, the \emph{right asid number} $\alpha_{r}$ and the \emph{left asid number} $\alpha_{\ell}$ 
which are introduced in \cite{adasore}. 
%These numbers play important role in this paper as shown in Theorem \ref{Introduction: Adaching theorem 1} below. 
The original definitions is recalled in Definition \ref{adasore:asid number definition}. 
Here we give a formula for the right asid number $\alpha_{r}$
under the assumption that $\Lambda$ is finite dimensional. 
%The formula below tells that  $\alpha_{r}$ is the amplitude of the degree of the socles  of the graded cosyzygies $\Omega^{-n}A$ as graded $A$-modules. 
%\begin{equation}\label{asid number formula}
%\alpha_{r}  = \max\{ a\in \ZZ \mid \exists n \textup{ s.t. } \soc (\Omega^{-n}A)_{a} \neq 0\} -
%\min\{ a\in  \ZZ \mid \exists n \textup{ s.t. } \soc (\Omega^{-n}A)_{a} \neq 0\}. 
%\end{equation}
%\[
%\alpha_{r}  = \max\{ a\geq -1 \mid \exists n \textup{ s.t } \soc (\Omega^{-n}A)_{-a} \neq 0\} + 1. 
%\]
\[
\alpha_{r}  = 1 - \min\{ a\in \ZZ \mid \exists n \textup{ s.t. } \soc (\Omega^{-n}A)_{a} \neq 0\}. 
\]
This formula  tells us that
it essentially counts the minimal degree of the socle of the graded cosyzygies $\Omega^{-n}A$ as graded $A$-modules. 
The left asid number $\alpha_{\ell}$ is given by the same formula involving the left graded cosyzygies. 
Thus the first statement of Theorem \ref{Introduction: Adaching theorem 1} below 
can be regarded as a result concerning on right-left symmetry of graded self-injective resolution of $A$.  

\begin{theorem}[Theorem \ref{Adaching theorem 1}]\label{Introduction: Adaching theorem 1}
Assume that $\gldim \Lambda < \infty$. 
If $A= \Lambda \oplus C$ is IG, then the following assertions hold.   
\begin{enumerate}[(1)]
\item   
We have $\alpha_{r} = \alpha_{\ell}$. We put $\alpha := \alpha_{r} = \alpha_{\ell}$. 

\item 
The subcategory $\sfT$ of (2)  of Theorem \ref{Introduction:Adaching theorem 0} is uniquely determined as 
$\sfT = \thick C^{\alpha}$. 

\item 
We have 
$
\sfT^{\perp} = \Ker (- \lotimes_{\Lambda} C^{\alpha}). 
$
\end{enumerate}
\end{theorem}

%\color{blue}
The next result reveals  the categories $\sfT$ and $\sfT^{\perp}$  are  strongly related to representation theory of IG-algebra $A$. 
 More precisely the asid subcategory $\sfT$ is shown to be equivalent to the stable category $\stabgrCM A$ via generalized Happel's functor $\cH$ and the right perpendicular category $\sfT^{\perp}$ is identified with $\Ker \cH$. 
%\color{black}

To state this,  we recall original  Happel's functor.  
 Let $\sfD(\Lambda) := \Hom_{\kk}(\Lambda, \kk)$ be the $\kk$-dual of $\Lambda$ with the canonical bimodule structure. 
 Then the trivial extension algebra  $T(\Lambda) : =\Lambda \oplus \tuD(\Lambda)$ 
is self-injective  and $\stabgrCM T(\Lambda) = \stabgrmod T(\Lambda)$. 
Happel \cite{Happel book} showed that 
there is a canonical embedding functor 
\[
\cH : \sfD^{\mrb}(\mod \Lambda) \hookrightarrow \stabgrmod A. 
\]
He also showed that $\cH$ gives an equivalence if and only if $\Lambda$ is of finite global dimension. 

%The equivalence tells  existence of tilting object in the stable category $\stabgrmod T(\Lambda)$. 
%Later it is observed in \cite{Chen trivial} \cite{MM} that a finite dimensional graded self-injective  algebra $A$   
%Finally, the second author showed in \cite{Yamaura} that 

Observe that even in the case $C \neq \sfD(\Lambda)$, 
if $A = \Lambda \oplus C$ is IG, 
then there are an analogue of  Happel's functor 
constructed as the following composite functor 
\[
\cH : \sfD^{\mrb}(\mod \Lambda) \hookrightarrow \sfD^{\mrb}(\grmod A) \twoheadrightarrow 
\grSing A \xrightarrow{\beta^{-1}} \stabgrCM A 
\]
where the first two arrows are canonical functors and the third is the inverse of the functor $\beta$ in \eqref{equivalences}.  
For a finitely graded IG-algebra $A = \bigoplus_{i = 0}^{\ell}A_{i}$, 
the quasi-Veronese algebra construction yields the functor below which is also denoted by $\cH$.
\[
\cH : \sfD^{\mrb}(\mod \nabla A) \to \stabgrCM A^{[\ell]} \xrightarrow{ \simeq } \stabgrCM A
\]
%for a general finitely graded IG-algebra $A=\bigoplus_{i=0}^{\ell}A_{i}$ 
%we have a functor $\cH: \sfD^{\mrb}(\mod \nabla A) \to \stabgrCM A$ via 
%quasi-Veronese algebra construction. 
%A class of algebras $A$ such that $\cH$ gives an equivalence is obtained by Lu \cite{Lu}. 
In general  $\cH$ does  neither  give an equivalence nor even is fully faithful.
In \cite{anodai} we characterize 
 a finitely graded IG-algebra $A = \bigoplus_{i= 0}^{\ell} A_{i}$ such that generalized Happel's functor $\cH$ is fully faithful or gives an equivalence. 
In this paper, we study general properties of $\cH$.

\begin{theorem}%[Theorem \ref{Adaching theorem 1}, Theorem \ref{Adaching theorem 2}]
\label{Introduction: Adaching theorem 2}
Assume that $\gldim \Lambda < \infty$ and  $A= \Lambda \oplus C$ is IG. 
Let $\sfT$ be the asid subcategory.   Then the following assertions hold.   
\begin{enumerate}[(1)]
\item 
The functor $\cH$ restricts to give an equivalence 
\[
\cH|_{\sfT} : \sfT \xrightarrow{\sim} \stabgrCM A.
\]

%\color{blue}
In particular $\stabgrCM A$ is realized as an admissible subcategory of $\sfD^{\mrb}(\mod \Lambda)$. 
%\color{black}

\item 
We have 
$\sfT^{\perp} = \Ker \cH $. 

%\color{blue}
\item 
The asid number $\alpha = \alpha_{r} = \alpha_{\ell}$ is a minimal integer 
that satisfies either 
$ \Ker( -\lotimes C^{\alpha})= \Ker \cH$ or 
$\cH|_{\thick C^{\alpha}}: \thick C^{\alpha} \xrightarrow{\simeq} \stabgrCM A$. 
%\color{black}
\end{enumerate}
\end{theorem}

The assertion 
(2) is proved in Theorem \ref{Adaching theorem 1} and 
(3) is a consequence of Lemma \ref{abstraction 1},  Proposition \ref{asid number proposition}. 

The assertion (1) is proved in Theorem \ref{Adaching theorem 2}. 
What we actually show is that the functor $\cH|_{\sfT}$ is an equivalence which 
fits into the following commutative diagram 
three arrows of which are equivalences of  
\eqref{equivalences} and \eqref{Introduction: Orlov equivalence}. 
\[
\begin{xymatrix}{
\sfK^{\mrac} (\grproj A) \ar[rr]^{\qquad \underline{\tuZ}^{0}}_{\sim} \ar[d]_{\frkp_{0} }^{\wr} & & 
\stabCM A \ar[d]^{ \beta}_{\wr} \\
\sfT  \ar[r]^{\sim}_{\mathsf{in}|_{\sfT}} \ar[rru]^{\cH|_{\sfT}}_{\sim} & \sfO \ar[r]^{\sim \ \ \ \ \ }_{\pi|_{\sfO}\ \ \ } & \grSing A 
}
\end{xymatrix}
\]
The functor  
$\mathsf{in}|_{\sfT}$ is the restriction of 
 the canonical embedding 
$\mathsf{in}: \sfD^{\mrb}(\mod \Lambda) \hookrightarrow \sfD^{\mrb}(\grmod A)$ 
and 
$\frkp_{0}$ is  a functor which assign a projective complex $P$ with its ``degree $0$-generators". 
These equivalence can be proved in the more general case  where $\Lambda$ is IG. 
For this we need to introduce  ``locally perfect" complexes of graded modules.  

%\color{blue}
Recall that a graded module $M = \bigoplus_{i \in \ZZ} M_{i}$ is called locally finite if $\dim M_{i} < \infty$ for all $i \in \ZZ$. 
In the same  way the notion of  locally perfect complexes is defined in Definition \ref{locally perfect}. 
It seems that when we generalize Orlov's equivalence $\sfO \xrightarrow{\sim} \grSing A$ 
to a (not necessary finitely) graded algebra $A = \bigoplus_{ i \geq 0} A_{i}$ with $A_{0}$ IG of $\gldim A_{0}= \infty$, 
it is necessary to use the category of   locally perfect complexes.  
Moreover a similar notion of  ``locally free'' modules play a key role in \cite{GLS}.  
Thus we can expect that the notion of locally perfect complexes will be of use in study of IG-algebras.
%\color{black}

As a summary, we give the following theorem 
which can be stated for a general finitely graded IG-algebra $A$ which is not necessary  a trivial extension algebra.

\begin{theorem}\label{Introduction recollement theorem}
Let $A = \bigoplus_{i=0}^{\ell} A_{i}$ be a finite dimensional graded IG-algebra. 
Assume that $\gldim \Lambda < \infty$. 
Then there exists the recollement of the following form 
\[
\begin{xymatrix}{
\stabgrCM A \ar[rr] 
&& \sfD^{\mrb}(\mod \nabla A) \ar[rr] \ar@/^4mm/[ll]^{\cH} \ar@/_4mm/[ll]
&& \Ker \cH \ar@/^4mm/[ll]^{\mathsf{in}}  \ar@/_4mm/[ll]
}\end{xymatrix}
\]
where $\mathsf{in}$ is a canonical inclusion. 
\end{theorem}

We mention that Theorem \ref{Introduction: theorem 1} is a consequence of this theorem.

This is the end of abstract results. We give three concrete applications. 
In the first two application we discuss a finite dimensional graded IG-algebra of finite CM-type. 
To obtain  conclusions of ungraded CM-modules from the results  of graded CM-modules 
obtained before, we use a CM-version of Gabriel theorem by the first  author.
%\color{blue} and M. Yoshiwaki \cite{MYY} Uh------NNNNN 
% \color{black}
This theorem asserts that a finite dimensional graded IG-algebra $A$ 
is of graded finite CM-type if and only if it is of (ungraded) CM-type. 

It is known that 
a finite dimensional algebra  $\Lambda$ is a iterated tilted algebra 
if and only if  the trivial extension algebra $\tuT(\Lambda) = \Lambda \oplus \tuD(\Lambda)$ by $\tuD(\Lambda)$ is of finite representation type 
(see Happel \cite[Section V.2]{Happel book}).  
Since the algebra $\tuT(\Lambda)$ is always self-injective, 
our first application can be looked as a CM-generalization of one implication of this result.  
A CM-generalization  of the other implication is discussed in \cite{MYY}.

\begin{theorem}[Theorem \ref{application:iterated theorem}]\label{Introduction: application theorem 1}
Let $\Lambda$ be an iterated tilted algebra of Dynkin type. 
If a trivial extension algebra $A = \Lambda \oplus C$ is IG, 
then it is of  finite  CM-type.
\end{theorem}

In this theorem, CM representation theory of $A$ is controlled by the degree $0$-part 
$\Lambda$. 

An easiest way to obtain a bimodule is 
taking a tensor product  $C = N \otimes_{\kk} M$ of 
a left $\Lambda$-module $N$ and a right $\Lambda$-module $M$.  
In the second application, 
we study this case 
under  the assumption  $\gldim \Lambda < \infty$. 
In Theorem \ref{tensor bimodule case theorem} 
we determine the condition that $A = \Lambda \oplus C$ is IG (or $\gldim A < \infty$) 
and give a description of $\stabCM A$. 
We see that if $A$ is IG, then it is always of finite CM type and the number of indecomposable CM-modules is given by $\# \ind \stabCM A = \pd M + 1$. 
Contrary to the first application, 
CM-representation theory is controlled by the degree $1$-part $N \otimes_{\kk} M$. 

As the third and final application, using Theorem \ref{Introduction:Adaching theorem 0}
we give a complete list of $\Lambda$-$\Lambda$-bimodules $C$ 
such that $\Lambda \oplus C$ is IG 
in the case where $\Lambda= \kk Q$ is 
the path algebra of a quiver $Q$ of $A_{2}$ type or $A_{3}$ type. 

\bigskip

The organization of the paper is the following.
In Section \ref{Preliminaries} collects  results which are need in the sequel.  
Among other things, 
in Section \ref{Graded algebras}, we recall   the  quasi-Veronese algebra construction 
and the decomposition functor $\frkp_{i}$ which plays a key role in the paper.  
%In Section \ref{Derived tensor products} we discuss derived tensor products of bi-modules. 
In Section \ref{section lp}, we introduce the notion of locally perfect complexes and locally perfect CM-modules. 
%Although these notions are defined under the assumption that a graded algebra $A$ is finitely graded. 
They can be defined over arbitrary graded algebras. 
In Section \ref{Categorical characterization}, we prove categorical characterization of asid bimodules 
(Theorem \ref{Adaching theorem 0}) 
and a uniqueness of asid subcategories $\sfT$ of an asid bimodule $C$ (Theorem \ref{Adaching theorem 1}).  
In Section \ref{around}, 
we show that the asid subcategory $\sfT$ is equivalent to the stable category $\stablpCM A$ of locally perfect CM-modules. 
As a consequence, 
for a finitely graded algebra $A = \bigoplus_{i =0}^{\ell} A_{i}$ 
we see that $\stablpCM A$ is realized as an admissible subcategory of 
$\sfK^{\mrb}(\proj \nabla A)$. 
In Section \ref{Applications} we apply our result to study  two particular classes of trivial extension algebras. 
In Section \ref{Classification}, we give a complete list of asid bimodules $C$ 
in the case where $\Lambda= \kk Q$ is the path algebra of a quiver $Q$ of $A_{2}$ type or $A_{3}$ type.

\subsection{Notation and convention}

Throughout this paper the symbol $\kk$ denotes a field  and ``algebra" means $\kk$-algebra. 
For generalization to the case where the base commutative ring $\kk$ is not a field, 
see remark \ref{general base ring remark}. 
The symbol $\tuD$ denotes the $\kk$-dual functor $\tuD := \Hom_{\kk}(-,\kk)$. 

Let $\Lambda$ be an algebra. 
Unless otherwise stated, the word  ``$\Lambda$-modules" means a right $\Lambda$-modules.  
%We denote by $\mod \Lambda$ the category of  finite dimensional (right) $\Lambda$-modules. 
We denote the opposite algebra  by $\Lambda^{\op}$. 
We identify left $\Lambda$-modules with (right) $\Lambda^{\op}$-modules.   
A $\Lambda$-$\Lambda$-bimodule $D$ is always assumed to be $\kk$-central, 
i.e., $ad = da $ for $d \in D, \ a \in \kk$, 
Therefore we may identify $\Lambda$-$\Lambda$-bimodules  
with modules over the enveloping algebra $\Lambda^{\mre}:= \Lambda^{\op} \otimes_{\kk} \Lambda$. 
For a $\Lambda$-$\Lambda$-bimodule $D$, 
we denote by $D_{\Lambda}$ and ${}_{\Lambda} D$ 
the underlying right  and left $\Lambda$-modules respectively. 
So for example, 
$\injdim D_{\Lambda}$ denotes the injective dimension of $D$ regarded a (right) $\Lambda$-module.

For an additive category $\cA$, we denote by $\sfC(\cA)$ and $\sfK(\cA)$  
the category of cochain complexes and cochain morphisms 
and its homotopy category respectively. 
For complexes $X, Y \in \sfC(\cA)$, 
we denote by $\Hom_{\cA}^{\bullet}(X, Y)$ the $\Hom$-complex. 
For an abelian category $\cA$, we denote by $\sfD(\cA)$ the derived category of $\cA$. 

%All triangulated categories and exact  functors between them are assumed to be $\kk$-linear. 
The shift of a triangulated categories are denoted by $[1]$.

In the paper except Section \ref{Classification}, 
the degree of graded modules $M$ is usually  indicated by the characters $i,j ,\dots$. 
The degree of complexes $X$  is usually indicated by the characters $m,n, \dots$.

\vspace{10pt}
\noindent
\textbf{Acknowledgment}
%
%The authors give hearty gratitude to Takahide Adachi 
%who very kindly allows us to use the term ``adaching module".
%
The authors thank O. Iyama for giving comments on earlier version of the results given this paper. 
The first author  was partially  supported by JSPS KAKENHI Grant Number 26610009.
The second author  was partially  supported by JSPS KAKENHI Grant Number 26800007.

\section{Preliminaries}\label{Preliminaries}

\subsection{Graded algebra and graded modules}\label{Graded algebras}

In this paper, a graded algebra is always a non-negatively graded algebra $A = \bigoplus_{i \geq 0} A_{i}$. 

Let $A$ be a graded algebra. 
We denote by $\GrMod A$ the category of graded (right) $A$-modules $M= \bigoplus_{i \in \ZZ} M_{i}$ 
and graded $A$-module homomorphisms $f: M \to N$, which, by definition, preserve degree of $M$ and $N$, 
i.e., $f(M_{i}) \subset N_{i}$. 
We define the truncation $M_{\geq  j}$ by 
$(M_{\geq  j})_{i} =M_{i} \ (i \geq j), \ \ (M_{\geq j})_{i} = 0  \ ( i < j)$. 
We set $M_{ < j} := M/M_{\geq j}$ so that we have an exact sequence $ 0\to M_{\geq j} \to M \to M_{< j} \to 0$. 

For a graded $A$-module $M$ and an integer $j \in \ZZ$, 
we define the shift $M(j) \in \GrMod A$ by $(M(j))_{i} = M_{i+j}$. 
For $M, N \in \GrMod A, \ n \in \NN$ and $i \in \ZZ$,  
we set $\grExt_{A}^{n}(M,N)_{i}:= \Ext_{\GrMod A}^{n}(M,N(i))$ 
and 
\[
\grExt_{A}^{n}(M, N) := \bigoplus_{i \in \ZZ} \grExt_{A}^{n}(M, N)_{i} = \bigoplus_{i \in \ZZ} \Ext_{\GrMod A}^{n}(M, N(i)). 
\]
We note the obvious equality $\grHom_{A}(M,N)_{0} = \Hom_{\GrMod A}(M,N)$. 
We use the similar notation for $\grRHom_{A}(M,N)$ where $M, N$ are objects of $\sfD(\GrMod A)$.

We denote by $\grmod A\subset \GrMod A$ the full subcategory of finitely generated graded $A$-modules.  
For $i \in \ZZ$, we denote by $\Mod^{\geq i} A \subset \GrMod A$ 
the full subcategory consisting of $M \in \GrMod A$ such that $M_{< i} = 0$. 
We set $\mod^{\geq i}A := (\grmod A) \cap (\Mod^{\geq i} A)$. 
 Similarly, we define the full subcategories $\Mod^{< i} A$ and $\mod^{< i} A$.

\subsubsection{Quasi-Veronese algebras}\label{quasi-Veronese} 

A (non-negatively) graded algebra $A = \bigoplus_{i \geq 0} A_{i}$ is called \textit{finitely graded} if $A_{i} = 0$ for $i \gg 0$. 

%In this Section \ref{quasi-Veronese} 
%which  shows that every finitely graded algebra $A$ is graded Morita equivalent to a trivial extension algebra 
%$\Lambda\oplus C$ with the canonical grading $\deg \Lambda = 0, \deg C = 1$. 
%Thus, any representation theoretic problem of a 
%finitely graded algebra is, in principle, reduced to that of a trivial extension algebra.  

Let $A$ be a finitely graded algebra. 
We fix a natural number $\ell$ such that $A_{i} = 0$ for $ i \geq \ell +1$. 
(It is not necessary to assume that $A_{\ell} \neq 0$.) 
%We recall the quasi-Veronese algebra construction introduced by I. Mori \cite{Mori B-construction},    
We recall that  the \textit{Beilinson algebra} $\nabla A$ of $A$ 
(which  rigorously should be  called the Beilinson algebra of the pair $(A, \ell)$) 
and its bimodule $\Delta A$ are defined to be 
\[
\nabla A: = 
\begin{pmatrix} 
A_{0} & A_{1} & \cdots & A_{\ell -1} \\
0         & A_{0} & \cdots  & A_{\ell -2} \\
\vdots & \vdots     &        & \vdots \\
0         & 0   & \cdots  & A_{0}
\end{pmatrix}, \ \ \ 
\Delta A: = 
\begin{pmatrix} 
A_{\ell} & 0 & \cdots & 0 \\
A_{\ell- 1} & A_{\ell} & \cdots  &0 \\
\vdots & \vdots     &        & \vdots \\
A_{1} & A_{2}   & \cdots  & A_{\ell}
\end{pmatrix} 
\]
where the algebra structure and the bimodule structure are 
given by the  matrix  multiplications. 
Then, 
the trivial extension algebra 
 $\nabla A \oplus \Delta A$ 
with the grading $\deg \nabla A = 0, \deg  \Delta A = 1$ 
is   the $\ell$-th quasi-Veronese algebra $A^{[\ell]}$ of $A$ 
introduced by Mori \cite[Definition 3.10]{Mori B-construction}. 
\[
A^{[\ell]} = \nabla A \oplus \Delta A. 
\]

By \cite[Lemma 3.12]{Mori B-construction} 
$A$ and $A^{[\ell]} $ are graded Morita equivalent to each other. 
More precisely, 
the functor $\sfqv :=\sfqv_{A}$ below gives a $\kk$-linear equivalence. 
\[
\begin{split}
 & \sfqv: \GrMod A \xrightarrow{ \ \simeq  \ } \GrMod A^{ [ \ell ]}, \\
& \sfqv(M) := \bigoplus_{i \in \ZZ} \sfqv(M)_{i}, \ \ \ \ 
\sfqv(M)_{i} =  M_{i\ell} \oplus M_{i \ell +1 } \oplus \cdots \oplus M_{( i+ 1)\ell -1 }.
\end{split}
\]
This equivalence is restricted to an equivalence  $\Mod^{\geq 0} A \xrightarrow{\sim} \Mod^{\geq 0}A^{[\ell]}$. 

The  functor $\sfqv$ has the following compatibility with the degree shift functors over $A$ and $A^{[\ell]}$. 
\[
(1) \circ \sfqv \cong \sfqv \circ (\ell).
\]

For simplicity we set $B = A^{[\ell]}$. 
We may identify $(A^{\op})^{[\ell]}$ with $B^{\op}$. 
The composite functor  $\sfqv'$ gives an equivalence.  
\[\sfqv' := \sfqv_{A^{\op}} \circ ( -\ell +1): \GrMod A^{\op} \to \GrMod B^{\op}.\]
It induces an equivalence $\Mod^{>0} A^{\op} \xrightarrow{\sim} \Mod^{>0} B^{\op}$ 
and fits into the following commutative diagram 
\[
\begin{xymatrix}{
\GrMod A^{\op} \ar[rr]^{\grHom_{A^{\op}}(-, A) } \ar[d]_{\sfqv'} && 
\GrMod A \ar[d]^{\sfqv} \\
\GrMod B^{\op} \ar[rr]_{\grHom_{B^{\op}}(-, B) }  && 
\GrMod B
}\end{xymatrix}
\]

\subsubsection{Decomposition of a complex of graded projective $A$-modules}\label{dgpm}

Let $A = \bigoplus_{i=0}^{\ell} A_{i}$ be a finitely graded algebra. 
We recall from \cite{adasore}  a decomposition of a complex of  graded projective $A$-modules. 
For notational simplicity, we set $\Lambda := A_{0}$.

First we deal with a graded projective $A$-module. 
For an integer $i \in \ZZ$, 
we denote by  $\frkp_{i}: \GrProj A \to \Proj \Lambda$ 
the functor $\frkp_{i}P := (P \otimes_{A} \Lambda)_{i}$. 
Then  $\frkt_{i} P :=  (\frkp_{i} P) \otimes_{\Lambda} A ( -i)$ is a graded projective $A$-module 
and there exists an isomorphism of graded $A$-modules, from which we deduce the following lemma.  
%\begin{equation}%\label{decomposition of graded projective modules}
%\begin{split}
\[
 P   \cong \bigoplus_{ i\in \ZZ} \frkt_{i} P. 
 \]
 % P_{i} & \cong (\frkt_{i-1} P)_{i}  \oplus  (\frkt_{i} P)_{i} \cong (\frkp_{i-1} P \otimes_{\Lambda} C) \oplus \frkp_{i}P 
%\ ( \textup{as $\Lambda$-modules}). 
%\end{split} 
%\end{equation}

%The following observation will be used later. We leave the proof to the readers. 

\begin{lemma}\label{fg lemma}
A graded projective $A$-module  $P \in \GrProj A$ is finitely generated 
if and only if 
$\frkp_{i} P$ is finitely generated and $\frkp_{i}P = 0$ for $|i| \gg 0$. 
\end{lemma}

We use the same symbol $\frkp_{i}, \frkt_{i}$ for graded $A^{\op}$-modules. 
If $P\in \GrProj A$ is such that  $\frkp_{i} P $ is finitely generated for $i \in \ZZ$, 
 then $\grHom_{A}(P, A) $ is a graded projective $A^{\op}$-module 
 and moreover we have 
an isomorphism of $\Lambda$-modules. 
\[
\frkp_{i}\grHom_{A}(P,A) \cong \Hom_{\Lambda}(\frkp_{-i}P, \Lambda).
\]
 
 We discuss compatibility of $\frkp_{i}$ with the quasi-Veronese algebra construction. 
For $r = 0, \cdots, \ell$, 
we define a projective $\nabla A$-module $R_{r}$ to be 
\[
R_{r} = (0,\cdots , 0, A_{0}, A_{1}, \cdots, A_{\ell -1 -r}).  
\]
It is clear that $\nabla A \cong \bigoplus_{r= 0}^{\ell -1} R_{r}$ as $\nabla A$-modules. 
We note that $R_{r}$ has a canonical $\Lambda^{\op}$-module structure. 
%It is easy to check that for $r = 0, \cdots, \ell -1$ 
%\[\frkp_{i}( \sfqv A(-(j\ell +r))) = \begin{cases} R_{r} & i =j \\ 0 &  i \neq j. \end{cases}\]
%Using this we deduce the following lemma. 

We leave the verification of the following lemma to the readers. 

\begin{lemma}\label{quasi-Veronese lemma} 
For $P \in \GrProj A$,  we have the following isomorphism of $\nabla A$-modules. 
\[
\frkp_{i} \sfqv P \cong \bigoplus_{r = 0}^{\ell -1} (\frkp_{i\ell +r} P) \otimes_{\Lambda} R_{r}.
\]
\end{lemma}

%As a corollary, we have 
%\begin{corollary} 
%The functor $\sfqv$ induces  equivalences
%\[\proj^{< 0} A \xrightarrow{\sim} \proj^{< 0} A^{[\ell]}, \   
%\proj^{\geq 0} A \xrightarrow{\sim} \proj^{\geq 0} A^{[\ell]}. \ \]
%\end{corollary}

 From now we deal with a complex of graded projective $A$-modules. 
By abuse of notations, 
we denote the functors $\frkp_{i}: \sfC(\GrProj A) \to \sfC(\Proj \Lambda)$, 
$\frkp_{i}: \sfK(\GrProj A) \to \sfK(\Proj \Lambda)$ 
induced from the functor $\frkp_{i}: \GrProj A \to \Proj \Lambda$ by the same symbols. 
The symbol $\frkt_{i}$ is used in the same way. 
The following lemma plays an important role in the sequel. 

\begin{lemma}[{\cite[Lemma 4.2]{adasore}}]\label{adasore lemma 4.2}
Assume that $A = \Lambda \oplus C$. 
Let $P \in \sfC(\GrProj A)$. 
Then there exists 
a morphism $\sfq_{i}: \frkp_{i} P \to (\frkp_{i -1}P ) \otimes_{\Lambda} C$  in $\sfC(\Mod \Lambda)$,  
which gives an exact triangle in $\sfK(\Mod \Lambda)$. 
\begin{equation}\label{kurumimochi}
\frkp_{i-1} P \otimes_{\Lambda} C \to P_{i} \to \frkp_{i} P \xrightarrow{\sfq_{i}} \frkp_{i-1}P \otimes_{\Lambda} C[1]. 
\end{equation}
In particular $\tuH(P)_{i} = 0$ if and only if the morphism $\sfq_{i}$ is an isomorphism in $\sfD(\Mod \Lambda)$. 
\end{lemma}

Assume that $\frkp_{i}P$ belongs to $\sfC^{-}(\Proj  \Lambda)$ for some $i \in \ZZ$. 
Then the complex $\frkp_{i} P \otimes_{\Lambda} C$ is quasi-isomorphic to 
 the derived tensor product $\frkp_{i} P \lotimes_{\Lambda} C$.
This observation yields the following corollary. 

\begin{corollary}\label{black thunder}
Assume that $A = \Lambda \oplus C$. 
Let $P \in \sfC^{-}(\Proj A)$ such that $\frkp_{i} P \in \sfC^{-}(\Proj \Lambda)$ for $i \in \ZZ$. 
Then the following assertions hold. 
\begin{enumerate}[(1)]
 \item 
If there exists $j \in \ZZ$ such that $\tuH(P)_{> j} = 0$,  
then  for $i \geq j$, we have an isomorphism 
$\frkp_{i} P \cong \frkp_{j} P \lotimes_{\Lambda} C^{i -j}[i-j]$ 
 in $\sfD(\Mod \Lambda)$. 
 
 \item 
If there exists $j \in \ZZ$ such that $\tuH(P)_{\leq  j} = 0$,  
then  for $i \leq j$, we have an isomorphism 
$\frkp_{j} P \cong \frkp_{i} P \lotimes_{\Lambda} C^{j-i}[j-i]$ 
 in $\sfD(\Mod \Lambda)$.
 \end{enumerate}
\end{corollary}

%\begin{remark}\label{graded Hom remark}
%Let $\Lambda$ be an ungraded algebra. 
%We may regard $\Lambda$ as a graded algebra concentrated at $0$-th degree.   
%If we do so, the notation $\Hom_{\Lambda}$ has two meaning that graded $\Hom$-space defined above and the $\Hom$-space of the category $\Mod \Lambda$ of ungraded modules. 
%In this paper, the notation $\Hom_{\Lambda}$ is always understood in the latter sense. Namely, $\Hom_{\Lambda} = \Hom_{\Mod \Lambda}$. 
%However, the two  exceptions occur 
%in the description of $\frks_{i}$ in Section \ref{Decomposition of graded injective $A$-modules} 
%and in the proof of  Lemma \ref{A key observation for the criterion}.  \end{remark}

\subsection{(Graded) Iwanaga-Gorenstein algebras}

In this section, we collect  basic facts about Iwanaga-Gorenstein (IG) algebra. 
In particular, we recall the constructions of equivalences in \eqref{equivalences}, 
since we use them in this paper. 
First we recall the definition of IG-algebras.
 
\begin{definition}
A (graded) algebra $A$ is said to be Iwanaga-Gorenstein(IG) 
if it is (graded) Noetherian on both sides and 
$(\textup{gr-})\injdim A_{A} < \infty, (\textup{gr-})\injdim {}_{A}A < \infty$. 
\end{definition}

We recall the fundamental observation  due to  Iwanaga, which is frequently and tacitly used in the sequel.

\begin{proposition}[Iwanaga \cite{Iwanaga}]\label{Iwanaga}
Let $A$ be a (graded) IG-algebra. 
For a (graded) $A$-module $M$ the following conditions are equivalent
\begin{enumerate}[(1)]
\item 
$(\textup{gr-})\pd M < \infty$. 

\item 
$(\textup{gr-})\injdim M < \infty$. 
\end{enumerate}
\end{proposition}

Next we recall the definition of Cohen-Macaulay modules. 

\begin{definition}
Let $A$ be a (graded) IG-algebra. 
A (graded) $A$-module $M$ is called \textit{Cohen-Macaulay} (CM) 
if it is finitely generated and satisfies the condition 
\[
\Ext_{A}^{n} (M, A) = 0 \textup{ for } n \geq 1.
\]
\end{definition}

We denote by $\grCM A \subset \grmod A$ the full subcategory of graded CM-modules. 
The ungraded version is denoted by $\CM A \subset \mod A$. 

Let $A$ be a graded IG-algebra. 
The category $\grCM A$ is a Frobenius category, 
whose admissible projective-injective modules are precisely graded projective modules. 
Therefore, by \cite{Happel book}, the stable category $\stabgrCM A:= \grCM A/[\grproj A]$ canonically has a structure of triangulated category.

Now we recall fundamental equivalences of triangulated categories 
which relate 
the stable category $\stabgrCM A$ to 
other important triangulated categories.

The graded singularity category $\grSing A$ is defined to be the Verdier quotient 
\[\grSing A := \sfD^{\mrb}(\grmod A)/ \sfK^{\mrb}(\grproj A).\] 
We consider the following diagram
\[
\begin{xymatrix}{ 
\grCM A \  \ar[d] \ar@{^{(}->}[r]^{i} & \ \grmod A \  \ar@{^{(}->}[r]^{j \ } & \ \sfD^{\mrb}(\grmod A) \ar[d]^{\pi}\\ 
\stabgrCM A  \ \ar@{..>}[rr]_{\beta} &&\  \grSing A
}\end{xymatrix}
\]
where the top arrows are canonical embeddings and the vertical arrows are canonical functors. 
Since $\pi j i (\grproj A) = 0$, 
there exists a unique functor $\beta$ 
which makes the above diagram commutative.

\begin{theorem}[Buchweitz \cite{Buchweitz}, Happel \cite{Happel}]\label{Buchweitz-Happel theorem}
The functor $\beta$ gives   an equivalence of triangulated categories. 
\end{theorem}

We denote by $\sfC^{\mrac}(\grproj A) \subset \sfC(\grproj A)$ the full subcategory of 
acyclic complexes of finitely generated graded projective modules. 
Let   $\tuZ^{0}: \sfC^{\mrac}(\grproj A) \to \grmod A$ 
be the functor which sends a complex $X$  to its $0$-th cocycle group $\tuZ^{0}(X)$. 
It  can be shown that the functor $\tuZ^{0}$ 
has  its image in $\grCM A$ 
and descents   to the  functor  $\underline{\tuZ}^{0} : \sfK^{\mrac}(\proj A) \to\stabgrCM A$. 
\[
\begin{xymatrix}{ 
\sfC^{\mrac}(\grproj A) \ar[d] \ar[rr]^{\tuZ^{0}}   && \grCM A \ar[d] \\ 
\sfK^{\mrac}(\grproj A) \ar[rr]_{\underline{\tuZ}^{0} } && \stabgrCM A 
}\end{xymatrix}\]

\begin{theorem}[{Buchweitz \cite{Buchweitz}}]\label{Buchweitz theorem}
The functor $\underline{\tuZ}^{0}$ gives   an equivalence  of triangulated categories. 
\end{theorem} 

%Now there exist the following equivalences of triangulated categories.
%\[\sfK^{\mrac}(\grproj A) \xrightarrow{\ \underline{\tuZ}^{0} \ \simeq \ } \stabgrCM A \xrightarrow{ \ \  \beta \  \simeq \ } \grSing A.\]

\subsection{Derived tensor products and derived Hom functors which involve bimodules}\label{Derived tensor products}

In this section \ref{Derived tensor products} 
we discuss the compatibility between 
the derived functors involving bimodules 
and the restriction functors.

Let $\Lambda$ be an algebra over a field $\kk$. 
Recall that we identify ($\kk$-central) $\Lambda$-$\Lambda$-bimodules 
with module over the enveloping algebra $\Lambda^{\mre}= \Lambda^{\op} \otimes_{\kk} \Lambda$ 
and regard the category $\Mod \Lambda^{\mre}$ as the category of $\Lambda$-$\Lambda$-bimodules. 

Let  $D$ be a $\Lambda$-$\Lambda$-bimodules. 
We denote by $D_{\Lambda}$ and ${}_{\Lambda}D$ 
the underlying right and left $\Lambda$-modules of $D$.  
The assignments $D \mapsto D_{\Lambda}$ and $D\mapsto {}_{\Lambda}D$ 
extend to functors, which are called the restriction functors. 
\[ \Mod \Lambda^{\mre} \to \Mod \Lambda, \ D \mapsto D_{\Lambda}, \ \ 
\Mod \Lambda^{\mre} \to \Mod \Lambda^{\op}, \ D \mapsto {}_{\Lambda}D,  \\
\]

We equip the $\Hom$-space $\Hom_{\Lambda}(D_{\Lambda}, E_{\Lambda})$  
and the tensor product $(D_{\Lambda} ) \otimes_{\Lambda} ({}_{\Lambda} E)$ 
with   a canonical  $\Lambda$-$\Lambda$-bimodule structures 
and  denote them 
 by $\Hom_{\Lambda}(D, E)$ and $D \otimes_{\Lambda} E$ respectively 
the bimodules so obtained.  
A $\Lambda$-$\Lambda$-bimodule $D \in \Mod \Lambda^{\mre}$ induces (covariant or contravariant) functors 
\[
\begin{split}
&D \otimes_{\Lambda} -, \ - \otimes_{\Lambda} D, \   \Hom_{\Lambda}(D,-), \  
\Hom_{\Lambda}(-, D)  : \Mod \Lambda^{\mre} \to \Mod \Lambda^{\op},  \\
&- \otimes_{\Lambda} D, \ \Hom_{\Lambda}(D, -) : \Mod \Lambda \to \Mod \Lambda, \ \ 
D \otimes_{\Lambda} - : \Mod \Lambda^{\op} \to \Mod \Lambda^{\op}, \\
& \Hom_{\Lambda}(-, D) : \Mod \Lambda \to \Mod \Lambda^{\op}. \\
\end{split}
\]

The aim of this section is to prove the following lemma. 
Since the restriction functors are exact, we denote the derived functors of them by the same symbol. 

\begin{lemma}\label{compatibility lemma}
For objects  $D ,E \in \sfD(\Mod \Lambda^{\mre})$, 
there exist natural isomorphisms 
\[\begin{split}
& (D \lotimes_{\Lambda}E )_{\Lambda} \cong  (D_{\Lambda}) \lotimes_{\Lambda} E, \ 
{}_{\Lambda}(D \lotimes_{\Lambda}E ) \cong   D \lotimes_{\Lambda} ({}_{\Lambda} E), \\
&{}_{\Lambda} \RHom_{\Lambda}(D, E)  \cong \RHom_{\Lambda}(D_{\Lambda}, E),\  
\RHom_{\Lambda}(D, E)_{\Lambda}  \cong \RHom_{\Lambda}(D, E_{\Lambda}) . 
\end{split}
\]
\end{lemma}

For example, the first natural isomorphism says that 
the left diagram of \eqref{two commutative diagrams} 
is commutative up to a natural isomorphism.
\begin{equation}\label{two commutative diagrams}
\begin{xymatrix}{
\sfD(\Mod \Lambda^{\mre}) \ar[d]_{(-)_{\Lambda}} \ar[rr]^{ -\lotimes_{\Lambda} E} && 
\sfD(\Mod \Lambda^{\mre}) \ar[d]^{(-)_{\Lambda}} \\
\sfD(\Mod \Lambda) \ar[rr]^{ -\lotimes_{\Lambda} E} && 
\sfD(\Mod \Lambda),  
}\end{xymatrix} 
\quad 
\begin{xymatrix}{
\sfC(\Mod \Lambda^{\mre}) \ar[d]_{(-)_{\Lambda}} \ar[rr]^{ -\otimes_{\Lambda} E} && 
\sfC(\Mod \Lambda^{\mre}) \ar[d]^{(-)_{\Lambda}} \\
\sfC(\Mod \Lambda) \ar[rr]^{ -\otimes_{\Lambda} E} && 
\sfC(\Mod \Lambda).  
}\end{xymatrix} 
\end{equation}

The key is the following lemma for which we need to assum
the base ring $\kk$ is a field.

\begin{lemma}\label{sublemma of compatibility lemma}
The restriction functor $D \mapsto D_{\Lambda}$ 
sends projective (resp. injective) $\Lambda^{\mre}$-modules to projective (resp. injective) $\Lambda$-modules. 

Similar statements hold for the restriction functor $D \mapsto {}_{\Lambda} D$. 
\end{lemma}

\begin{proof}
In this proof, we make use of the assumption that we are dealing with an algebra over a field. 

Since $\Lambda$ is free over the base field $\kk$, 
the restriction $\Lambda^{\mre}_{\Lambda} = (\Lambda^{\op} \otimes_{\kk} \Lambda)_{\Lambda}$ is free $\Lambda$-module. 
Thus the assertion for projective modules holds. 

To prove the assertion for injective modules, 
first note that the restriction functor is the restriction along 
the algebra homomorphisms 
$\iota: \Lambda \to \Lambda^{\mre}= \Lambda^{\op} \otimes_{\kk} \Lambda$ 
defined by $\iota(a) :=1_{\Lambda^{\op}} \otimes a \  (a \in \Lambda)$. 
Namely, $\Lambda^{\mre}$ has a $\Lambda$-$\Lambda^{\mre}$-bimodule structure 
whose left module structure is induced from $\iota$ 
and 
there exists a natural isomorphism $D_{\Lambda} \cong \Hom_{\Lambda^{\mre}}(\Lambda^{\mre}, D)$ 
of $\Lambda$-modules. 
Since $\Lambda$ is free over the base field $\kk$, 
$\Lambda^{\mre}$ is flat as a left $\Lambda$-module 
when 
it is regarded as left $\Lambda$-module. 
Now it is easy to check that if $D$ is an injective $\Lambda^{\mre}$-module, 
then the $\Lambda$-module  $D_{\Lambda} = \Hom_{\Lambda^{\mre}}(\Lambda^{\mre}, D)$ 
is injective by using $\otimes$-$\Hom$-adjunction. 
\end{proof}

\begin{proof}[Proof of Lemma \ref{compatibility lemma}.]
For simplicity, we only give a proof of the first natural isomorphism. 
The others can be proved in similar ways. 

First  we recall the construction of left derived functors from \cite{Keller:ddc}. 
An object $P$ of $\sfC(\Proj \Lambda)$ 
is said to have property (P) if it has an increasing  filtration $0 = P_{(0)} \subset P_{(1)} \subset \cdots  $ 
such that each graded quotient $P_{(i)}/P_{(i -1)}$ is an object of $\sfC(\Proj \Lambda)$ 
with $0$ differential 
and that $\bigcup_{i \geq 0} P_{(i)} = P$. 
Every object $M \in \sfC(\Mod \Lambda)$ 
is quasi-isomorphic to an object $P \in \sfC(\Proj \Lambda)$ 
having property (P) 
and the derived tensor product $M \lotimes_{\Lambda} E$ 
is defined to be the quasi-isomorphism class of $P \lotimes_{\Lambda} E$. 

It follows from Lemma \ref{sublemma of compatibility lemma} 
that the restriction functor $(-)_{\Lambda}$ preserves the property (P). 
Now the desired statement follows from  
the commutativity of  the right   diagram of \eqref{two commutative diagrams}. 
\end{proof}

Thanks to Lemma \ref{compatibility lemma}, in what follows 
we safely drop the sign of restriction functors.

\begin{remark}\label{general base ring remark}
We remark that 
if we deal with derived functors involving bimodules 
in a suitable way, 
then we can generalize the contents of the paper 
to the case where the base $\kk$ is not a field but a commutative ring. 

As is pointed out in \cite[Remark 1.12]{Yekutieli dualizing complexes} 
in the case where the base commutative ring $\kk$ is not a field,  
a proper way to deal with bimodules is provided by theory of DG-algebras (see e.g. \cite{Positselski}). 

First,  we need to resolve $\Lambda$ as DG-algebras. 
 Namely,  we take a quasi-isomorphism $\tilde{\Lambda} \xrightarrow{\simeq} \Lambda$ of DG-algebra
 with a cofibrant DG-algebra. 
 It is known  that 
 the derived  category $\sfD(\tilde{\Lambda}) $ of DG-$\tilde{\Lambda}$-modules 
 is triangulated equivalent to $\sfD(\Mod \Lambda)$. 
Then,  
 the derived category  $\sfD(\tilde{\Lambda}^{\op} \otimes_{\kk} \tilde{\Lambda})$ 
of DG-$\tilde{\Lambda}^{\op} \otimes_{\kk} \tilde{\Lambda}$-modules  
is served  as a proper derived category of $\Lambda$-$\Lambda$-bimodules. 
An analogue of Lemma \ref{compatibility lemma} holds for this derived category 
even when $\kk$ is a commutative ring. 
Moreover, 
using 
 the derived category  $\sfD(\tilde{\Lambda}^{\op} \otimes_{\kk} \tilde{\Lambda})$ 
 we can prove all the results given in this paper. 
\end{remark}

\subsection{Admissible subcategories and recollements}\label{Recollection of triangulated categories}

In this Section we recall the notion of  an admissible subcategory and related results. 
%For details, we refer \cite{Krause:Localization theory}, \cite{Positselski}. 
%\color{blue}
%I think there exist more appropriate references. 
%\color{black} 

Let $\sfD$ be a triangulated category. 
Recall that a full triangulated subcategory $\sfE  \subset \sfD$ is called \textit{thick} 
if it is closed under taking  direct summands. 
For an object $d \in \sfD$, we denote 
the smallest thick  subcategory of $\sfD$ which contains $d$ 
by $\thick d$  and call it the \textit{thick hull} of $d$. 
An object $e \in \sfD$ belongs to $\thick d$ 
if and only if 
it is obtained from $d$ by taking shifts, cones and direct summands 
finitely many times. 
For a thick subcategory $\sfE$, we define its right orthogonal subcategory to be  
\[
{\sfE}^{\perp} : = \{ d \in \sfD \mid \Hom_{\sfD}(e,d) = 0 \textup{ for all } e \in \sfE\}.
\]  
In the similar way, we define the left orthogonal subcategory ${}^{\perp}\sfE$.

Recall that a thick subcategory $\sfE \subset \sfD$  is said to be  \textit{right} (resp. \textit{left}) \textit{admissible} 
if the inclusion functor $\mathsf{in}_{\sfE} : \sfE \hookrightarrow \sfD$ 
has a right (resp. left) adjoint functor, which will be  often denoted by $\tau$ in the sequel. 
A thick subcategory is called \textit{admissible} 
if it is both left and right admissible. 

We say that a triangulated category $\sfD$ has a \textit{semiorthogonal decomposition}  
$\sfD = \sfE \perp \sfF$ by thick subcategories $\sfE, \sfF \subset \sfD$ 
if $\sfE$ is right admissible  and $\sfE^{\perp} = \sfF$. 
We will use the following  characterization of right admissibility for thick subcategories 
in the proof of Lemma \ref{T proposition}.

%
%\color{blue}
%However, 
%in \cite{Krause:Localization theory} statement is not precise, 
%in \cite{Positselski} no proof is written. 
%\color{black}  

\begin{lemma}[{\cite{Bondal-Kapranov}%see  \cite[Proposition 4.9.1]{Krause:Localization theory}, \cite[Lemma 1.3]{Positselski}
}]\label{right admissible lemma} 
For a thick subcategory $\sfE \subset \sfD$, the following conditions are equivalent. 
\begin{enumerate}[(1)] 
\item $\sfE \subset \sfD$ is right admissible.

\item   
Every $d \in \sfD$   
fits into  an exact triangle $ e \to d \to k \to e[1]$ 
such that $ e \in \sfE, \ k \in {\sfE}^{\perp}$. 
\end{enumerate}
Moreover if one of the above conditions is satisfied, 
then  the right adjoint functor  $\tau$ sends objects $d \in \sfD$ 
to $e \in \sfE$ which is appeared as the left most term 
of the exact triangle in (2).  
\end{lemma}

%Later in Section \ref{Locally perfect complexes:Applications},  
%we need to use following property of a semiorthogonal decomposition and a recollement. 
Assume that we have a semiorthogonal decomposition $\sfD = \sfE \perp \sfF$. 
We have the following commutative diagram
\[
\begin{xymatrix}{ 
\sfE \ \ar@{^{(}->}[rr]^{\mathsf{in}_{\sfE}} 
\ar@{..>}[d]^{\simeq}_{(\mathsf{qt}_{\sfF} \circ \mathsf{in}_{\sfE})|_{\sfE} \ \ } && 
\sfD \ar@{=}[d] \ar@{->>}[rr]^{\mathsf{qt}_{\sfE}} && 
\sfD/\sfE \\
\sfD/\sfF && \sfD \ar@{->>}[ll]_{\mathsf{qt}_{\sfF}}  
&&\ \sfF, \ar@{..>}[u]_{\simeq} \ar@{_{(}->}[ll]_{\mathsf{in}_{\sfF}}
}\end{xymatrix}
\quad
\begin{xymatrix}{
\sfE && \sfD \ar[ll]_{\tau}  \ar@{=}[d] \\ 
\sfD/\sfF \ar[u]^{ ( (\mathsf{qt}_{\sfF} \circ \mathsf{in}_{\sfE})|_{\sfE})^{-1} \ \ } 
&& 
\sfD \ar[ll]^{\mathsf{qt}_{\sfF}}
}\end{xymatrix}
\]
where $\mathsf{in}_{\sfE}, \mathsf{in}_{\sfF}$ are canonical inclusions and 
$\mathsf{qt}_{\sfE}, \mathsf{qt}_{\sfF}$ are canonical quotient functors. 
It is known that 
the restriction functor  
$(\mathsf{qt}_{\sfF} \circ \mathsf{in}_{\sfE})|_{\sfE}$ 
is an equivalence and that 
the composite functor 
$((\mathsf{qt}_{\sfF} \circ \mathsf{in}_{\sfE})|_{\sfE})^{-1} \circ \mathsf{qt}_{\sfF}$ 
is naturally isomorphic to the right adjoint functor $\tau$ of $\mathsf{in}_{\sfE}$. 
We mention that  the similar composite functor gives an equivalence $\sfF \xrightarrow{\sim} \sfD/\sfE$.  
We can immediately check the following lemma. 

\begin{lemma}\label{right adjoint lemma}
For $d, d' \in \sfD$, we have the following natural isomorphism
\[
\Hom_{\sfD}(\tau (d), d') \cong \Hom_{\sfD/\sfF} ( \mathsf{qt}_{\sfF}(d), \mathsf{qt}_{\sfF}(d'))
\]
where we suppress the inclusion functor $\mathsf{in}_{\sfE}$ as custom and 
write the object $\mathsf{in}_{\sfE} (\tau (d))$ of $\sfD$  as $\tau(d)$.
\end{lemma}

A thick subcategory $\sfE \subset \sfD$ is admissible 
if and only if it fits into the following recollement of the left of \eqref{201803061422}.
\begin{equation}\label{201803061422}
\begin{xymatrix}{
\sfE \ar[rr]^{\mathsf{in}_{\sfE}} 
&& \sfD \ar@/^5mm/[ll]^{\tau} \ar@/_5mm/[ll] \ar[rr]^{\mathsf{qt}_{\sfE}}
&& \sfD/\sfE, \ar@/^5mm/[ll]  \ar@/_5mm/[ll]
}\end{xymatrix}
\begin{xymatrix}{
\sfD/ \sfF \ar[rr]  
&& \sfD \ar@/^5mm/[ll]^{\mathsf{qt}_{\sfF}} \ar@/_5mm/[ll] \ar[rr]
&& \sfF. \ar@/^5mm/[ll]^{\mathsf{in}_{\sfF}}  \ar@/_5mm/[ll]
}\end{xymatrix}
\end{equation}
If an admissible subcategory $\sfE$ is a piece of  a semi-orthogonal decomposition $\sfD = \sfE \perp \sfF$, 
then the recollement of the left of \eqref{201803061422} 
become the recollement of the right of \eqref{201803061422} 
via the equivalences $\sfE \simeq \sfD/\sfF, \sfD/\sfE \simeq \sfF$ mentioned above.

\section{Locally perfect complexes and locally perfect CM-modules}\label{section lp}

In Section \ref{section lp}, $A = \bigoplus_{i =0 }^{\ell} A_{i}$ denotes a finitely graded Noetherian algebra.  
We introduce notions of locally perfect complex and  locally perfect CM-module over $A$.

\subsection{Locally perfect complexes}

\begin{definition}\label{locally perfect}
A complex $P \in \sfC(\grproj A)$ is called \emph{locally perfect} 
if the complex $\frkp_i P$ belongs to $\sfC^{\mrb}(\proj \Lambda)$ for any $i \in \mathbb{Z}$.
We denote by 
$\sfC_{\mrlp}(\grproj A)$ 
the full subcategory of $\sfC(\grproj A)$ consisting of locally perfect complexes.
We define 
\[
\sfC^{\bullet}_{\mrlp}(\grproj A):=\sfC_{\mrlp}(\grproj A) \cap \sfC^{\bullet}(\grproj A).
\] 
where $\bullet =-,\mrb,\mrac$, etc.
We denote by 
$\sfK^{\bullet}_{\mrlp}(\grproj A)$ 
the homotopy category of $\sfC^{\bullet}_{\mrlp}(\grproj A)$.

Next, we discuss locally perfectness 
of objects of derived categories. 
For simplicity, 
we  only deal with the bounded derived category $\sfD^{\mrb}(\grmod A)$. 
An object $M \in \sfD^{\mrb}(\grmod A)$ is called \emph{locally perfect} 
if it is represented by  a locally perfect complex $P \in \sfC_{\mrlp}(\grproj A)$. 
We note that 
$M \in \sfD^{\mrb}(\grmod A)$  is locally perfect 
if and only if 
it is represented by $P \in \sfC_{\mrlp}^{-, \mrb}(\grproj A)$. 
We denote by $
\sfD^{\mrb}_{\mrlp}(\grmod A)
$
the full subcategory of $\sfD^{\mrb}(\grmod A)$ consisting of locally perfect objects. 
Obviously we have
\begin{eqnarray}
\sfK^{\mrb}(\grproj A) \subset \sfD^{\mrb}_{\mrlp}(\grmod A) \subset \sfD^{\mrb}(\grmod A). \label{D_lp>K^b}
\end{eqnarray}
\end{definition}

%We remark that the notion of locally perfect complex works for  not necessary finitely graded algebras. 

Since the functor $\frkp_{i}: \sfK(\grproj A) \to \sfK(\proj A_{0})$ is exact, 
the first statement of  the  lemma  below follows. 

\begin{lemma}
\begin{enumerate}[(1)]
\item $\sfK_{\mrlp}^{\bullet}(\grproj A)$ is a thick subcategory of $\sfK^{\bullet}(\grproj A)$.

\item $\sfD_{\mrlp}^{\mrb}(\grmod A)$ is a thick subcategory of 
$\sfD^{\mrb}(\grmod A)$. 
\end{enumerate}
\end{lemma} 

We can deduce the following lemma from Lemma \ref{quasi-Veronese lemma}. 

\begin{lemma}\label{qv lp lemma}
The equivalence $\sfqv$ induces an equivalence 
\[
\sfqv: \sfD^{\mrb}_{\mrlp}(\grmod A) \xrightarrow{\sim} \sfD^{\mrb}_{\mrlp}(\grmod A^{[\ell]}).
\]
\end{lemma} 

%From now, we mainly deal with the case where $A = \Lambda \oplus C$ is a trivial extension algebra. 

If $\pd_{A_{0}}A < \infty$, 
then the locally perfectness of $M \in \sfD^{\mrb}(\grmod A)$ can be checked by looking 
graded pieces $M_{i} \in \sfD^{\mrb}(\mod A_{0})$. 

\begin{proposition}\label{lp_resol}
Assume that $\pd_{A_{0}} A < \infty$.  
Then, an object $M \in \sfD^{\mrb}(\grmod A)$ is locally perfect 
if and only if  $M_i$ belongs to $\sfK^{\mrb}(\proj A_{0} )$  for any $i\in\ZZ$. 
 %In other word we have 
%\[\sfD^{\mrb}_{\mrlp}(\grmod A)=\{ M \in \sfD^{\mrb}(\grmod A) \mid  M_i  \in \sfK^{\mrb}(\proj \Lambda)  \textup{ for } i \in \ZZ \}.\]
\end{proposition}

First we show that we can reduce the problem to the case of a trivial extension algebra  by using the quasi-Veronese algebra construction. 

\begin{lemma}\label{sublemma of lp_resol}
Assume that $\pd_{A_{0}} A < \infty$.  
Let $M $ be an object of $\sfD^{\mrb}(\grmod A)$.
Then the $i$-th graded submodule $M_{i}$ belongs to $\sfK^{\mrb}(\proj A_{0})$ for $i \in \ZZ$ 
if and only if 
$(\sfqv M)_{i}$ belongs to $\sfK^{\mrb}(\proj \nabla A)$ for $i \in \ZZ$. 
\end{lemma}

\begin{proof}
Let $\Gamma := \begin{pmatrix} \Gamma_{0} & D \\ 0 & \Gamma_{1} \end{pmatrix}$ be an upper triangular matrix algebra 
and $e_{0} := \begin{pmatrix} 1_{\Gamma_{0}} & 0 \\ 0 & 0 \end{pmatrix}, 
e_{1} := \begin{pmatrix} 0 & 0 \\ 0 & 1_{\Gamma_{1}}\end{pmatrix}$.  
Assume that $\Gamma$ is Noetherian. 
We note that it is equivalent  to assume that $\Gamma_{0}$ and $\Gamma_{1}$ are Noetherian and
the modules  ${}_{\Gamma_{0}}D$ and $D_{\Gamma_{1}}$ are  finitely generated over $\Gamma_{0}$ and $\Gamma_{1}$ respectively.  
Assume moreover that $\pd_{\Gamma_{1}} D < \infty$. 
Then, by a similar argument to the proof of \cite[Proposition 6.1]{adasore} 
we can prove that 
an object $N \in \sfD^{\mrb}(\mod \Gamma)$ belongs to $\sfK^{\mrb}(\proj \Gamma)$ 
 if and only if 
 $Ne_{i}$ belongs to $\sfK^{\mrb}(\proj \Gamma_{i})$ 
 for $i = 0,1$.
 
Applying the last statement repeatedly to the quasi-Veronese algebra construction, 
we prove the desired statement. 
\end{proof}

For the proof, we intrduce a notion, which is also used later. 
Let $\Lambda$ be an algebra 
and $ Q \in \sfC(\proj \Lambda)$ a complex 
which is  homotopic to a bounded complex $Q' \in \sfC^{\mrb}(\proj \Lambda)$. 
We define  the \textit{upper bound} $\ub Q \in \ZZ$ of $Q$ 
by  $\ub Q:= \max\{n \in \ZZ \mid  \tuH^{n}(Q) \neq 0 \}$.

\begin{proof}[Proof of Proposition \ref{lp_resol}] 
By Lemma \ref{sublemma of lp_resol}, 
we may assume $A = \Lambda \oplus C$.

We prove ``if" part. 
Let $M\neq 0 $ be an object  in $\sfD^{\mrb}(\grmod A)$ such that $M_i \in \sfK^{\mrb}(\proj \Lambda)$ for  $i\in\ZZ$.
Let  $P \in \sfC^{-,\mrb}(\grproj A)$ be a projective resolution of $M$.   
We note that $M_{i} = P_{i}$ in $\sfD^{\mrb}(\mod \Lambda)$.  
We set $k: = \min\{ i \in \ZZ \mid P_{i} \neq 0\}$.

We claim that $\frkp_{i}P$ belongs to $\sfK^{\mrb}(\proj \Lambda)$ for $i \in \ZZ$.  
The  case $i <  k$ is clear since $\frkp_{i} P =0$. 
The case $i= k$ follows from $P_{k} = \frkp_{k}P$. 
The case $i  > k$ is shown by induction.  
Assume that $\frkp_{i-1}P$ belongs to $\sfK^{\mrb}(\proj \Lambda)$. 
Since  $\pd C_{\Lambda} < \infty$,  
the object $\frkp_{i-1} P \lotimes_{\Lambda} C$ belongs to $\sfK^{\mrb}(\proj \Lambda)$. 
By the assumption $P_{i}$ belongs to $\sfK^{\mrb}(\proj \Lambda)$. 
Using the exact triangle  \eqref{kurumimochi}, we deduce that $\frkp_{i}P$ belongs to $\sfK^{\mrb}(\proj \Lambda)$.

By the claim,  the complex  $\frkp_{i} P \in \sfC^{-}(\proj \Lambda)$ 
is homotopic to some complex $Q_{i} \in \sfC^{\mrb}(\proj \Lambda)$ for  $i \in \ZZ$. 
We may take $Q_{i} = 0$ for $i < k$ and assume that  for $i \geq k$ the following equality holds  
\begin{eqnarray}\label{ineq_P'} 
\max\{ m \in\ZZ  \mid  (Q_{i})^m \neq 0 \} =\ub (\frkp_{i} P). 
\end{eqnarray}
By \cite[Lemma 4.7]{adasore}, 
there is a complex $P' \in \sfC(\grProj A)$ which is homotopic to $P$ such that 
$\frkp_{i}P' = Q_{i}$ in $\sfC(\proj \Lambda)$  for $i \in \ZZ$. 
Observe that $P'$ belongs to $\sfC_{\mrlp}^{-,\mrb}(\GrProj A)$.

To complete the proof, we only have to show 
that  the graded projective $A$-module  $(P')^{n}$ is finitely generated   for $n \in \ZZ$.  
We fix $n \in \ZZ$ and use the criterion of  Lemma \ref{fg lemma}. 
By the construction  
each $\frkp_{i}(P')^{n} = Q_{i}^{n}$ is finitely generated for $i \in \ZZ$ 
and 
$\frkp_{i} (P')^{n} =Q_{i}^{n} = 0$ for $i < k$. 
Therefore, 
it is enough to show  that $\frkp_{i}(P')^{n} = Q_{i}^{n} = 0$ for $i \gg 0$. 

Since $M$ is assumed to have bounded cohomology groups and $A$ is finitely graded, 
there exists $j \in \ZZ$  such that $\tuH(M)_{i} = 0$ for $i \geq j$. 
Thus, using Corollary \ref{black thunder}.(1), 
we deduce that $\ub (\frkp_{i} P )< n $ for $i \gg 0$. 
By \eqref{ineq_P'}, we conclude   that $Q_{i}^{n} = 0$ for $i  \gg 0$ as desired. 
This  completes the proof of ``if" part.

We prove ``only if" part. 
Let  $M \in \sfD^{\mrb}_{\mrlp}(\grmod A)$.
There exists a locally perfect complex $P \in \sfC^{-,\mrb}_{\mrlp}(\grproj A)$ which is quasi-isomorphic to $M$.
Since by the assumption $\frkp_{i} P $ and $\frkp_{i-1} P \lotimes_{\Lambda} C$ belong to $\sfK^{\mrb}(\proj \Lambda)$, 
we conclude by the exact triangle \eqref{kurumimochi} that $M_{i}$ belongs to $\sfK^{\mrb}(\proj \Lambda)$. 
\end{proof}

From the above proposition, 
we deduce a condition that every object $M \in \sfD^{\mrb}(\grmod A)$ is locally perfect.

\begin{corollary}
The following conditions are equivalent. 
\begin{enumerate}[(1)]
\item $\sfD^{\mrb}_{\mrlp}(\grmod A)=\sfD^{\mrb}(\grmod A)$. 

\item $\sfK^{\mrb}(\proj A_{0}) = \sfD^{\mrb}(\mod A_{0})$. 
\end{enumerate}
\end{corollary}

\begin{proof}
We may assume $A = \Lambda \oplus C$. 
Assume that the condition (1) holds. 
Then $\sfD^{\mrb}(\mod \Lambda) \subset \sfD^{\mrb}_{\mrlp}(\grmod A)$. 
It follows from \cite[Lemma 4.13]{adasore} that $\sfD^{\mrb}(\mod \Lambda) \subset \sfK^{\mrb}(\proj \Lambda)$.
This shows the implication (1) $\Rightarrow$ (2). 

Assume that the condition (2) holds. 
We remark that  since $A_{i}$ is finitely generated $A_{0}$-module, it has finite projective dimension. 
Now, it is easy to deduce the condition (1)  by using Proposition \ref{lp_resol}.
\end{proof}

In the case where $A$ is finite dimensional, 
the condition (2) is equivalent to the condition $\gldim A_{0} < \infty$. 

\begin{corollary}\label{mrlp corollary}
Assume that $A$ is a finite dimensional graded  algebra. 
Then the equality $\sfD^{\mrb}_{\mrlp}(\grmod A)=\sfD^{\mrb}(\grmod A)$ holds if and only if $\gldim A_0 < \infty$.
\end{corollary}

The following proposition concerning on the $A$-dualities plays important roles in the sequel.

\begin{proposition}\label{lp_A-dual}
Assume that $A$ is IG and that it has finite projective dimension as 
a left and  right $A_{0}$-module. 
%$\pd_{A_{0}}A< \infty $ and $\pd_{A_{0}^{\op}}A< \infty$.  
Then the $A$-duality $(-)^{*}$  induces an equivalence 
\[
(-)^{*}: =\grRHom_{A}(-,A): \sfD_{\mrlp}^{\mrb}(\grmod A) \stackrel{\sim}{\rightleftarrows}
 \sfD_{\mrlp}^{\mrb}(\grmod A^{\op})^{\op}:
\grRHom_{A^{\op}}(-,A) =:(-)^{*}.
\]
\end{proposition}

\begin{proof}
By \cite[Proposition 6.1]{adasore}, 
$\Delta A$ has finite $\nabla A$-projective dimension on both sides. 
Hence, we may reduce the problem to the case where  $A = \Lambda \oplus C$ by Lemma \ref{qv lp lemma}.
 
It is enough to show that $(-)^{*}$ preserves locally perfectness. 
Since the argument is left-right symmetric, 
 it is enough to prove that if $M \in \sfD^{\mrb}(\grmod A)$ is locally perfect, 
then so is $M^{*}$. 
Since $M^{*}$ belongs to $\sfD^{\mrb}(\grmod A^{\op})$, 
it is enough to show that $(M^{*})_{i} $ belongs to $\sfK^{\mrb}(\proj \Lambda^{\op})$ by the left version of Proposition \ref{lp_resol}. 
 
Let  $P \in \sfC^{-,\mrb}_{\mrlp}(\grproj A)$ be a projective resolution of  $M$.
Then, substituting $P' =A(i)$ in the exact sequence (4-5) of \cite{adasore}, 
we obtain  an exact triangle in $\sfD(\mod \Lambda^{\op})$   for $i \in \ZZ$
\[
\Hom_{\Lambda}( \frkp_{-i+1}P, C) \to (M^{*})_{i} \to \Hom_{\Lambda}(\frkp_{-i}P, \Lambda) \to .  
\]
The right most term belongs to $\sfK^{\mrb}(\proj  \Lambda^{\op})$. 
The assumption  $\pd_{\Lambda^{\op}} C < \infty$ implies that  
so does the left most term. 
 Thus we conclude that  so does $(M^{*})_{i}$. 
\end{proof}

%\begin{remark}We can verify the results of this section  for a finitely graded algebra $A= \bigoplus_{i\geq 0}^{\ell}A_{i}$ 
%satisfying suitable assumptions by using quasi-Veronese algebra construction (see \cite[Section 2.4, Section 6]{adasore}).\end{remark}

\subsection{Locally perfect graded CM-modules}

\begin{definition}
For a graded IG-algebra $A$, we define
\[
\grCM_{\mrlp} A=\{ \tuZ^0(P) \ | \ P \in \sfC^{\mrac}_{\mrlp} (\grproj A)  \}.
\]
This is a Frobenius full subcategory of $\grCM A$ containing $\grproj A$.
So the stable category $\stabgrCM_{\mrlp}A$ is a triangulated full subcategory of $\stabgrCM A$.

By the definition, 
it is obvious that the equivalence 
$\underline{\tuZ}^0:\sfK^{\mrac}(\grproj A) \xrightarrow{\simeq} \stabgrCM A$ induces an equivalence 
\[
\underline{\tuZ}^0:\sfK^{\mrac}_{\mrlp}(\grproj A) \xrightarrow{\simeq} \stabgrCM_{\mrlp} A.
\]
\end{definition}

The equivalence $\beta: \stabgrCM A \xrightarrow{\sim} \grSing A$ can be restricted to the locally perfect subcategories. 

\begin{lemma}\label{lp=finproj}
Let $A$ be a finitely graded IG-algebra.
If $A$ has finite projective dimension as a left and right $A_0$-module, we have 
\[
\grCM_{\mrlp} A
=(\grCM A) \cap  \sfD^{\mrb}_{\mrlp}(\grmod A)  
=\{ M \in \grCM A \ | \ \pd (M_{A_0}) < \infty  \}.
\]
\end{lemma}
\begin{proof}
By Proposition \ref{lp_resol}, we have the second equality. 
It is obvious that $\grCM_{\mrlp} A$ is contained in $
(\grCM A) \cap  \sfD^{\mrb}_{\mrlp}(\grmod A)  $.
In the following, we prove the converse inclusion.

Let $M \in (\grCM A)\cap \sfD^{\mrb}_{\mrlp}(\grmod A)$.
By  Proposition \ref{lp_A-dual}, $M^*$ belongs to $(\grCM A^{\op})\cap \sfD^{\mrb}_{\mrlp}(\grmod A^{\op})$. 
Let $f: Q \xrightarrow{\sim} M^{*}$ be a quasi-isomorphism with $Q \in \sfC^{-,\mrb}_{\mrlp}(\grproj A^{\op})$. 
We may assume that $Q^{> 0} = 0$. 
We regard $Q$ as a projective resolution  
\[ \cdots\cdots \to Q^{1} \to Q^{0} \to M^* \to 0
\]of $M^*$ in $\grmod A^{\op}$.
Applying $\grHom_{A^{\op}}(-,A)$, we have an exact sequence
\begin{equation}\label{201803140122I}
0 \to M^{**} \to P^{1} \to P^{2} \to \cdots\cdots 
\end{equation}
in $\grmod A$ 
where we set $P^{i} := \grHom_{A^{\op}}(Q^{-i +1}, A)$. 
Similarly, we take a projective resolution $P \xrightarrow{\sim} M$ with 
$P \in \sfC^{-,\mrb}_{\mrlp}(\grproj A)$ 
satisfying $P^{> 0} = 0$ 
and regard it  
as a projective resolution 
\begin{equation}\label{201803140122II}
\cdots\cdots \to P^{-1} \to P^{0} \to M \to 0
\end{equation}
of $M$ in $\grmod A$.
Splicing  the  exact sequences \eqref{201803140122I} and \eqref{201803140122II}, 
we have an acyclic locally perfect complex
\[
P \ : \ \xymatrix{
\cdots\cdots \ar[r] & P^{-1} \ar[r] & P^{0} \ar[r] \ar[d] & P^{1} \ar[r] & P^{2} \ar[r] & \cdots\cdots \\
 & & M \ar[r]^{\cong \hspace{2mm}} & M^{**} \ar[u] & & 
}
\]
such that $\tuZ^0(P) \simeq M$.
Thus $M$ belongs to $\grCM_{\mrlp} A$, and so the first equality holds.
\end{proof}

\begin{definition}
Let $A$ be a graded IG-algebra. 
Then by the observation \eqref{D_lp>K^b}, we can define the \emph{locally perfect singularity category} as the Verdier quotient
\[
\lpSing A:= \sfD_{\mrlp}^{\mrb}(\grmod A)/\sfK^{\mrb}(\grproj A),
\]
which can be regarded as a triangulated full subcategory of $\grSing A$.
\end{definition}

\begin{lemma}
Let $A$ be a finitely graded IG-algebra.
We assume that $A$ has finite projective dimension as a left and right $A_0$-module.
Then the equivalence $\beta:\stabgrCM A  \xrightarrow{\simeq} \grSing A$ induces an equivalence 
\[
\beta:\stabgrCM_{\mrlp} A \xrightarrow{\simeq} \lpSing A.
\]
\end{lemma}

\begin{proof}
It follows  from Lemma \ref{lp=finproj} that $\beta(M)$ belongs to $\lpSing A$ for $M \in \stabgrCM_{\mrlp} A$.
So we get a fully faithful functor $\beta:\stabgrCM_{\mrlp} A \to \lpSing A$.
We prove this functor is essentially surjective. 
Let $X$ be an object of $\lpSing A$ and $X' \in \sfD_{\mrlp}^{\mrb}(\grmod A)$ a representative of $X$.  
There exists   $M \in \grCM A$ such that $\beta(M) \cong X$ by Theorem \ref{Buchweitz-Happel theorem}. 
Then there exists a diagram inside $\sfD^{\mrb}(\grmod A)$ 
\[
M \xleftarrow{ f} N \xrightarrow{g} X' 
\]
such that the cones $\cone(f), \cone(g)$ of $f,g$ belong to $\sfK^{\mrb}(\grproj A)$.  
In other words there exist exact triangles 
\begin{equation}\label{201803061452}
N \xrightarrow{g}  X' \to \cone(g) \to, \ \ N \xrightarrow{f} M \to \cone(f) \to.  
\end{equation} 
From the left exact triangle of \eqref{201803061452}, we see that $N$ is locally perfect.  
Then from the right exact triangle of \eqref{201803061452}, we deduce $M \in \sfD_{\mrlp}^{\mrb}(\mod A)$. 
Therefore, $M$ belongs to $\grCM A \cap \sfD_{\mrlp}^{\mrb}(\mod A) = \grCM_{\mrlp} A$.   
\end{proof}

\subsection{Orlov's equivalence}\label{Orlov's equivalence}

Let $A$ be a finitely graded IG-algebra.
Assume that $A$ has finite projective dimension as a left and right $A_0$-module.

Following Orlov \cite{Orlov}, we set 
\[
\mathsf{O} := \sfD^{\mrb}_{\mrlp}(\mod^{\geq 0} A) \cap \sfD^{\mrb}_{\mrlp}(\mod^{> 0} A^{\op})^*.
\]

\begin{theorem}\label{lp Orlov theorem}
The canonical functor $\pi: \sfD^{\mrb}_{\mrlp}(\grmod A) \to \grSing_{\mrlp} A$ induces an equivalence
\[
\pi|_{\sfO} : \mathsf{O} \to \grSing_{\mrlp} A.
\]
\end{theorem}

Although we can prove this theorem by the same method with \cite{Orlov}, 
we give a different proof in Section \ref{around}.

\section{Categorical characterizations of an asid bimodule
%Categorical characterization of asid module for an IG algebra}\label{Categorical characterization
}\label{Categorical characterization}

Thanks to the quasi-Veronese algebra construction, 
representation theoretic problem of  a finitely graded algebra 
$A = \bigoplus_{i=0}^{\ell} A_{i}$ 
can be reduced to that of a trivial extension algebra $A= \Lambda \oplus C$ 
with the canonical grading $\deg \Lambda = 0, \deg C = 1$.  

It was obtained in \cite{adasore} 
the condition that $A = \Lambda \oplus C$ is IG 
in terms of $\Lambda$ and $C$ by using derived tensor products 
and derived $\Hom$. 
The aim of Section \ref{Categorical characterization} is  
to prove two theorems. 
The first one, Theorem \ref{Adaching theorem 0}, 
gives two categorical characterizations 
that $A = \Lambda \oplus C$ is IG. 
The second, Theorem \ref{Adaching theorem 1}, 
verifies several properties of thick subcategories and invariants  
appearing in the categorical characterizations.

%Thus, after Section \ref{section lp}, we mainly deal with a trivial extension algebra. 

\subsection{The asid   conditions and the asid numbers }\label{review of adasore}

In  Section \ref{review of adasore}, we recall the notion of asid-bimodules and related results from \cite{adasore}.   

For a bimodule $C$ over $\Lambda$, 
we define  a morphism  $\tilde{\lambda}_{r}: \Lambda \to \Hom_{\Lambda}(C,C)$ by the formula 
$\tilde{\lambda}_{r}(x)(c) := xc$ for $x \in \Lambda$ and $c \in C$. 
We denote the composite morphism  $\lambda_{r} = \mathsf{can} \circ \tilde{\lambda}_{r}$ in $\sfD(\Mod \Lambda)$ 
where $\mathsf{can}$ is the canonical morphism $\Hom_{\Lambda}(C,C) \to \RHom_{\Lambda}(C,C)$.
\[
\lambda_{r}: 
\Lambda \xrightarrow{\tilde{\lambda}_{r}} \Hom_{\Lambda}(C,C) \xrightarrow{\mathsf{can}} \RHom_{\Lambda}(C,C). 
\]
We denote by $\lambda_{\ell}$  the left version of $\lambda_{r}$. 
\[
\lambda_{\ell}: \Lambda \to \RHom_{\Lambda^{\op}}(C, C). 
\]
Using these morphisms, we give a condition for $C$ that the trivial extension algebra  $A = \Lambda \oplus C$ is IG.

\begin{theorem}[{\cite[Theorem 5.14, Proposition 5.16]{adasore}}]\label{adasore:right asid theorem} 
Let $\Lambda$ be an IG-algebra and $C$ a bimodule over $\Lambda$ 
which is finitely generated on both sides. 
Then the trivial extension algebra $A = \Lambda \oplus C$ is IG 
if and only if the following conditions are satisfied. 

\begin{enumerate}[(1)]
\item  $C$ has finite projective dimensions on both sides. 

\item 
``The right asid condition''.     
The morphism $\RHom_{\Lambda}(C^{a}, \lambda_{r} )$ is an isomorphism for $a \gg 0$. 

\item 
``The left asid condition''. 
The morphism $\RHom_{\Lambda^{\op}}(C^{a}, \lambda_{\ell} )$ is an isomorphism for $a \gg 0$. 
\end{enumerate}
\end{theorem}

\begin{remark}
In \cite{adasore} the right and the left asid  conditions are called 
``the right and the left  \asid \ condition 3". 
The first condition in the above theorem is called the ``the right and the left \asid \ condition 1".
\end{remark} 

\begin{definition}\label{adasore:asid number definition}
\begin{enumerate}[(1)]
\item 
A $\Lambda$-$\Lambda$-bimodule $C$ is called 
a \textit{asid (attaching self-injective dimension) bimodule} 
if the trivial extension algebra $A = \Lambda \oplus C$ is IG.

\item 
For an  asid bimodule $C$, we define the \textit{right asid number} $\alpha_{r}$ 
and the \textit{left asid number} $\alpha_{\ell}$  to be 
\[
\begin{split}
\alpha_{r}  &:= \min\{ a \geq 0 \mid \RHom_{\Lambda}(C^{a}, \lambda_{r} ) \textup{ is an isomorphism} \}, \\
 \alpha_{\ell}  &:= \min\{ a \geq 0 \mid \RHom_{\Lambda^{\op}}(C^{a}, \lambda_{\ell} ) \textup{ is an isomorphism} \}. 
 \end{split}
\]
\end{enumerate}
\end{definition}

Another description of the right and the left asid numbers is given in the following proposition. 
%We only give a statement for the right asid number, since the same equality holds for the left asid number. 

\begin{proposition}[{\cite[Corollary 5.12]{adasore}}]\label{adasore:asid number corollary}
Let  $C$ be a  asid bimodule over $\Lambda$ 
and $A = \Lambda \oplus C$ the trivial extension algebra.  
We regard  a minimal injective resolution $I$ of $A$ as a complex. 
Then, 
\begin{equation}\label{asid number formular 1}
\alpha_{r}  =1-  \min\{ a\in \ZZ \mid \RHom_{A}(\Lambda, A)_{a} \neq 0\}  =
1 - \min \{ a \in \ZZ \mid \frki_{a} I \neq 0 \textup{ in } \sfD(\Mod \Lambda) \}.
\end{equation}
\end{proposition}

\begin{remark}
For $a > 1$, we always have $\RHom_{A}(\Lambda, A)_{a} = 0$ and $\frki_{a} I = 0$. 
Therefore the proposition above is the same with \cite[Corollary 5.12]{adasore}. 
\end{remark}

For the definition of $\frki_{-a}I$, we refer \cite[Section 2.3]{adasore}. 
Roughly speaking, it is a complex of injective $\Lambda$-modules 
formed by the cogenerating module of each term $I^{n}$ in the graded degree $-a$. 

In the case where $\Lambda$ is a finite dimensional algebra, 
the above formula turns out to be written down by graded cosyzygies of $A$. 

\begin{corollary} 
Let $\Lambda, A$ and $C$ be as in the above proposition. 
Assume moreover  that $\Lambda$ is a finite dimensional algebra. 
We consider the graded cosyzygies $\Omega^{-n}A$ of $A$. 
Then, 
\[
\alpha_{r}  = 1 - \min\{ a\in \ZZ \mid \exists n \textup{ s.t. } \soc (\Omega^{-n}A)_{a} \neq 0\}. 
\]
\end{corollary}

\begin{remark}
If we assume $C \neq 0$, 
then $\max\{ a\in \ZZ \mid \exists n \textup{ s.t. } \soc (\Omega^{-n}A)_{a} \neq 0\}  = 1$. 
Hence we obtain the equation below which tells that $\alpha_{r}$ is the amplitude of the degrees of the socles of graded cosyzygys $\Omega^{-n}A$.   
\begin{equation}\label{asid number formula}
\alpha_{r}  = \max\{ a\in \ZZ \mid \exists n \textup{ s.t. } \soc (\Omega^{-n}A)_{a} \neq 0\} -
\min\{ a\in  \ZZ \mid \exists n \textup{ s.t. } \soc (\Omega^{-n}A)_{a} \neq 0\}. 
\end{equation}
\end{remark}

\subsubsection{Happel's  functor $\varpi$ }\label{The functor varpi}

 Let $A = \Lambda \oplus C$ be a trivial extension algebra. 
 We may regard  a $\Lambda$-modules $M$ as a  graded $A$-module $M$ such that $M_{i} = 0$ for $ i\neq 0$
and   embed the category $\Mod \Lambda$ 
 as the full subcategory of $\GrMod A$ consisting of graded $A$-modules $M$ 
 such that $M_{i} = 0 $ for $ i \neq 0$. 
 We also embed the derived category $\sfD(\Mod \Lambda)$ 
 as a full subcategory of $\sfD(\GrMod A )$ consisting of objects $M$ 
 such that $\tuH(M)_{i} = 0$ for $i \neq 0$. 
 
We assume that $A$ is Noetherian or equivalently $\Lambda$ is Noetherian and $C$ is finitely generated over $\Lambda$. 
 We define a functor $\varpi: \sfD^{\mrb} (\mod \Lambda) \to \grSing A$ 
 which plays a key role in the paper 
 as the following composite functor 
 \[
 \varpi: \sfD^{\mrb}(\mod \Lambda) \hookrightarrow \sfD^{\mrb}(\grmod A)  \xrightarrow{ \pi } \grSing A
 \]
 where $\pi$ is the canonical quotient functor. 
As is mentioned in the introduction, 
in the case where $A$ is IG, 
the functor $\cH :=\beta^{-1}\varpi: \sfD^{\mrb}(\mod \Lambda ) \to \stabgrCM A$ 
is a generalization of the functor $\cH : \sfD^{\mrb}(\mod \Lambda) \to \stabmod T(\Lambda)$ 
constructed by Happel. 
Thus, we call  $\varpi$ Happel's functor.

 Assume moreover that $\pd C_{\Lambda} < \infty$. 
 Then for $a \geq 0$, the functor $(- \lotimes_{\Lambda}C^{a}) $ 
 is restricted to the endofunctor $(- \lotimes_{\Lambda}C^{a})|_{\sfK}$ 
 of $\sfK^{\mrb}(\proj \Lambda)$. 
Observe that there exists the following increasing sequence of  thick subcategories of $\sfK^{\mrb}(\proj \Lambda)$. 
\begin{equation}\label{kernel sequence}
\Ker(- \lotimes_{\Lambda}C)|_{\sfK} \subset 
\Ker(- \lotimes_{\Lambda}C^{2})|_{\sfK} \subset \cdots 
\subset 
\Ker(- \lotimes_{\Lambda}C^{a})|_{\sfK} \subset \cdots.  
\end{equation}
It was shown that the union of this sequence is the kernel of $\varpi$. 
\begin{proposition}[{\cite[Corollary 4.18]{adasore}}]\label{adasore:description of kernel} 
Under the above situation, we have 
\[
\Ker\varpi = \bigcup_{a \geq 0} \Ker  (- \lotimes_{\Lambda}C^{a})|_{\sfK}.
\]
\end{proposition}

We end this section by pointing out a relationship 
between the functor $- \lotimes_{\Lambda}C$ and 
the degree shift functor $(1)$ on $\grSing A$. 
We denote by $\varpi|_{\sfK}$ the restriction to  the homotopy category $\sfK^{\mrb}(\proj \Lambda)$. 
 
\begin{lemma}\label{varpi shift lemma} 
Under the above situation, 
we have the following natural isomorphism 
between the functors from $\sfK^{\mrb}(\proj \Lambda) $ to $\grSing A$. 
\[
\varpi|_{\sfK} \circ ( -\lotimes_{\Lambda}C[1])|_{\sfK}  \cong (1) \circ \varpi|_{\sfK}. 
\]
\end{lemma}
 
 We note that 
 since $\Lambda$ is the degree $0$-subalgebra of $A$, 
 we may regard   
 a graded  $A$-$A$-bimodule ${}_{A}X_{A}$ 
 as a graded $\Lambda$ -$A$-bimodule ${}_{\Lambda} X_{A}$. 
 
 \begin{proof}
 Let $M \in \sfK^{\mrb}(\proj \Lambda)$. 
 Applying $M \lotimes_{\Lambda}-$ to the standard exact sequence 
 \[
 0 \to C \to A(1)  \to  \Lambda(1) \to 0,  
 \] 
 we obtain a morphism $f: M (1) \to M \lotimes_{\Lambda} C[1] $ in $\sfD^{\mrb}(\grmod A)$ 
 which is natural in $M$. 
 It is easy to see that the cone $\cone(f) = M \lotimes_{\Lambda} A [1]$ 
belongs to $\sfK^{\mrb}(\grproj A)$. 
Therefore $f$ becomes an isomorphism after passing to $\grSing A$.  
 \end{proof}

\subsection{Statements of Theorems}\label{Statements of Theorems}

In Section \ref{Statements of Theorems}, we give statements of 
main theorem of this section.

Assume that $\Lambda$ is Noetherian and $C$ has finite projective dimension on both sides. 
Then the $C$-dual functors $(-)^{\star} =\RHom(-,C)$  
can be restricted to the perfect derived categories 
\[
(-)^{\star}:= \RHom_{\Lambda} ( -, C): \sfK^{\mrb}(\proj \Lambda) \rightleftarrows \sfK^{\mrb}(\proj \Lambda^{\op})^{\op} 
:  \RHom_{\Lambda^{\op}}(-,C) = : (-)^{\star}. 
\]
We also use the $\Lambda$-dual functors. 
Since it is necessary to distinguish the right $\Lambda$-duality and the left $\Lambda$-duality, 
we denote by $(-)^{\rightvee} := \RHom_{\Lambda}(-, \Lambda)$ 
the right $\Lambda$-dual functor 
and by $(-)^{\leftvee} :=\RHom_{\Lambda^{\op}}(-, \Lambda)$
the left $\Lambda$-dual functor. 
\[
(-)^{\rightvee}:= \RHom_{\Lambda} ( -, \Lambda): 
\sfD(\mod \Lambda) \rightleftarrows \sfD(\mod \Lambda^{\op})^{\op} 
:  \RHom_{\Lambda^{\op}}(-,\Lambda) = : (-)^{\leftvee}. 
\]
We remark that these functors induce contravariant equivalences 
between the homotopy categories 
$\sfK^{\mrb}(\proj \Lambda)$ and 
$\sfK^{\mrb}(\proj \Lambda^{\op})$. 
\[
(-)^{\rightvee}: \sfK^{\mrb}(\proj \Lambda) 
\stackrel{\sim}{\rightleftarrows} \sfK^{\mrb}(\proj \Lambda^{\op})^{\op} 
: (-)^{\leftvee}. 
\]

The aim of this section is to prove the following two theorems. 
The first one gives categorical characterizations of asid bimodules. 

\begin{theorem}\label{Adaching theorem 0}
Let $\Lambda$ be an IG-algebra and  
$C$ a $\Lambda$-$\Lambda$-bimodules which is finitely generated on both sides and 
has finite projective dimension on both sides. 
The following conditions  are equivalent 
\begin{enumerate}[(1)]
\item $C$ is an asid bimodule.
\item 
$\sfK^{\mrb} (\proj \Lambda)$ has an admissible subcategory $\sfT$ such that 
\begin{enumerate}[(2-a)]
\item the functor $ \cT = -\lotimes_{\Lambda} C$ acts on $\sfT$ as an equivalence, 
 i.e., 
$\cT({\sfT}) \subset {\sfT}$ and 
the restriction functor $\cT|_{\sfT}$ is an autoequivalence.

\item the functor $\cT = -\lotimes_{\Lambda} C$ nilpotently acts on $\sfT^{\perp}$, i.e., 
$\cT({\sfT}^{\perp}) \subset {\sfT}^{\perp}$ and 
there exists a natural number $a \in \NN$ such that $\cT^{a}({\sfT}^{\perp}) =0$. 
\end{enumerate} 

\item 
$\sfK^{\mrb} (\proj \Lambda)$ has a thick subcategory $\sfT$ 
such that 
\begin{enumerate}[(3-a)]
\item 
the $C$-dual functors induce  equivalences 
\[
(-)^{\star} : \sfT \stackrel{\sim}{\rightleftarrows} \sfT^{\rightvee} : (-)^{\star}
\]
where $\sfT^{\rightvee}$ denotes the image of $\sfT$ by the functor $(-)^{\rightvee}$. 

\item 
there exists a natural number $a \geq 0$ such that 
$C^{a},(C^{a})^{\leftvee} \in \sfT$. 
\end{enumerate} 
\end{enumerate}
\end{theorem}

\begin{remark}
We remark that what we actually show in the proof provided in the next section is 
that if $\Lambda$ is Noetherian and a bimodule $C$ over it has finite projective dimension 
from both sides, then 
the right and the left asid conditions are equivalent 
to the conditions (2) or (3) of the above theorem. 
\end{remark}

The second main result of this section shows that the left and right asid numbers coincide with each other 
and that 
the triangulated subcategories $\sfT$ and $\sfT^{\perp}$ appeared in the above theorem 
are uniquely determined by an asid bimodule $C$.

\begin{theorem}\label{Adaching theorem 1}
Let $\Lambda$ be an IG-algebra and $C$ an asid bimodule. 
Then the following assertions hold. 

\begin{enumerate}[(1)]
\item   
We have $\alpha_{r} = \alpha_{\ell}$. We put $\alpha := \alpha_{r} = \alpha_{\ell}$. 

\item 
The subcategory $\sfT$ of (2) or (3) of Theorem \ref{Adaching theorem 0} is uniquely determined as 
$
\sfT = \thick C^{\alpha}.$
\item 
We have 
$
\sfT^{\perp} = \Ker \varpi = \Ker (- \lotimes_{\Lambda} C^{\alpha})|_{\sfK}. 
$
\end{enumerate}

\end{theorem}

We introduce a terminology.

\begin{definition}
Let $\Lambda$ be an IG-algebra, $C$ an asid bimodule and $\alpha:=\alpha_{r} = \alpha_{\ell}$. 
We call  the subcategory $\thick C^{\alpha}$ the \textit{asid} subcategory of $C$. 
\end{definition}

\subsection{Proof of Theorem \ref{Adaching theorem 0} and Theorem \ref{Adaching theorem 1}}

In the rest of Section \ref{Categorical characterization}, we keep the notation and the assumption of Theorem \ref{Adaching theorem 0}. 

\subsubsection{Interpretations of  the  asid conditions}
First, we give  interpretations of  the  asid conditions. 
We only discuss the right asid condition,   since the left version follows from the dual argument.

For the purpose, we collect natural isomorphisms which relate to 
the morphism $\lambda_{r}: \Lambda \to \RHom_{\Lambda}(C,C)$.

%Let $\cT = -\lotimes C: \sfD^{-} (\Mod \Lambda) \to \sfD^{-}(\Mod \Lambda)$ be the tensor functor.
We will show up natural isomorphisms which are also used in the proofs  
of Theorem \ref{Adaching theorem 0}  and Theorem \ref{Adaching theorem 1} 
so that the readers can easily consult their definitions. 

\vspace{4pt}
\noindent 
\underline{ $\bullet$ The natural isomorphism  $\scrT$:}  
We denote by 
$\scrT_{M,N} : \RHom_{\Lambda}(M,N) \to \RHom_{\Lambda}(M \lotimes_{\Lambda} C , N\lotimes_{\Lambda} C)$ 
the morphism induced 
by the functor $\cT = -\lotimes_{\Lambda} C$.

\vspace{4pt}
\noindent 
\underline{ $\bullet$ The natural isomorphism  $\scrE$:}  
Let $D$ be a $\Lambda$-$\Lambda$-bimodule and $M \in \Mod \Lambda$. 
We define a morphism  $\tilde{\scrE}_{D, M}:\Hom_{\Lambda}(D,M)\otimes_{\Lambda} D \to M$ 
by the formula 
$ \tilde{\scrE}_{D,M}(f\otimes d) = f(d)$ 
for $f\in \Hom_{\Lambda}(D,M) $ and $d \in D$. 
We note that $\tilde{\scrE}_{D, M} $ is natural in $D$ and $M$ 
and that 
it is the counit morphism of the adjoint pair 
$-\otimes_{\Lambda} D: \Mod \Lambda \rightleftarrows \Mod \Lambda : \Hom_{\Lambda}(D, -)$.

Let $D$ be a complex of $\Lambda$-$\Lambda$-bimodules and $M \in \sfD(\Mod \Lambda)$. 
We denote by $\scrE_{D, M}:\RHom_{\Lambda}(D,M)\lotimes_{\Lambda} D \to M$ 
the morphism which is derived  from the morphism $\tilde{\scrE}$. 
We note that $\scrE_{D, M}$ is 
the counit morphism 
of the adjoint pair $-\lotimes_{\Lambda}D : \sfD(\Mod \Lambda) \rightleftarrows \sfD(\Mod \Lambda): \RHom_{\Lambda}(D, -)$.

\vspace{4pt} 
\noindent 
\underline{ $\bullet$ The natural isomorphisms  $\epsilon_{r}, \epsilon_{\ell}$:} 
We set $\epsilon_{r}:= \scrE_{C, \Lambda}$ and $\epsilon_{\ell}$ to be the left version of $\epsilon_{r}$. 
\[
\epsilon_{r}: C^{\rightvee} \lotimes_{\Lambda} C = \RHom_{\Lambda}(C, \Lambda) \lotimes_{\Lambda} C  \to \Lambda, \ \ \ 
\epsilon_{\ell}: C \lotimes_{\Lambda} C^{\leftvee}  =C\lotimes_{\Lambda} \RHom_{\Lambda^{\op}}(C, \Lambda) \to \Lambda.
\]

\begin{proposition}\label{isomorphisms proposition}
For a natural number $a\in \NN$,  the following conditions  are equivalent. 

\begin{enumerate}[(1)]
\item The morphism $\RHom_{\Lambda}(C^{a}, \lambda_{r}) $ is an isomorphism. 

\item The morphism $\scrT_{C^{a} ,\Lambda} : \RHom_{\Lambda}(C^{a} , \Lambda) \to \RHom_{\Lambda}(C^{a+1}, C)$ 
is an isomorphism. 

\item  The morphism 
$\cT_{ M \lotimes C^{a}, N}: 
\Hom_{\Lambda}(M \lotimes_{\Lambda} C^{a}, N) \to 
\Hom_{\Lambda}(M \lotimes_{\Lambda} C^{a+1}, N \lotimes_{\Lambda} C)$ 
is an isomorphism  for $M ,N \in \sfK^{\mrb}(\proj \Lambda)$ 
where $\cT_{M \lotimes C^{a}, N}$ is the morphism of $\Hom$-spaces 
associated to the functor $\cT = - \lotimes_{\Lambda} C$.

\item 
$C^{a} \lotimes \epsilon_{\ell}$ is an isomorphism. 
\end{enumerate}
\end{proposition}

%We record an immediate consequence of Proposition \ref{isomorphisms proposition} for  the later quotation. 
%\begin{corollary}\label{IG:asid corollary}If a bimodule 
%$C$ satisfies the condition  right $\asid$ 3, then 
%\[\begin{split}\alpha_{r}  &= \min\{a \mid C^{a} \lotimes \epsilon_{\ell} \textup{ is an isomorphism} \} \\ 
%%                    & = \min\{a \mid  \cT_{M \lotimes C^{a} , N }  \textup{ is an isomorphism for } M ,N \in \sfK^{\mrb}(\proj \Lambda) \}.
%\end{split}\] \end{corollary}

Before starting the proof, we need to introduce a natural isomorphism $\scrF$. 

\vspace{4pt}
\noindent 
\underline{ $\bullet$ The natural isomorphism  $\scrF$:}  
For $\Lambda$-modules $M, N$ and $\Lambda$-$\Lambda$-bimodules $D, E$, 
we define a morphism $\tilde{\scrF}_{M, N , D, E}$ to be 
\[
\begin{split}
&\tilde{\scrF}_{M ,N , D, E} :  
N \otimes_{\Lambda} \Hom_{\Lambda}(D, E) \otimes_{\Lambda} \Hom_{\Lambda}(M, \Lambda) 
\to 
\Hom_{\Lambda}(M \otimes_{\Lambda} D, N \otimes_{\Lambda} E), \\
&\tilde{\scrF}_{M, N , D, E}( n\otimes f\otimes \phi )(m \otimes d)  
:= n \otimes f( \phi(m) d)
\end{split}
\]
for $n \in N, f \in \Hom_{\Lambda}(D, E), \phi \in \Hom_{\Lambda}(M , \Lambda), m \in M , d \in D$. 
For $M ,N \in \sfD(\Mod \Lambda)$ and complexes $D, E$ of $\Lambda$-$\Lambda$-bimodules, 
we define $\scrF_{M , N , D ,E} $ to be the morphism derived from $\tilde{\scrF}$. 
\[
\scrF_{M, N , D, E}: 
N \lotimes_{\Lambda} \RHom_{\Lambda}(D, E) \lotimes_{\Lambda} M^{\rightvee} 
\to 
\RHom_{\Lambda}(M \lotimes_{\Lambda} D, N \lotimes_{\Lambda} E). 
\]
We note that if $M ,N $ belong to $\sfK^{\mrb}(\proj \Lambda)$, 
then $\scrF_{M,N, D, E}$ is an isomorphism.

In the proof and thereafter, we (tacitly) use the following natural isomorphisms. 

\noindent 
\underline{ $\bullet$ The natural isomorphism  $\scrG$:}  
By $\scrG_{M}: M \to M^{\rightvee \leftvee}$  
we denote the unit map of the adjoint pair $(-)^{\rightvee} \dashv (-)^{\leftvee}$. 
We note that it is an isomorphism for $M \in \sfK^{\mrb}(\proj \Lambda)$.

\begin{proof}[Proof  of Proposition \ref{isomorphisms proposition}]
Let 
$\mathsf{adj}_{C^{a} ,C}: \RHom_{\Lambda}(C^{a}, \RHom_{\Lambda}(C,C)) 
\xrightarrow{\cong} \RHom_{\Lambda}(C^{a +1},C)$ 
be the isomorphism induced from the adjoint pair 
$- \lotimes_{\Lambda}C \dashv \RHom_{\Lambda}(C, -)$. 
Then, we can check the equation 
$\scrT_{C^{a},\Lambda} = \mathsf{adj}_{C^{a}, C} \circ \RHom_{\Lambda}(C^{a}, \lambda_{r})$, 
which implies the equivalence (1) $\Leftrightarrow$ (2).  
\[
\begin{xymatrix}{ 
\RHom_{\Lambda}(C^{a}, \Lambda) \ar[rr]^{\RHom(C^{a}, \lambda_{r}) \qquad \ \ } \ar[drr]_{\scrT_{C^{a}, \Lambda}} && 
\RHom_{\Lambda}(C^{a}, \RHom_{\Lambda}(C,C))  \ar[d]^{\mathsf{adj}_{C^{a}, C}}_{\cong} \\
&& 
\RHom_{\Lambda}(C^{a+1}, C)
}\end{xymatrix}
\]

We  prove the implication (3) $\Rightarrow$ (2). 
First observe that 
the condition (3) is satisfied if and only if 
(3') the morphism 
$\scrT_{M \lotimes C^{a} ,N } : 
\RHom_{\Lambda}(M \lotimes_{\Lambda} C^{a} , N) \to 
\RHom_{\Lambda}(M \lotimes_{\Lambda} C^{a+1}, N \lotimes_{\Lambda} C)$ 
is an isomorphism  for $M ,N \in \sfK^{\mrb}(\proj \Lambda)$.  
If we substitute both $M$ and $N$ with $\Lambda$ in the condition (3'), 
then we see that (3') implies (2).

We  prove the implication (2) $\Rightarrow$ (3). 
We  can check that   the following diagram is commutative. 
\[
\begin{xymatrix}{ 
N \lotimes_{\Lambda} \RHom_{\Lambda}(C^{a}, \Lambda ) \lotimes_{\Lambda} M^{\rightvee} 
\ar[d]_{\scrF_{M, N , C^{a}, \Lambda} }^{\cong} 
\ar[rrr]^{ N \lotimes \scrT_{C^{a}, \Lambda} \lotimes M^{\rightvee}} &&& 
N \lotimes_{\Lambda} \RHom_{\Lambda}(C^{a+1}, C) \lotimes_{\Lambda} M^{\rightvee}
\ar[d]^{\scrF_{M, N , C^{a + 1}, C}}_{\cong} \\
\RHom_{\Lambda}(M \lotimes_{\Lambda} C^{a} , N) 
\ar[rrr]_{ \scrT_{M \lotimes C^{a} ,N} } &&&
\RHom_{\Lambda}(M \lotimes_{\Lambda} C^{a+1} , N \lotimes_{\Lambda} C). 
}\end{xymatrix}
\]
Since  we are assuming that $M, N$ belongs to $\sfK^{\mrb}(\proj \Lambda)$, 
the vertical arrows are isomorphism. 
This shows that the condition (2) implies the condition (3') 
and hence the condition (3).

To prove the equivalence  (1) $\Leftrightarrow$ (4), 
we need to introduce one more natural morphism,  which is denoted by $\scrH$.   
First for   $\Lambda$-$\Lambda$-bimodules $D, E$, 
we define a morphism $\tilde{\scrH}_{D, E}$ to be 
\[
\begin{split}
&\tilde{\scrH}_{D, E} : D \otimes_{\Lambda}E \otimes_{\Lambda} \Hom_{\Lambda^{\op}}(E, \Lambda) \to 
\Hom_{\Lambda^{\op}}( \Hom_{\Lambda} (D, \Hom_{\Lambda}(E, E)) , \Lambda), \\
&\tilde{\scrH}_{D, E}(d\otimes e \otimes f)(\phi) := f(\phi(d)(e)) 
\end{split}
\] 
for $d \in D, \ e \in E, \ f \in \Hom_{\Lambda^{\op}}(E, \Lambda), \phi \in \Hom_{\Lambda} (D, \Hom_{\Lambda}(E,E))$. 
For   complexes $D, E$ of $\Lambda$-$\Lambda$-bimodules, 
we define a morphism $\scrH_{D, E}$ to be the morphism derived from $\tilde{\scrH}$. 
\[
\scrH_{D, E}: D \lotimes_{\Lambda} E \lotimes_{\Lambda} E^{\leftvee} 
\to \RHom_{\Lambda}( D, \RHom_{\Lambda}(E,E) )^{\leftvee}.
\]
We can check that 
the following diagram is commutative  and that 
the left vertical arrow $\scrH_{C^{a}, C}$ is an isomorphism 
\[
\begin{xymatrix}{
C^{a} \lotimes_{\Lambda} C \lotimes_{\Lambda} C^{\leftvee} \ar[rrr]^{C^{a} \lotimes \epsilon_{\ell} } 
\ar[d]^{\cong}_{\scrH_{C^{a}, C}} 
&&&
C^{a} \lotimes_{\Lambda} \Lambda \ar[d]_{\cong}^{\scrG_{C^{a}}} \\ 
\RHom_{\Lambda}(C^{a}, \RHom_{\Lambda}(C,C))^{\leftvee} \ar[rrr]_{\ \ \ \RHom(C^{a}, \lambda_{r})^{\leftvee}}  &&&
\RHom_{\Lambda}(C^{a}, \Lambda)^{\leftvee}  
}\end{xymatrix}
\]
where the right vertical arrow is a canonical isomorphism. 
Now it is clear that (1) $\Leftrightarrow$ (4).  
\end{proof}

\begin{remark}
The equivalence $(1) \Leftrightarrow (2) \Leftrightarrow (3)$ of Proposition \ref{isomorphisms proposition} 
is true without the assumption that $C \in \sfK^{\mrb}(\proj \Lambda)$ and $C \in \sfK^{\mrb}(\proj \Lambda^{\op})$.
\end{remark}

\subsubsection{The thick hull $\thick C^{a}$}

We  remind the following fact which will be frequently used in the sequel. 
Let $M, N$ be objects of $\sfK^{\mrb}(\proj \Lambda)$ and $D$  a complex of $\Lambda$-$\Lambda$-bimodules. 
Then $\thick N \subset \thick M$ if and only if $N \in \thick M$. 
Moreover if $N \in \thick M$, then $ N \lotimes_{\Lambda} D \in \thick M\lotimes_{\Lambda} D$. 

By the assumption $C_{\Lambda}$ belongs to  $ \thick \Lambda = \sfK^{\mrb}(\proj \Lambda)$. 
Thus taking $M = \Lambda, N=C, D= C^{a}$ in the above consideration, 
we see that  $\thick C^{a+1} \subset \thick C^{a}$ 
and obtain the following descending chain of thick subcategories 
\begin{equation}\label{thick C sequence}
\sfK^{\mrb}(\proj \Lambda) \supset \thick C \supset \thick C^{2} \supset \cdots \supset \thick C^{a} \supset \cdots .
\end{equation}
We show that if the  left and right asid conditions  are satisfied, 
then this chain terminates. 
But actually, we prove more properties of these and related  thick subcategories in the following lemma.

\begin{lemma}\label{T lemma} 
We assume that $C$ satisfies  
the  left and right asid conditions. 
Let $\alpha_{r}$ and $\alpha_{\ell}$ be  right and left asid number. 
Then we have the following equalities in $\sfK^{\mrb}(\proj \Lambda)$. 
 
\begin{enumerate}[(1)]

\item 
$\thick C^{a} = \thick C^{a +1} $ for $a \geq \alpha_{\ell}$.

\item 
$\thick C^{a} = \thick (C^{a})^{\leftvee}$ for $a \geq \max\{ \alpha_{r}, \alpha_{\ell} \}$.

\item 
$\thick (C^{a})^{\leftvee} = \thick (C^{a+1})^{\leftvee}$  for $a \geq \alpha_{r}$. 

\item 
$\Ker(- \lotimes_{\Lambda}C^{a})|_{\sfK} = \Ker(- \lotimes_{\Lambda}C^{a+1})|_{\sfK}$ for $a \geq \alpha_{r}$. 

\end{enumerate}
\end{lemma}

To prove this, we need to verify a compatibility between 
the iterated derived tensor products $C^{a}$ and the $\Lambda$-duality. 
For this, we introduce the following natural morphism. 

\vspace{4pt}
\noindent 
\underline{ $\bullet$ The natural isomorphism  $\scrI$:}  
Let $D$ be a complex of $\Lambda$-$\Lambda$-bimodules and $M \in \sfD(\Mod \Lambda)$. 
For simplicity, we set $\scrI_{D,M} := \scrF_{M, \Lambda, D, \Lambda}$ 
and regard this as a morphism 
$\scrI_{D,M}: D^{\rightvee} \lotimes_{\Lambda} M^{\rightvee} \to (M\lotimes_{\Lambda} D)^{\rightvee}$ 
via canonical isomorphisms as below 
\[
\scrI_{D,M} : 
D^{\rightvee} \lotimes_{\Lambda} M^{\rightvee}\cong \Lambda \lotimes_{\Lambda} \RHom_{\Lambda}(D, \Lambda) \lotimes_{\Lambda}M^{\rightvee} 
\xrightarrow{\scrF_{M, \Lambda, D, \Lambda}}
\RHom_{\Lambda}(M\lotimes_{\Lambda}D, \Lambda \lotimes_{\Lambda}\Lambda ) \cong  
 (M\lotimes_{\Lambda} D)^{\rightvee}
\]

We leave the proof of the following lemma to the readers.

\begin{lemma}\label{natural morphism 2}
Let $D_{1}, D_{2}$ be $\Lambda$-$\Lambda$-bimodule complexes which are 
perfect as left and right $\Lambda$-complexes.
Then, 
 We have the following commutative diagram 
\[
\begin{xymatrix}@C=30pt{
D_{2}^{\rightvee} \lotimes_{\Lambda} D_{1}^{\rightvee} \lotimes_{\Lambda} D_{1} \lotimes_{\Lambda} D_{2}  
\ar[rr]^{ \qquad{ (D_{2})^{\rightvee} \lotimes \scrE^{(1)} \lotimes  D_{2} }}  
\ar[d]_{\scrI_{D_2, D_{1}} \lotimes D_{1} \lotimes D_{2}}&&
D_{2}^{\rightvee} \lotimes_{\Lambda} \Lambda \lotimes_{\Lambda} D_{2}  
\ar[r]^{\cong} 
& D_{2}^{\rightvee}  \lotimes_{\Lambda} D_{2}  \ar[d]^{\scrE^{(2)}} 
\\ 
(D_{1} \lotimes_{\Lambda} D_{2})^{\rightvee} \lotimes_{\Lambda} D_{1}\lotimes_{\Lambda} D_{2} \ar[rrr]_{\scrE^{(12)}} 
&&&
\Lambda
}\end{xymatrix}\] 
where 
$\scrE^{(i)} =\scrE_{D_{i}, \Lambda}$ for  $i = 1,2$ and 
$\scrE^{(12)} = \scrE_{D_{1} \lotimes D_{2}, \Lambda}$.
\end{lemma}

The thick hull of an object $L$ of $\sfK^{\mrb}(\proj \Lambda^{\op})$ 
is denoted by $\thick_{\Lambda^{\op}} L $.

\begin{proof}[Proof of Lemma \ref{T lemma}] 
(1) By the left version of  Proposition \ref{isomorphisms proposition}, 
we have an isomorphism $C^{a} \cong C^{\rightvee} \lotimes_{\Lambda} C^{a + 1}$ 
for $a \geq \alpha_{\ell}$.

For an object $M \in \sfD(\Mod \Lambda)$, 
we denote by $\Loc M$ the localizing subcategory generated by $M$, 
i.e., the smallest triangulated subcategory of $\sfD(\Mod \Lambda)$ containing $M$ 
which is closed under taking arbitrarily coproducts. 
Since $\sfD(\Mod \Lambda)= \Loc \Lambda$, we see that 
$C^{a} \cong C^{\rightvee} \lotimes_{\Lambda}C^{a + 1}$ belongs to $\Loc C^{a +1}$. 
Therefore $C^{a}$ belongs to $\sfK^{\mrb}(\proj \Lambda) \cap \Loc C^{a +1}$.

For a triangulated category $\sfD$, we denote by $\sfD^{\textup{cpt}}$ the full subcategory of compact objects. 
It is well-known that $\sfD(\Mod \Lambda)^{\textup{cpt}} = \sfK^{\mrb}(\proj \Lambda)$.  
Therefore by \cite[Theorem 2.1, Lemma 2.2]{Neeman}, we have
\[
\sfK^{\mrb}(\proj \Lambda) \cap \Loc C^{a +1} = (\Loc C^{a + 1})^{\textup{cpt}} = \thick C^{a +1}. 
\]
Thus, we conclude $C^{a} \in \thick C^{a +1}$ as desired.

(2) 
Let $a \geq \alpha_{\ell}$. 
Since $C^{a} \cong C^{\rightvee}\lotimes_{\Lambda} C^{a+1} \cong  (C^{\rightvee})^{b} \lotimes_{\Lambda}C^{a+b}
\cong (C^{b})^{\rightvee} \lotimes_{\Lambda} C^{a+b}$ 
for $b \geq 1$, we have  ${}_{\Lambda}C^{a} \in \thick_{\Lambda^{\op}} 
(C^{b})^{\rightvee} \subset \sfD^{\mrb}(\mod \Lambda^{\op})$.  
Thus $(C^{a})^{\leftvee} \in \thick C^{b}$ for $b \geq 1$. 

In the same way,  for $a \geq \alpha_{r}$ 
we deduce 
 $C^{a} \in \thick (C^{b})^{\leftvee}$ for $b \geq 1$. 
 Hence we have $\thick C^{a} = \thick (C^{a})^{\leftvee}$.
 
 (3) and (4) follows from the left version of   (1). 
 \end{proof}

By Lemma \ref{T lemma}.(4),  
if $C$ satisfies the right and the left asid conditions, 
then the increasing sequence \eqref{kernel sequence} terminates at 
$a = \max\{ \alpha_{r},  \alpha_{\ell}\}$.  
From  Corollary \ref{adasore:description of kernel}, 
we deduce the following description of $\Ker \varpi$.  

\begin{corollary}\label{T corollary}
Assume that $C$ satisfies the right and the left asid conditions.
Then we have 
\[\Ker \varpi = \Ker( -\lotimes_{\Lambda} C^{a} )|_{\sfK} \] 
for $a \geq \max\{ \alpha_{r}, \alpha_{\ell}\}$.  
\end{corollary}

By Lemma \ref{T lemma}.(1),  
if $C$ satisfies the right and the left asid conditions, 
then the decreasing sequence \eqref{thick C sequence} terminates at 
$a =  \alpha_{\ell}$. 
In the next lemma which is a key for the main theorems  
  we study the subcategory $\thick C^{\max\{\alpha_{r}, \alpha_{\ell}\}}$.

\begin{lemma}\label{T proposition}
Assume that $C$ satisfies the right and the left asid conditions. 
We set $\alpha = \max\{\alpha_{r}, \alpha_{\ell}\}$,  $\sfT := \thick C^{\alpha}$.  
Then the following assertions hold. 

\begin{enumerate}[(1)]
\item The functor $\cT :=- \lotimes_{\Lambda} C$ acts $\sfT$ as an equivalence. 
%i.e., $\cT(\sfT) \subset \sfT$ and the restriction functor $\cT|_{\sfT} : \sfT \to \sfT$ is an equivalence.  

\item We have a semi-orthogonal decomposition 
\[
\sfK^{\mrb}(\proj  \Lambda) = \sfT \perp \Ker \varpi.
\]\end{enumerate}
\end{lemma}

In the proof and thereafter, we (tacitly) use the following natural isomorphisms. 

\noindent 
\underline{ $\bullet$ The natural isomorphism  $\scrJ$:} 
Let $M, N$ be complex of $\Lambda$-modules. 
For simplicity, 
we set $\scrJ_{M, N} := \scrF_{M, N, \Lambda, \Lambda}$ 
and regard this as a morphism 
$\scrJ_{M, N}: N \lotimes_{\Lambda} M^{\rightvee} \to \RHom_{\Lambda}(M, N)$. 
We note that $\scrJ_{M, N}$ is an isomorphism 
if one of $M, N$ belongs to $\sfK^{\mrb}(\proj \Lambda)$.

\begin{proof}
(1)  It is clear that $\cT(\sfT) \subset \sfT$ by Lemma \ref{T lemma}. (1).  

 Applying Proposition \ref{isomorphisms proposition}.3,  
we see that   
the morphism $\cT_{C^{\alpha}, C^{\alpha}[n]}$ is an isomorphism for $n \in \ZZ$.  
\[
\cT_{C^{\alpha}, C^{\alpha}[n]}: \Hom_{\Lambda}(C^{\alpha}, C^{\alpha}[n]) 
\to \Hom_{\Lambda}(C^{\alpha + 1}, C^{\alpha + 1} [n])
\] 
Since $\sfT = \thick C^{\alpha}$, 
we conclude that $\cT|_{\sfT} $ is fully faithful by standard argument. 

Since $C^{\alpha +1} \in \image \cT|_{\sfT} $ 
and $\cT|_{\sfT}$ is fully faithful, 
we see that $\thick C^{ \alpha + 1} \subset \image \cT|_{\sfT}$ by standard argument. 
Thus by Lemma \ref{T lemma},  we conclude that $\cT|_{\sfT}$ is essentially surjective. 

(2)
First we claim that 
 $\Hom_{\Lambda}(X,Y)= 0$ for $X \in \sfT, Y \in \Ker \varpi$. 
Indeed, 
since  $\sfT = \thick (C^{\alpha})^{\leftvee}$, 
it is enough to check the case $X = (C^{\alpha})^{\leftvee}$. 
It follows from $\Ker \varpi = \Ker(-\lotimes_{\Lambda} C^{\alpha})$
that 
\[
\RHom_{\Lambda}((C^{\alpha})^{\leftvee},Y) \cong 
Y \lotimes_{\Lambda} (C^{\alpha})^{\leftvee\rightvee} \cong Y \lotimes_{\Lambda} C^{\alpha}= 0. 
\]

Now it is enough to prove that 
 every $M \in \sfK^{\mrb}(\proj \Lambda)$ fits into an exact triangle $X \to M \to Y \to $ 
with $X \in \sfT, Y \in \Ker \varpi$. 
Since $\RHom_{\Lambda}(C^{\alpha},M)\lotimes_{\Lambda}C^{\alpha}$ belongs to $\sfT = \thick C^{\alpha}$ 
it is enough to show  that the cone of the derived evaluation map 
$\scrE_{C^{\alpha}, M} :\RHom_{\Lambda}(C^{\alpha},M)\lotimes_{\Lambda}C^{\alpha} \to M$ 
belongs to $\Ker(-\lotimes_{\Lambda} C^{\alpha})$. 
In other words, 
if we set $\scrE_{M}' := \scrE_{C^{\alpha}, M}$,   
we only have to show that $\scrE'_{M}\lotimes_{\Lambda}C^{\alpha}$ is an isomorphism. 
 
First observe that 
we have the following commutative diagram where the bottom arrow is the canonical isomorphism.
\[
\begin{xymatrix}{
M \lotimes_{\Lambda} (C^{\alpha})^{\rightvee} \lotimes_{\Lambda} C^{\alpha} 
\ar[d]_{M \lotimes \scrE'_{\Lambda} }  \ar[rr]^{\scrJ \lotimes C^{\alpha} }  && 
\RHom_{\Lambda} (C^{\alpha}, M) \lotimes_{\Lambda} C^{\alpha} 
\ar[d]^{\scrE'_{M} } \\ 
M \lotimes_{\Lambda} \Lambda \ar[rr]^{\cong} && M
}\end{xymatrix}
\]
Therefore the problem is reduced to the case where $M = \Lambda$. 

Let $\scrI': (C^{\rightvee})^{\alpha} \to  (C^{\alpha})^{\rightvee}$ 
be the isomorphism 
which is obtained from the morphism $\scrI$.  
For simplicity we set $\epsilon_{r}^{(a)} := (C^{\rightvee})^{a}\lotimes \epsilon_{r} \lotimes C^{a}$. 
We identify $(C^{\rightvee})^{a-1}\lotimes_{\Lambda} \Lambda \lotimes_{\Lambda} C^{a-1}$ 
with $(C^{\rightvee})^{a-1}\lotimes_{\Lambda} C^{a-1}$ via the canonical isomorphism 
and regard $\epsilon_{r}^{(a)}$ as the morphism 
from $(C^{\rightvee})^{a}\lotimes_{\Lambda} C^{a}$ to $(C^{\rightvee})^{a-1}\lotimes_{\Lambda} C^{a-1}$.
Then by Lemma \ref{natural morphism 2}, 
we have $\scrE'_{\Lambda} \circ (\scrI' \lotimes C^{\alpha}) = \epsilon_{r} \circ \epsilon_{r}^{(1)} \circ \cdots \circ \epsilon_{r}^{(\alpha -1)}$. 
\[
(C^{\rightvee})^{\alpha} \lotimes_{\Lambda} C^{\alpha} 
\xrightarrow{ \epsilon_{r}^{(\alpha-1)}}
(C^{\rightvee})^{\alpha -1} \lotimes_{\Lambda} C^{\alpha -1 } 
\xrightarrow{ \epsilon_{r}^{(\alpha-2)}}
\cdots 
\to
C^{\rightvee}\lotimes_{\Lambda} C
\xrightarrow{ \epsilon_{r}} 
\Lambda
\] 
Since $\epsilon_{r} \lotimes C^{\alpha}$ is an isomorphism by the left version of Proposition \ref{isomorphisms proposition}, 
we conclude that $\scrE'_{\Lambda} \lotimes C^{\alpha}$ is an isomorphism as desired. 
\end{proof}

In the proof of Proposition \ref{T proposition}, we obtained the following result. 

\begin{corollary}\label{right adjoint corollary}
The functor $\tau: 
\sfK^{\mrb}(\proj \Lambda) \to \sfT, \  \tau(M) := \RHom_{\Lambda}(C^{\alpha}, M) \lotimes_{\Lambda} C^{\alpha}$ 
is 
a right adjoint functor of the inclusion functor $\mathsf{in}: \sfT \to \sfK^{\mrb}(\proj \Lambda)$. 
\end{corollary}

\subsubsection{Lemmas}

The following abstract lemma clarify the situation.

\begin{lemma}\label{abstraction 1}
Let $\sfD$ be a triangulated subcategory, 
$\sfE \subset \sfD$ an admissible subcategory 
and $G:\sfD\rightleftarrows  \sfD: F$ an adjoint pair of  exact endofunctors. 
Assume that 
$F$ acts on $\sfE$ as an equivalence and nilpotently acts on ${\sfE}^{\perp}$.   
Then for a natural number $a\geq 0 $, the following conditions  are equivalent.  
 
\begin{enumerate}[(1)] 
\item $F^{a}(\sfD) \subset \sfE$. 

\item $F^{a}(\sfE^{\perp}) = 0$.

\item 
The morphism $F_{F^{a}(d),d'}: \Hom_{\sfD}(F^{a}(d),d') \to \Hom_{\sfD}(F^{a+1}(d),F(d'))$ is an isomorphism 
for $d,d' \in \sfD$.

\item  $G^{a}(\sfD) \subset \sfE$. 

\item $G^{a}({}^{\perp}\sfE) = 0$.

\item 
The morphism $G_{ d, G^{a}(d')}: \Hom_{\sfD}(d, G^{a}(d')) \to \Hom_{\sfD}(G(d) , G^{a+1}(d'))$ is an isomorphism 
for $d,d' \in \sfD$.
 \end{enumerate}
\end{lemma}

\begin{proof}

%(1)$\Rightarrow$(2). $F^{a}(\sfE^{\perp}) \subset \sfE \cap \sfE^{\perp} = 0$. 

(2) $\Rightarrow$ (1). 
Let $d$ be an object of $\sfD$ and $e' \to d \to k\to $ an exact triangle such that $e' \in \sfE, k\in \sfE^{\perp}$. 
Since $F^{a}(k) = 0$, we have  $F^{a}(d) \cong F^{a}(e')$.  

(3) $\Rightarrow$ (2). 
Let $k \in \sfE^{\perp}$. 
Then 
the morphism \[
F^{b}_{F^{a}(k),F^{a}(k)}: \Hom_{\sfD}(F^{a}(k), F^{a}(k)) \to \Hom_{\sfD}(F^{a + b}(k),F^{a + b}(k))\] 
is an isomorphism for $b \geq 0$. 
Since $F^{a + b}(k)=0$ for $b \gg 0$, we conclude $F^{a}(k) = 0$. 

For 
the implication 
(1)$\Rightarrow$(3), it is enough to show that  
 the map $F_{e,d}:\Hom_{\sfD}(e,d) \to \Hom_{\sfD}(F(e),F(d))$ is an isomorphism 
for $e \in \sfE, d \in \sfD$.  

Let $e' \xrightarrow{\varphi} d \to k\to $ be the exact triangle such that $e' \in \sfE, k\in \sfE^{\perp}$.  
Since $\Hom_{\sfD}(e,k[n]) = 0$ for $n =0 ,-1$, 
the induced map 
$\varphi_{*}: \Hom_{\sfD}(e,e') \to \Hom_{\sfD}(e,d)$ is an isomorphism.
By the assumptions, we have $F(\sfE) \subset \sfE$ and $F(\sfE^{\perp}) \subset \sfE^{\perp}$. 
Thus, $\Hom_{\sfD}(F(e),F(k)[n]) = 0$ for $n =0 ,-1$. 
Hence,  the induced map  $ F(\varphi)_{*}: \Hom_{\sfD}(F(e),F(e')) \to \Hom_{\sfD}(F(e),F(d))$ 
is an isomorphism.
Since $F_{e,d}\circ \varphi_{*} = F(\varphi)_{*} \circ F_{e,e'}$ and $F_{e,e'}$ is an isomorphism, we conclude that $F_{e,d}$ is an isomorphism. 

We prove  (1) $\Rightarrow$ (5). 
Let  $\ell \in {}^{\perp}\sfE, d\in \sfD$. 
We have  
 $\Hom_{\sfD}(G^{a}(\ell),d) = \Hom_{\sfD}(\ell, F^{a}(d)) =0$. 
Therefore $G^{a}(\ell) = 0$.

We have proved the equivalence  (1) $\Leftrightarrow$ (2) $\Leftrightarrow $ (3) 
and the implication (1) $\Rightarrow$ (5).

Next we claim that

\begin{claim}\label{abstraction claim}
\begin{enumerate}[(a)] 

\item  $G$ acts $\sfE$ as an equivalence. 

\item $G$ nilpotently acts on ${}^{\perp}\sfE$. 
\end{enumerate}
\end{claim} 

\begin{proof}[Proof of Claim]
(a) 
For $e \in \sfE $, we have $\Hom_{\sfD}(G(e),k ) = \Hom_{\sfD}(e,F(k)) = 0$ for any $k \in \sfE^{\perp}$. 
Hence $ G(e) $ belongs to $ {}^{\perp}(\sfE^{\perp}) = \sfE$. 
This shows that $G(\sfE) \subset \sfE$. 
Now $G|_{\sfE}$ can be  regarded as an endofunctor of $\sfE$. 
Then it is a left adjoint of the equivalence $F|_{\sfE}$. 
Hence  $G|_{\sfE} = (F|_{\sfE})^{-1}$ and in particular $G|_{\sfE}$ is an equivalence. 

(b)
We claim that $G({}^{\perp}\sfE) \subset {}^{\perp} \sfE $. 
Indeed, for $\ell \in {}^{\perp}\sfE, t \in \sfE$, 
we have the equality $\Hom_{\sfD}(G(\ell),e) = \Hom_{\sfD}(\ell, F(e)) = 0$. 
Thus $G(\ell)  \in {}^{\perp}\sfE$. 

By the assumption $F$ nilpotently acts on $\sfE^{\perp}$. 
Therefore by the implication (2) $\Rightarrow$ (5), 
we check that $G^{b}({}^{\perp}\sfE) = 0$ for $b \gg 0$. 
\end{proof}

Since $G$ acts $\sfE$ as an equivalence and nilpotently acts on ${}^{\perp } \sfE$, 
we can apply dual argument to $G$ to prove the equivalences
 (4) $\Leftrightarrow$ (5) $\Leftrightarrow$ (6) and the implication (4) $\Rightarrow$ (2).
\end{proof}

To apply above lemma to the proofs of 
Theorem \ref{Adaching theorem 0} and Theorem \ref{Adaching theorem 1}, 
we need to show that the functor $- \lotimes_{\Lambda} C$ has 
a left adjoint functor.

%mmmmmm
\begin{lemma}\label{adjoint lemma}
We have the following  adjoint pair 
\[
%\begin{split}
-\lotimes_{\Lambda} C^{\leftvee} : %&
\sfK^{\mrb}(\proj \Lambda)
 \rightleftarrows \sfK^{\mrb}(\proj \Lambda) : -\lotimes_{\Lambda} C. 
%, \\C^{\rightvee} \lotimes_{\Lambda}-  : &\sfK^{\mrb}(\proj \Lambda^{\op})\rightleftarrows \sfK^{\mrb}(\proj \Lambda^{\op}) 
%: C\lotimes_{\Lambda}-.\end{split}
\]
\end{lemma} 

\begin{proof}
By the assumption that the complex $C$ is perfect as left modules and right modules, 
we see that the adjoint pair $- \lotimes C^{\leftvee} \dashv \RHom_{\Lambda}(C^{\leftvee}, -)$ 
of endofunctors on $\sfD(\Mod \Lambda)$ 
can be restricted $\sfK^{\mrb}(\proj \Lambda)$. 
\[
-\lotimes_{\Lambda} C^{\leftvee} : 
\sfK^{\mrb}(\proj \Lambda) \rightleftarrows \sfK^{\mrb}(\proj \Lambda) : \RHom_{\Lambda}(C^{\leftvee},-).
\]
On the other hands
by the assumption, we have isomorphism of functors below induced by the natural morphisms $\scrJ$ and the left version $\scrG_{\ell}$
 of $\scrG$
\[
-\lotimes C \xrightarrow{\ \cong \scrG_{\ell}} 
-\lotimes C^{\leftvee\rightvee} 
\xrightarrow{\ \cong\scrJ} 
\RHom_{\Lambda}(C^{\leftvee}, -).
\]
Combining these observations, we obtain the desired adjoint pair. 
\end{proof}

\subsubsection{Proof of Theorem \ref{Adaching theorem 0}}

We proceed a proof of Theorem \ref{Adaching theorem 0}.

\begin{proof}[Proof   of Theorem \ref{Adaching theorem 0}]
We prove the implication (1) $\Rightarrow$ (2). 
We set $\sfT = \thick C^{\max\{\alpha_{r}, \alpha_{\ell} \}}$.  
By Lemma \ref{T proposition}, 
the functor $\cT$ acts on $\sfT$ as an equivalence and 
the subcategory $\sfT$ is right admissible. 
It follows from Lemma \ref{T proposition} and Corollary \ref{T corollary} that 
the functor $\cT$ nilpotently acts on $\sfT^{\perp}$.

By the  left version of Lemma \ref{T proposition}, 
the subcategory $\thick_{\Lambda^{\op}}C^{\alpha}$ is  a  right admissible subcategory of $\sfK^{\mrb}(\proj \Lambda^{\op})$. 
Since the $\Lambda$-dual functor  $(-)^{\leftvee}: \sfK^{\mrb}(\proj \Lambda^{\op}) 
\xrightarrow{\sim} \sfK^{\mrb}(\proj \Lambda)$ 
gives a contravariant equivalence, 
the subcategory $\sfT = \thick(C^{\alpha})^{\leftvee} = (\thick_{\Lambda^{\op}} C^{\alpha})^{\leftvee}$ is left admissible. 
This finishes the proof (1) $\Rightarrow$ (2).

We prove the implication (2) $\Rightarrow$ (3). 
First note that by  the assumption, Proposition \ref{isomorphisms proposition} and Lemma \ref{adjoint lemma}, 
we can apply Lemma \ref{abstraction 1} and Claim \ref{abstraction claim} 
to the adjoint pair $-\lotimes_{\Lambda} C^{\leftvee} \dashv -\lotimes_{\Lambda} C$.

By Claim \ref{abstraction claim},  
 the functor  $- \lotimes_{\Lambda} C^{\leftvee}$ equivalently act on $\sfT$. 
 Therefore 
the functor $C \lotimes_{\Lambda} -$ equivalently acts on 
the full subcategory $\sfT^{\rightvee}$ of $\sfK^{\mrb}(\proj \Lambda^{\op})$. 

Recall that we have an isomorphism $\scrJ_{C, M}: C\lotimes_{\Lambda} M^{\rightvee} \cong M^{\star}$ 
which is natural in $M \in\sfK^{\mrb}(\proj \Lambda)$. 
We also  have 
an isomorphism $N^{\leftvee} \lotimes_{\Lambda} C \cong N^{\star}$ 
as the left version of the above isomorphism 
which is natural in $N \in \sfK^{\mrb}(\proj \Lambda^{\op})$. 
Therefore, we see that 
the $C$-dual functors $(-)^{\star}$ induce an equivalence 
\[
(-)^{\star} : \sfT \stackrel{\sim}{\rightleftarrows} \sfT^{\rightvee} : (-)^{\star}.  
\]

By (1) of  Lemma \ref{abstraction 1}, we conclude that $C^{a} \in \sfT$. 
By (4) of Lemma \ref{abstraction 1}  we conclude that   $(C^{a})^{\leftvee} \cong (C^{\leftvee})^{a}  \in \sfT$. 
This finishes the proof of the implication (2) $ \Rightarrow$ (3). 

Finally,  we prove the implication (3) $\Rightarrow$ (1). 
Let $\scrK_{N}: N \to N^{\star\star}$ be the evaluation morphism for $N \in \sfD(\Mod \Lambda^{\op})$. 
We remark that  the assumption (3-a) implies that 
if $N \in \sfT$, then $\scrK_{N}$ is an isomorphism. 
For $M \in \sfD(\Mod \Lambda)$, 
we denote by $\scrL_{M}: M \lotimes_{\Lambda} C \rightarrow M^{\rightvee \star}$ 
the composite morphism 
\[
\scrL_{M}: M \lotimes_{\Lambda} C \xrightarrow{ \scrG \lotimes C} 
M^{\rightvee \leftvee} \lotimes_{\Lambda} C \xrightarrow{ \scrJ_{\ell} } 
\RHom_{\Lambda^{\op}}(M^{\rightvee}, C)  = M^{\rightvee\star}.
\]
Then we can check that  the following diagram is commutative  for $M \in \sfD(\Mod \Lambda)$
\[
\begin{xymatrix}{
M^{\rightvee}  \ar@{=}[d]\ar[rr]^{\scrK_{M^{\rightvee}}} && M^{\rightvee\star\star} \ar[d]^{(\scrL_{M})^{\star}} \\
\RHom_{\Lambda}(M, \Lambda) \ar[rr]_{\scrT_{M,\Lambda}} && \RHom_{\Lambda}( M\lotimes_{\Lambda} C, C ) .
}\end{xymatrix}
\]
Since the object  $C^{\alpha}$ belongs to $\sfT$ by the assumption (3 -b),  
the morphism $\scrK_{(C^{a})^{\rightvee}}$ is an isomorphism. 
Since $C^{a}$ belongs to $\sfK^{\mrb}(\proj \Lambda)$, the morphism $\scrL_{C^{a}}$ 
is an isomorphism and hence so is $\scrT_{C^{a}, \Lambda}$. 
By Proposition \ref{isomorphisms proposition} we check  the right asid condition. 
Since the condition (3) is right-left symmetric, 
in the same way,  
we can check the left asid condition.   
This finishes the proof of (3)  $\Rightarrow$ (1). 
\end{proof}

\subsubsection{Proof of Theorem \ref{Adaching theorem 1}}

We proceed a proof of Theorem \ref{Adaching theorem 1}. 

\begin{proof}[Proof of Theorem \ref{Adaching theorem 1}]
(1) 
For a triangulated category $\sfD$ and its exact endofunctor $F$, 
we set \[
\mathsf{a}(F, \sfD) :=\min\{a \geq 0 \mid F_{F^{a}(d), d'} \textup{ is an isomorphism for all } d, d' \in \sfD. \}\] 
where we are assuming that the set of which we take the minimal value is not empty.

By Proposition \ref{T proposition},
we can apply Lemma \ref{abstraction 1} 
to the case where $\sfD= \sfK^{\mrb}(\proj \Lambda),$ 
$ \sfE = \thick C^{\alpha}, $
 $F = \cT = -\lotimes_{\Lambda} C$ and $ G= - \lotimes_{\Lambda} C^{\leftvee}$. 
Set $\cT_{\ell}:= C\lotimes_{\Lambda} -$. 
Then we have the following equalities
\[
\begin{split}
\alpha_{r} &\stackrel{(a)}{=} 
\mathsf{a}(\cT, \sfK^{\mrb}(\proj \Lambda) ) \stackrel{(b)}{=} 
\min\{a \geq 0 \mid G_{d, G^{a}(d')} \textup{ is an isomorphism for all } d, d' \in \sfD. \} \\ 
&\stackrel{(c)}{=} \mathsf{a}(\cT_{\ell} , \sfK^{\mrb}(\proj \Lambda^{\op})) 
 \stackrel{(d)}{=} \alpha_{\ell}
 \end{split}
\]
where the equality $(a)$ follows from Proposition
\ref{isomorphisms proposition}, 
the equality $(b)$ follows from  Lemma \ref{abstraction 1}, 
the equality $(d)$ follows from the left version of Proposition
\ref{isomorphisms proposition} 
and the equality $(c)$ follows from the  observation that  
the endofunctor $G$ of $\sfK^{\mrb}(\proj \Lambda)$
 corresponds to the endofunctor $ \cT_{\ell} = C\lotimes_{\Lambda} -$  of $\sfK^{\mrb}(\proj \Lambda^{\op})$ 
via the contravariant equivalence $(-)^{\rightvee}$. 

(2) 
Assume that an admissible subcategory $\sfT \subset \sfK^{\mrb}(\proj \Lambda)$ satisfies 
the conditions (2-a) and (2-b) of Theorem \ref{Adaching theorem 0}. 
Then, by Lemma \ref{abstraction 1}, $C^{\alpha} = \cT^{\alpha}(\Lambda) $ belongs to $\sfT$. 
Hence $\thick C^{\alpha} \subset \sfT$. 
On the other hand, 
since $\sfT \subset \sfK^{\mrb}(\proj \Lambda)$, we have 
$\cT^{\alpha}(\sfT) \subset \thick C^{\alpha}$. 
Moreover, 
the conditions (2-a) implies $\sfT = \cT^{\alpha}(\sfT)$ and hence 
$\sfT = \cT^{\alpha}(\sfT) \subset \thick C^{\alpha}$. 
Thus, $\sfT = \thick C^{\alpha}$. 

Assume that a thick  subcategory $\sfT \subset \sfK^{\mrb}(\proj \Lambda)$ satisfies 
the conditions (3-a) and (3-b) of Theorem \ref{Adaching theorem 0}. 
By (3-b), $\thick C^{\alpha} \subset \sfT$. 
Therefore, we have a semi-orthogonal decomposition 
$\sfT = (\thick C^{\alpha}) \perp (\sfT \cap \Ker \varpi)$ induced from that of $\sfK^{\mrb}(\proj \Lambda)$ in Lemma \ref{T proposition}. 
Since $M \lotimes_{\Lambda} C \cong M^{\rightvee \star}$ for $M \in \sfK^{\mrb}(\proj \Lambda)$, 
the functor $- \lotimes C$ acts $\sfT$ as an equivalence. 
It follows that $\sfT \cap \Ker \varpi  = 0$ by Lemma \ref{T corollary}. 
Thus we conclude that $\sfT = \thick C^{\alpha}$. 

(3) follows from Corollary \ref{T corollary} and Lemma \ref{T proposition}. 
\end{proof}

We can  verify the following assertion 
whose proof is left to the readers. 

\begin{proposition}\label{asid number proposition}
Let $\Lambda$ and $C$ as in Theorem \ref{Adaching theorem 0}. 
Then we have 
\[
\alpha_{r} = \min \{ a \geq 0 \mid (\Ker \varpi) \lotimes_{\Lambda}C^{a} \}.
\]
\end{proposition}

%We close this section by  giving  an optimistic conjecture. 
%\begin{conjecture}
%For any Noetherian algebras $\Lambda$ and any asid modules $C$, the right and the left asid numbers coincide. 
%\end{conjecture}

\section{On an asid subcategory}\label{around}

In Section \ref{around}, we discuss an asid subcategory $\sfT = \thick C^{\alpha}$ for an asid bimodule $C$ over an IG-algebra $\Lambda$. 
The main result given in  Theorem \ref{Adaching theorem 2} proves  that $\sfT$ is equivalent to 
the stable category $\stabgrCM_{\mrlp} A$ of locally perfect  graded CM-modules over $A = \Lambda \oplus C$.

\subsection{$A$-duality and $C$-duality}\label{dual dual}

The aim of Section \ref{dual dual}  
 is to prove Theorem \ref{dual_prop2} 
which gives conditions that $M \in \sfK^{\mrb}(\proj \Lambda)$ belongs to $\sfT$ 
in terms of $A$-duality $(-)^{*} = \grRHom_{A}(-,A)$.  

We need a preparation. 
Let $\Lambda$ be a Netherian algebra and $C$ a bimodule over $\Lambda$ 
finitely generated on both sides.  
Then the trivial extension algebra  $A = \Lambda \oplus C$ is Noetherian. 
Let $M$ be an object of $\sfD^{\mrb}(\mod \Lambda)$. 
and 
$\gamma_{M}$ the morphism induced from the canonical morphism $C \to A(1)$.  
\[
\gamma_{M} : M^{\star}  = \RHom_{\Lambda}(M, C) \to \grRHom_{A}(M ,A)(1) = M^{*} (1).
\]

We give a description of $(M^{*})_{i}$ by using $\gamma_{M}$ and $\scrT$. 

\begin{lemma}\label{dual_lem2}
Let $M$ be an object in $\sfD^{\mrb}(\mod\Lambda)$.
Then,  
there exist isomorphisms  below in $\sfD(\mod\Lambda^{\op})$
\[
\grRHom_{A}(M,A)_{i} 
\ \cong  \
\begin{cases}
0 & (i > 1) \\
\RHom_{\Lambda}(M,C) & (i=1) \\
\cone\left(\scrT_{M\lotimes_{\Lambda}C^{-i}[-i],\Lambda}\right)[-1] & (i < 1). 
\end{cases} 
\]
Moreover, the isomorphism of  the  case $i=1$ is the degree $0$-part of $\gamma_{M}$. 
\end{lemma}

Before giving  a proof, 
we collect two immediate consequences of Lemma \ref{dual_lem2}. 

\begin{corollary}
The following conditions are equivalent for $M \in \sfD^{\mrb}(\mod \Lambda)$.
\begin{enumerate}[(1)]
\item 
The morphism
\[
\scrT_{M \lotimes_{\Lambda}C^{i},\Lambda}: 
\RHom_{\Lambda}(M\lotimes_{\Lambda} C^i,\Lambda ) \to \RHom_{\Lambda}(M\lotimes_{\Lambda} C^{i+1},C )
\]
is an isomorphism in $\sfD(\mod \Lambda^{\op})$ for any $i \geq 0$.
\item 
The morphism $\gamma_{M}$ is an isomorphism in $\sfD(\grmod A^{\op})$.
\item
The object $(M^{*})_{\leq 0}$ is isomorphic to $0$ in $\sfD(\grmod A^{\op})$.
\end{enumerate}
\end{corollary}

\begin{corollary}\label{dual_prop1}
We assume that $\Lambda$ is IG and 
$C$ has finite injective dimension on both sides.
Then the following conditions are equivalent.
\begin{enumerate}[(1)]
\item 
$C$ is a right asid bimodule.
\item 
The morphism $\gamma_{C^{a}}$ is an isomorphism in $\sfD^{\mrb}(\grmod A)$ for some integer $a \geq 0$.
\end{enumerate} 
Moreover we have $\alpha_{r} = \min\{ a \geq 0 \mid \gamma_{C^{a}} \textup{ is an isomorphism} \}$.
\end{corollary}

\begin{proof}[Proof of Lemma \ref{dual_lem2}]
Let $P \in \sfC^{-}(\grproj A)$ be a projective resolution of $M$. 
We may assume that $P_{< 0} = 0$ and hence $\frkp_{i} P = 0$ for $i < 0$. 
Thus, for $i > 1$, 
\[
\grRHom_{A}(M,A)_{i} = \Hom_{\GrMod A}(P, A(i)) = 0.
\]

Since $\frkp_{0} P = P_{0}$ is a projective resolution of $M$ in $\sfC(\proj \Lambda)$ by Lemma \cite[Lemma 4.13]{adasore}, 
we have  isomorphisms 
\[
\begin{split}
\grRHom_{A}(M,A)_{1} & \cong \Hom_{\GrMod A}(P, A(1)) \\ 
&\cong \Hom_{ \Lambda}(\frkp_{0} P, C) \cong \RHom_{\Lambda}(M, C)
\end{split}
\]

We consider the remaining case $i < 1$. 
Substituting $P' = A(i)$ in the exact sequence (4-5) of  \cite{adasore} 
we obtain the following exact triangle in $\sfD(\mod \Lambda^{\op})$ 
\[
\RHom_{\Lambda}(\frkp_{- i+1} P , C) \to \RHom_{\GrMod A}(P, A(i))
 \to \RHom_{\Lambda}(\frkp_{-i}P, \Lambda) \xrightarrow{F_{i}} 
 \RHom_{\Lambda}(\frkp_{-i + 1} P, C[1] ). 
\]

We can check that  $F_{i}$ to be the composite morphism 
\[
\begin{split}
\RHom_{\Lambda}(\frkp_{-i}P, \Lambda) 
&\xrightarrow{\ \ \ \scrT_{{ \frkp_{-i}}P, \Lambda} \ \ \ } 
\RHom_{ \Lambda}(\frkp_{-i}P \otimes_{\Lambda} C , C) \\
& 
\xrightarrow{\ \ \ \Sigma \ \ \ } 
\RHom_{ \Lambda}(\frkp_{-i}P \otimes_{\Lambda} C[1] , C[1])
\xrightarrow{\Hom ( \sfq_{-i+1}, C[1] )} 
\RHom_{\Lambda}(\frkp_{-i+ 1} P, C[1]) 
\end{split}
\]
where $\Sigma$ is the map induced from the shift functor $[1]$. 

By Lemma \ref{adasore lemma 4.2}, $\sfq_{-i+1}$ is isomorphism in $\sfD^{\mrb}(\mod \Lambda)$ for $i < 1$. 
Since the morphism $\Sigma$ is also isomorphism, 
we see  that $\RHom_{\GrMod A}(M,A(i))$ is  a co-cone of $\scrT_{\frkp_{-i} P, \Lambda}$. 
On the other hand, 
since $\frkp_{-i}P \cong M \lotimes_{\Lambda} C^{-i}[-i]$ for $i \leq 0$ by Corollary \ref{black thunder}.(1),  
we finish the proof. 
\end{proof}

We give the main result of Section \ref{dual dual}.

\begin{theorem}\label{dual_prop2}
Assume that $\Lambda$ and $A = \Lambda \oplus C$ is IG. 
We set $\alpha: = \alpha_{r} = \alpha_{\ell}$ and  $\sfT := \thick C^{\alpha}$.  
Then for an object $M \in \sfK^{\mrb}(\proj \Lambda)$  
the following conditions are equivalent. 
\begin{enumerate}[(1)]
\item  $M \in \sfT$. 
\item  $\gamma_{M}$ is an isomorphism. 
\item $(M^{*})_{\leq 0} =0$ in $\sfD^{\mrb}(\grmod A^{\op})$. 
\end{enumerate} 
%In other words, we have 
%\[\begin{split}
%\sfT =& \{M \in \sfK^{\mrb}(\proj\Lambda) \mid  \mbox{ $\gamma_{M}$ is an isomorphism in $\sfD(\grmod A^{\op})$ } \} \\= & \{M \in \sfK^{\mrb}(\proj\Lambda) \mid  (M^{*})_{\leq 0} = 0 \}.\end{split} \]
%\item 
%The functor $(-)^{\star}:  \sfT_{\Lambda} \to  \sfT_{\Lambda^{\op}}$ is a duality such that $(-)^{\star\star} \simeq 1_{\sfT}$.
\end{theorem}

\begin{proof}
The equivalence (2) $\Leftrightarrow$ (3) follows from Lemma \ref{dual_lem2}. 

We prove the implication (1) $\Rightarrow$ (2). 
By Corollary \ref{dual_prop1}, $\gamma_{C^{\alpha}}$ is an isomorphism. 
Since $\sfT = \thick C^{\alpha}$, we conclude that $\gamma_{M} $ is an isomorphism for $M \in \sfT$. 

We prove the implication (3) $\Rightarrow$ (1).  
By the assumption $M^{*} \cong M^{\star}(-1)$  belongs to $\sfK^{\mrb}(\proj \Lambda^{\op})(-1)$. 
We remark that this implies that $M^{*}$ is locally perfect by  Proposition \ref{lp_resol}. 

Let $Q \in \sfC(\grproj A^{\op})$ be a projective resolution of $M^{*}$. 
We may assume that $Q_{\leq 0} = 0$ and $\frkp_{i} Q \in \sfC^{\mrb}(\proj \Lambda^{\op})$ 
for $i \geq 0$.     
Set   $P := \grHom_{A}^{\bullet}(Q, A)$. 
Then since $\frkp_{i}P = \Hom_{\Lambda}(\frkp_{-i} Q, \Lambda)$, 
we see that $\frkp_{i} P \in \sfK^{\mrb}(\proj \Lambda)$ for $i <0$ 
and $\frkt_{\geq 0} P = 0$. 
The latter property implies that $P_{0} =   \frkp_{-1}P \otimes_{\Lambda} C$. 
Since $P$ is isomorphic to $M^{**}= M$ in $\sfD(\grmod A)$, 
$P_{0} \cong M$ in $\sfD^{\mrb}(\mod \Lambda)$ and the complex $P_{< 0}$ is acyclic. 
By Corollary \ref{black thunder}, 
we obtain an isomorphism  $\frkp_{-1} P \cong \frkp_{-i} P \lotimes_{\Lambda} C^{i-1}[i-1]$ for $i  >0$.  
Therefore, 
$M  \cong \frkp_{-\alpha} P \lotimes_{\Lambda} C^{\alpha}[\alpha -1]$.  
\end{proof}

\subsection{The equivalence $\sfT\simeq \stabgrCM_{\mrlp} A$}\label{T=stabgrCM} 

In this section \ref{T=stabgrCM}, 
$\Lambda$ is an IG-algebra and $C$ is a $\Lambda$-$\Lambda$-bimodule 
such that the trivial extension algebra $A = \Lambda \oplus C$ is IG. 
We set $\alpha:= \alpha_{r} = \alpha_{\ell}$ and $\sfT := \thick C^{\alpha}$. 
To state the main theorem of this section, we point out the following observation. 

\begin{lemma}
\begin{enumerate}[(1)]
\item 
The essential image of the functor $\frkp_{0} : \sfK^{\mrac}_{\mrlp}(\grproj A) \to \sfK^{\mrb}(\proj \Lambda)$ 
is contained in $\sfT$. 
\item 
Let $\mathsf{in}: \sfD^{\mrb}(\mod \Lambda) \hookrightarrow \sfD^{\mrb}(\grmod A)$ 
be the canonical inclusion. 
Then the essential image  of $\sfT$ by $\mathsf{in}$  is contained in the Orlov subcategory $\sfO = \sfD^{\mrb}(\mod^{\geq 0} A) \cap \sfD^{\mrb}(\mod^{>0} A^{\op})^{*}$. 
\end{enumerate}
\end{lemma}

\begin{proof}
(1) 
Let $P$ be an object of $\sfC^{\mrac}_{\mrlp}(\grproj A)$.
Then, by Corollary \ref{black thunder}, we have isomorphisms 
$
\frkp_{0} P \cong \frkp_{-\alpha} P \lotimes_{\Lambda} C^{\alpha} [\alpha]. 
$
Thus $\frkp_{0} P \in \thick C^{\alpha} = \sfT$. 

(2) Let  $M$ be an object of $\sfT$. 
It is clear that $M$ belongs to  $\sfD^{\mrb}(\mod^{\geq 0} A)$. 
We set $\sfT_{\Lambda^{\op}} :=\thick_{\Lambda^{\op}} C^{\alpha}$. 
Then, 
$\sfT^{\star} = \sfT_{\Lambda^{\op}}$
 by Theorem \ref{Adaching theorem 0} and the left version of  Lemma \ref{T lemma}. 
In particular $M^{\star} $ belongs to $\sfT_{\Lambda^{\op}} \subset \sfD^{\mrb}(\mod \Lambda^{\op})$.
It follows from the isomorphism $M^{*} = M^{\star}(-1)$ obtained in Theorem \ref{dual_prop2} that 
$M^{*} \in \sfD^{\mrb}(\mod^{>0} A^{\op})$. 
This shows that $M$ belongs to $\sfO$. 
\end{proof}

Now we can state the main theorem of this section. 

\begin{theorem}\label{Adaching theorem 2} 
\begin{enumerate}[(1)]
\item  
The following diagram is commutative up to natural isomorphism.
\[
\begin{xymatrix}{
\sfK_{\mrlp}^{\mrac} (\grproj A) \ar[rr]^{\qquad \underline{\tuZ}^{0}} \ar[d]_{\frkp_{0} } & & 
\stablpCM A \ar[d]^{ \beta} \\
\sfT  \ar[r]_{\mathsf{in}|_{\sfT}} & \sfO \ar[r]_{\pi|_{\sfO}\ \ \ }  & \lpSing A 
}
\end{xymatrix}
\]

\item
All the functors appeared in the above diagram  are equivalences.

\end{enumerate}
\end{theorem}

We use the upper bound $\ub Q$ which was used in the proof of Proposition \ref{lp_resol}.  
Let $ Q \in \sfC(\proj \Lambda)$ be a complex which is  homotopic to a bounded complex $Q' \in \sfC^{\mrb}(\proj \Lambda)$. 
We define  the lower bound $\lb Q$ of $Q$ by $\lb Q  :=   -\ub ( Q^{\rightvee} )$.  
A complex $Q$  is said to be {\it of minimal amplitude}
if  $Q^{n} = 0$ for $n <\lb Q$ or $n  >\ub Q$. 
We remark that 
a complex $ Q \in \sfC(\proj \Lambda)$ which is  homotopic to a bounded complex $Q' \in \sfC^{\mrb}(\proj \Lambda)$ 
is homotopic to a complex of minimal amplitude.

\begin{lemma}\label{boundedness lemma} 
Let
 $P \in \sfC^{\ac}(\GrProj A)$. 
Assume that  for $i \in \ZZ$, 
the complex  $\frkp_{i} P$ belongs to $\sfC(\proj \Lambda)$ 
and is homotopic to an object of $\sfC^{\mrb}(\proj \Lambda)$ for all $i \in \ZZ$. 
 Then the following assertions hold. 
\begin{enumerate}[(1)]
\item  $\ub (\frkp_{i} P)  \leq \ub(\frkp_{0} P)  - i $ for $i \geq 1$.
\item $\lb(\frkp_{-i} P)  \geq  \lb(\frkp_{0} P) +i $ for $i \geq  1$. 
\item Assume moreover that  $\frkp_{i} P$ is of minimal amplitude for $i \in \ZZ$. 
Then $P$ belongs to $ \sfC_{\mrlp}^{\mrac}(\grproj A)$. 
\end{enumerate}
\end{lemma}

\begin{proof} 
(1)  
follows from Corollary \ref{black thunder}.(1). 

(2)  
Since $A$ is IG, the $A$-dual complex $P^{*} $ is acyclic. 
Since $\frkp_{i} (P^{*}) \cong (\frkp_{-i}P)^{\rightvee}$, 
the complex $P^{*} \in \sfC^{\mrac}(\GrProj A^{\op})$ 
satisfies the assumption of (1). 
Thus by (1)
for $i \geq   1$, 
we have the inequality $\ub \frkp_{i}(P^{*}) \leq \ub \frkp_{0} (P^{*}) - i$, 
from which we deduce the inequality of (2). 

(3)   
For a fixed integer $n$,  
we have $(\frkp_{i} P)^{n} = 0$ for $|i| \gg 0$ by (1)  and (2). 
Since  each term $P^{n}$ is  finitely generated, 
hence   
$P$ belongs to $\sfC^{\mrac}_{\mrlp}(\grproj A)$ by Lemma \ref{fg lemma}.   
\end{proof}

Let $n$ be an integer. 
We denote  by $\sigma^{\leq n} P$ the brutal truncation of $P \in \sfC(\GrProj A)$. 
\[
\sigma^{\leq n} P : \cdots  \to P^{n-2} \to P^{n-1} \to P^{n} \to 0 \to 0 \to \cdots .
\]

\begin{proof}[Proof of Theorem \ref{Adaching theorem 2}] 
(1)  
Let $P$ be an object of $\sfK_{\mrlp}^{\mrac}(\grproj A)$. 
By Lemma \cite[Lemma 4.7]{adasore} and  Lemma \ref{boundedness lemma}.(3),  
we may assume that  $P$ belongs to $\sfC_{\mrlp}^{\mrac}(\grproj A)$ 
such that $\frkp_{i} P $ is of minimal amplitude for $i \in\ZZ$.  
We construct an isomorphism $\beta \circ \underline{\tuZ}^{0}(P) \cong \varpi \circ \frkp_{0} (P)$.
The canonical map $\sigma^{\leq 0} P \to \tuZ^{0}(P)$ is an isomorphism in $\sfD^{\mrb}(\grmod A)$. 
Let $n >\max\{ \ub (\frkp_{0} P), 0\}$ be an integer.  
Since the kernel $\Ker \mathsf{cs}$ 
of the  canonical surjection  $\mathsf{cs} :\sigma^{\leq n} P \to \sigma^{\leq 0} P$ 
is of the following form 
\[
\Ker \mathsf{cs}: 
\cdots \to 0 \to 0 \to P^{1} \to P^{2} \to \cdots \to P^{n-1} \to P^{n} \to 0 \to 0 \to \cdots, 
\]
it belongs to $\sfC^{\mrb}(\grproj A)$.  
Therefore the morphism $\mathsf{cs}$
becomes an isomorphism in $\lpSing A$.   
Since $\sigma^{\leq n} P $ is bounded above, 
the complex  $\frkt_{<0} \sigma^{\leq n} P$ belongs to $\sfC^{\mrb}(\grproj A)$ 
by Lemma \ref{boundedness lemma}.(2). 
Therefore  
the canonical map $\sigma^{\leq n} P \to \frkt_{\geq 0} \sigma^{\leq n} P$ become an isomorphism 
in $\lpSing A$. 
By Lemma \ref{boundedness lemma}.(1)  
and the definition of $n$,  
the canonical map $\frkt_{\geq 0} P \to \frkt_{ \geq 0}\sigma^{\leq n} P$ is 
an isomorphism in $\sfC(\grproj A)$.  
 %Then the canonical map $\frkp_{0} P \to \frkp_{0} \sigma^{\leq a} P$ is also isomorphism. 
 Since $P$ is acyclic, $\frkt_{\geq 0} P$ is acyclic in degree greater than 0, i.e., $\tuH(\frkt_{ \geq 0} P)_{> 0} = 0$. 
Thus the canonical map $\frkt_{\geq  0} P \to (\frkt_{\geq 0} P)_{\leq 0} = \frkp_{0} P$ is an isomorphism in $\sfD^{\mrb}(\grmod A)$. 
Combining all, we obtain the following diagram which give an isomorphism $\beta \circ \underline{\tuZ}^{0}(P) \cong \varpi \circ \frkp_{0} (P)$ in  $\lpSing A$. 
\[
\tuZ^{0}(P) \leftarrow \sigma^{\leq 0} 
P \leftarrow \sigma^{\leq n} P \rightarrow \frkt_{\geq 0} \sigma^{\leq n} P 
\leftarrow \frkt_{\geq  0} P \rightarrow \frkp_{0} P.
\]

It can be easily checked that 
this isomorphism does not depend on the choice of $n$ and is natural in $P$. 
Therefore, this diagram gives the desired natural isomorphism $\beta \circ \underline{\tuZ}^{0} \cong \varpi \circ \frkp_{0}$.

(2)  We have already proved that $\beta, \underline{\tuZ}^{0}$ are equivalence.

We prove that the functor $\frkp_{0}$ is an equivalence. 

Since the functors $\beta, \underline{\tuZ}^{0}$ gives equivalences, 
it follows from (1) that the functor $\frkp_{0}$ is faithful.  
Thus, 
it is enough to prove that 
$\frkp_{0}: \sfK^{\mrac}_{\mrlp}(\grproj A) \to \sfT$ is essentially surjective and full.  

We prove $\frkp_{0}$ is essentially surjective. 
Let $M \in \sfT$.  
It is enough to give  a way  to construct 
$P \in \sfK^{\mrac}_{\mrlp}(\grproj A)$ such that $\frkp_{0}P  = M$. 
We denote by $\cT'$ the autoequivalence  $ - \lotimes_{\Lambda} C[1]$ of $\sfT$. 
For $i \in \ZZ$, 
we take $Q_{i} \in \sfC^{\mrb}(\proj \Lambda)$ be a projective representative of $(\cT')^{i}(M)$ 
having minimal amplitude. 
Then there exists  a quasi-isomorphism $\sfq'_{i} : Q_{i}  \to Q_{i-1} \otimes_{\Lambda} C[1]$ in $\sfC(\mod \Lambda)$. 
By the argument of beginning of \cite[Section 4.1.3]{adasore}  
there is an object $P \in \sfC(\GrProj A)$  
such that $\frkp_{i}P  = Q_{i}$ in $\sfC(\proj \Lambda)$ 
and that the morphisms $\sfq'_{i}$ equals to $\sfq_{i}$ in the exact triangle 
(\ref{kurumimochi}). 
Thus the complex $P$ is acyclic by Lemma \ref{adasore lemma 4.2}. 
%and $\frkq_{i}(d_{P}) = \sfq_{i}$ where $d_{P}$ denotes the differentail of $P$. 
By Lemma \ref{boundedness lemma}.(3),    $P$ belongs to $ \sfC^{\mrac}_{\mrlp}(\grproj A)$.

We proceed a proof that  $\frkp_{0}$ is full. 
Let  $g: M \to N $ be a morphism in $\sfT$. 
 We construct  
 a morphism $f: P^{M} \to P^{N}$ 
such that $\frkp_{0}(f) = g$ 
where $P^{M}$ and $P^{N}$ are objects of $\sfC^{\mrac}_{\mrlp}(\grproj A)$ 
such that $\frkp_{0} P^{M} = M, \frkp_{0} P^{N} = N$  
which are constructed as in the previous paragraph. 
Let $Q^{M}_{i}$ and $Q^{N}_{i}$ be the objects of $\sfC^{\mrb}(\proj \Lambda)$ 
which are constructed as in the previous paragraph for $M$ and $N$.

For $i \in \ZZ$,  
we take $g_{i}: Q_{i}^{M} \to Q_{i}^{N}$ 
to be a representative of $(\cT')^{i}(g) : (\cT')^{i}( M )  \to (\cT')^{i}( N )$. 
First note that since the complex $Q^{M}_{i}$ is a bounded above complex of projective $\Lambda$-modules, 
we have an isomorphism  
$\Hom_{\sfK(\mod \Lambda)}(Q_{i}^{M}, Q_{i-1}^{N} \otimes_{\Lambda} C[1])
 \cong \Hom_{\sfD(\mod \Lambda)}(Q_{i}^{M}, Q_{i-1}^{N} \otimes_{\Lambda} C[1])$.  
%By the definitions of $\sfq_{i}^{M}, \sfq_{i}^{N}$ and $g_{i}$
For $i \in \ZZ$,  
the morphisms 
$(g_{i-1} \otimes C[1])  \circ \sfq^{M}_{i}$ and $\sfq^{N}_{i} \circ g_{i}$ 
from $Q^{M}_{i}$ to $Q^{N}_{i-1}\otimes_{\Lambda} C[1]$ in $\sfC(\mod \Lambda)$ 
represent the same morphism 
$(\cT')^{i}(g) : (\cT')^{i}(M) \to (\cT')^{i}(N)$ in $\sfD(\mod \Lambda)$. 
Therefore 
these two morphisms are homotopic. 
%Let  $h_{i}: Q_{i}^{M} \to Q^{N}_{i-1} \otimes C$  be a homotopy between them. 
%Then we check the equation 
%\[(\partial_{i}^{N} \otimes C) \circ h_{i} - h_{i} \circ \partial_{i}^{M}  
 %+ \sfq^{N}_{i}\circ g_{i} -(g_{i-1} \otimes C ) \circ\sfq^{M}_{i}= 0.\]
Thus by \cite[Lemma 4.6]{adasore},  
%the collection $(g_{i},h_{i})_{i \in\ZZ}$ gives 
there exists a morphism $f: P^{M} \to P^{N}$ in $\sfC(\grproj A)$
such that $\frkp_{i}(f) = g_{i}$. In particular  we have $\frkp_{0}(f) = g$. 
This completes the proof that $\frkp_{0}$ is an equivalence.

We prove that the functor $\mathsf{in}|_{\sfT}$ is equivalence. 

Since $\mathsf{in}$ is  fully faithful, so is $\mathsf{in}|_{\sfT}$. 
It only remain to prove that $\mathsf{in}|_{\sfT}:\sfT_{\Lambda} \to \mathsf{O}$ is essentially surjective. 
We remark that since $A_{i} = 0 $ for $i \neq 0, 1$, 
we have $\sfD^{\mrb}_{\mrlp}(\mod^{\geq 1} A^{\op})^{*} \subset \sfD^{\mrb}_{\mrlp}(\mod^{\leq 0} A^{\op})$. 
Therefore $\mathsf{O}$ is contained in  the intersection 
$\sfD^{\mrb}_{\mrlp}(\mod^{\geq 0} A) \cap \sfD^{\mrb}_{\mrlp}(\mod^{\leq 0} A)$, 
which is identified with $\sfK^{\mrb}(\proj \Lambda )  $ via $\mathsf{in}$. 
Let $M$ be an object in $\mathsf{O}$. 
We regard  $M$ as an object of $\sfK^{\mrb}(\proj \Lambda)$. 
Then since $M^{*} \in \sfD^{\mrb}(\mod^{\geq 1}A^{\op})$, we have   $(M^{*})_{\leq 0} = 0$. 
Therefore by Theorem \ref{dual_prop2}, $M$ belongs to $\sfT$. 
This completes the proof that $\mathsf{in}|_{\sfT}$ is an equivalence.   

Finally, $\pi|\sfO$ is an equivalence, since the other functors in the diagram of (1)  are equivalences. 
\end{proof}

As a corollary, we obtain a description 
 of the image of the Happel's embedding functor 
in the case where $\Lambda$ is IG, which is obtained by Chen-Zhang \cite[Theorem 3.1]{CZ}.

\begin{corollary}
Let $\Lambda$ is IG finite dimensional algebra 
and $T(\Lambda) := \Lambda \oplus \tuD(\Lambda)$ the trivial extension algebra. 
Then the Happel embedding functor $\varpi: \sfD^{\mrb}(\mod \Lambda) \to \stabgrmod T(\Lambda)$ 
induces an equivalence 
\[
\varpi : \sfK^{\mrb}(\proj \Lambda) \xrightarrow{\sim } \stabgrmod_{\mrlp} T(\Lambda).
\]
\end{corollary}

We give our proof of Orlov's equivalence. 

\begin{proof}[Proof of Theorem \ref{lp Orlov theorem}] 
From the discussion in Section \ref{quasi-Veronese}, 
we see that the Orlov subcategory $\sfO$ is invariant under quasi-Veronese algebra construction. 
\[
\sfD^{\mrb}_{\mrlp}(\mod^{\geq 0} A) \cap \sfD^{\mrb}_{\mrlp}(\mod^{> 0} A^{\op})^* 
\xrightarrow{\sfqv \ \sim} 
\sfD^{\mrb}_{\mrlp}(\mod^{\geq 0} A^{[\ell]}) \cap \sfD^{\mrb}_{\mrlp}(\mod^{> 0} (A^{[\ell]})^{\op})^*.
\]
Therefore,  we may assume that $A = \Lambda \oplus C$.  
In that case $\pi|_{\sfO}$ gives an equivalence by Theorem \ref{Adaching theorem 2}. 
\end{proof}

\subsection{A recollement involving the canonical functor $\varpi|_{\sfK}$}\label{Locally perfect complexes:Applications}

Combining results which we have obtained, 
we see that 
the canonical functor $\varpi|_{\sfK}: \sfK^{\mrb}(\proj \Lambda ) \to \grSing_{\mrlp} A$ 
fits into a recollement as follows. 
We also obtain a description  of $\Hom$-space of the singularity category $\grSing_{\mrlp} A$.

\begin{theorem}\label{recollement theorem}
Let $\Lambda$ be an IG-algebra and $C$ a $\Lambda$-$\Lambda$-bimodule 
 such that  the trivial extension algebra $A = \Lambda \oplus C$ is IG. 
Then the following assertions hold.

\begin{enumerate}[(1)]
\item 

There exists the recollement of the following form 
\[
\begin{xymatrix}{
\grSing_{\mrlp} A \ar[rr] 
&& \sfK^{\mrb}(\proj \Lambda) \ar[rr] \ar@/^4mm/[ll]^{\varpi|_{\sfK}} \ar@/_4mm/[ll]
&& \Ker \varpi \ar@/^4mm/[ll]^{\mathsf{in}}  \ar@/_4mm/[ll]
}\end{xymatrix}
\]
where  $\mathsf{in}$ is a canonical embedding. 

\item 
 For $M, N \in \sfK^{\mrb}(\proj \Lambda)$, we have 
\[
\Hom_{\grSing A}(\varpi M, \varpi N) \cong \Hom_{\sfD^{\mrb}(\mod \Lambda)}(\RHom(C^{\alpha}, M), \RHom(C^{\alpha} ,N) ).
\]
\end{enumerate}
\end{theorem}

\begin{proof}
(1) follows from Theorem \ref{Adaching theorem 0}, Lemma \ref{T proposition}, Theorem \ref{Adaching theorem 2} and the relationship 
between an admissible subcategory and a recollement which is recalled in Section \ref{Recollection of triangulated categories}.

(2) is proved by the following string of isomorphisms 
\[
\begin{split}
\textup{(LHS)}
&  \stackrel{\textup{(a)}}{\cong} \Hom_{\sfK^{\mrb}(\proj \Lambda)}(\tau(M), N ) \\
& \stackrel{\textup{(b)}}{\cong} \Hom_{\sfK^{\mrb}(\proj \Lambda)}(\RHom_{\Lambda}(C^{\alpha}, M) \lotimes_{\Lambda}C^{\alpha}, N) 
 \stackrel{\textup{(c)}}{\cong} \textup{(RHS)}.  
\end{split} 
\]
where 
the isomorphism (a) can be  obtained from Lemma \ref{right adjoint lemma}, 
the isomorphism (b) is a consequence of Corollary \ref{right adjoint corollary} 
and the isomorphism (c) follows from the  $-\lotimes C^{\alpha} $ -$\RHom (C^{\alpha}, -)$-adjunction.
\end{proof}

Recalling the quasi-Veronese algebra construction, 
we obtain a result which can be applied to any finite dimensional graded IG-algebra $A$ 
which is  not necessary a trivial extension algebra. 
For a finite dimensional algebra $\Lambda$, 
we denote by $|\Lambda|$ the number of pairwise non-isomorphic indecomposable projective modules.

\begin{theorem}\label{Introduction K-gp corollary} 
Let $A= \bigoplus_{i= 0}^{\ell} A_{i}$ be a finite dimensional (over some field $K$ which is not necessary the base field $\kk$)  
graded IG-algebra. 
If the degree $0$-part algebra $A_{0}$ is of finite global dimension, 
then  the Gorthendieck group $K_{0}(\stabgrCM A)$ is free and 
its rank is bounded by $\ell |A_{0}|$ from above.
\[
\rank K_{0}(\stabgrCM A) \leq \ell |A_{0}|.
\]
\end{theorem}

\begin{proof}
Since $\nabla A$ is defined  as  an $\ell \times \ell $-upper triangular matrix 
all of whose diagonals are $A_{0}$, 
the assumption $\gldim A_{0} < \infty $ implies $\gldim \nabla A < \infty$ 
(see e.g. \cite[Corollary 6.3]{adasore}). 
Therefore $\sfK^{\mrb}(\proj \nabla A) = \sfD^{\mrb}( \mod \nabla A)$. 
By Corollary \ref{mrlp corollary} and Lemma \ref{lp=finproj}, $\stabgrCM_{\mrlp} A = \stabgrCM A$. 
Hence by Theorem \ref{recollement theorem}, 
$\stabgrCM A$ is realized as an admissible subcategory of $\sfD^{\mrb}(\mod \nabla A)$. 
In particular, 
$K_{0}(\stabgrCM A)$ is a direct summand of 
$K_{0}(\sfD^{\mrb}(\mod \nabla A))$, which is known to be free of rank $|\nabla A|$. 
Now the result follows from the equation $| \nabla A| = \ell |A_{0}|$.   
\end{proof}

%\section{Triangulated orbit categories (tentative title)}\label{Orbit category}

\section{Application I: two classes of IG algebra of finite CM-type}\label{Applications}

In this Section \ref{Applications}, 
we deal with two particular classes of trivial extension algebras. 
We show that if the classes of algebras are IG, 
then they are of finite (graded) Cohen-Macaulay (CM) type. 
For this, we use the following result which is a CM-version of Gabriel's theorem in covering theory \cite{Gabriel}.

For   a Krull-Schmidt category $\sfA$, we denote by $\ind \sfA$ the set of indecomposable objects of $\sfA$ up to isomorphisms.

Recall that  an  IG (resp. graded IG) algebra $A$ is said to be of finite CM 
(resp. graded CM) type, 
if  the number of isomorphisms classes of CM modules 
(resp. graded CM-module)  is finite (resp. finite up to degree shift)
$\# \ind \stabCM A < \infty$ (resp. $ \# \ind \stabgrCM A /(1) < \infty$). 

\begin{theorem}[{\cite{MYY}}]\label{MYY} 
Let $A$ be  a finite dimensional graded IG-algebra. 
Then, $A$ is of finite CM type if and only if it is of graded CM-type. 
Moreover, if this is the case, 
the functor $\grmod A \to \mod A$ which forgets the grading of graded modules 
induces the equality  $\ind \grCM A/(1) = \ind\CM A$. 
\end{theorem}

\subsection{The case where $\Lambda$ is iterated tilted of Dynkin type}\label{iterated}

The result of Section \ref{iterated} is the following.

\begin{theorem}\label{application:iterated theorem}
Let $\Lambda$ be an iterated tilted algebra of Dynkin type 
and $C$ is a finite dimensional bimodule over $\Lambda$. 
If a trivial extension algebra $A = \Lambda \oplus C$ is IG, 
then it is of  finite  CM type.
\end{theorem}

We need a preparation. 
Let $S$ be a set and $F:S \xrightarrow{\cong} S$ a bijection. 
We denote by $S/F$ the quotient set $S/\ZZ$ by the action of $\ZZ$ on  $S$ which is defined to be $ n\cdot_{F} s := F^{n}(s)$ 
for $n\in \ZZ$ and $s \in S$. 
We note that 
if we take a complete set of representatives $S_{0} \subset S$ of $S/F$,  
then any element $s$ is of the form $s= F^{q}(s_{0})$ for some $q \in \ZZ$ and $s_{0}\in S_{0}$.

\begin{lemma}\label{application:finite lemma}
Let $S$ be a set and $F,G:S \xrightarrow{\cong} S$ bijections such that $FG= GF$.
Assume that $\# S/F = r < \infty$ and that there exists a set of  representatives $\{ s_{1}, s_{2}, \dots, s_{r}\}$ 
such that 
for all $i= 1,\dots, r$ there exists a positive integer $p_{i}> 0$ such that 
$G(s_{i}) = F^{p_{i}}(s_{\sigma(i) } )$ for some $\sigma(i) = 1,\dots, r$. 
Then $\# S/G < \infty$. 
\end{lemma} 

\begin{proof}
First we claim that if $i\neq j$, then $\sigma(i) \neq \sigma(j)$. 
Indeed if $\sigma(i) = \sigma(j)$, then $F^{-p_{i}}G(s_{i}) =F^{-p_{j}}G(s_{j})$. 
Thus we have the equation  $s_{i} = F^{p_{i} - p_{j}}(s_{j})$ which contradicts to the assumption.

By the claim we may regard $\sigma$ as a permutation. 
We have $G^{r!}(s_{i}) = F^{P_{i}}(s_{i})$ where $P_{i} = \sum\nolimits_{a= 0}^{r! -1}p_{\sigma^{a}(i)}$. 
Note that $P_{i} > 0$. 
We claim that every $G^{r!}$-orbit contains  one of the elements  of the set $\{ F^{q}(s_{i}) \mid  0\leq q < P_{i} \}$. 
Indeed any element $s$ is of the form $s= F^{q}(s_{i})$ for some $q \in \ZZ$ and $ i= 1,\dots,r$. 
Let $q_{1}, q_{2} \in \ZZ$ be such that $q= q_{1}P_{i}+ q_{2},\,\, 0\leq q_{2} < P_{i}$. 
Then we have $G^{-q_{1}r!}(s) = F^{q_{2}}(s_{i})$. 

Now by the second claim  $\# S/G \leq \# S/G^{r!} \leq \sum_{i=1}^{r}P_{i} < \infty$. 
\end{proof}

\begin{proof}[Proof of Theorem \ref{application:iterated theorem}]
Let $C$ be an asid $\Lambda$-$\Lambda$-bimodule 
and $\alpha$ the asid number of $C$. 
We set $\sfT := \thick C^{\alpha}$. 
We claim that  $\#\ind \sfT/[1] < \infty$.  
Indeed, let $Q$ be a Dynkin quiver such that $\Lambda$ is derived equivalent to $\kk Q$. 
Since all indecomposable object $M \in \sfD^{\mrb}(\mod \kk Q)$ are of forms 
$M = M'[a]$ for some $M' \in \ind \mod \Lambda$ and $a \in \ZZ$,  
we have $\ind \sfD^{\mrb}(\mod \kk Q)/[1] = \ind \mod \kk Q$. 
Hence $\# \sfT/[1] \leq \# \ind \sfD^{\mrb}(\mod \Lambda)/[1]   < \infty$. 

Let $\{M_{1}, \dots, M_{r}\}$ be a set of representatives of $\ind \sfT/[1]$. 
We may assume that $\tuH^{0}(M_{i}) \neq 0$ and $\tuH^{\geq 1}(M_{i})=0$. 
For $i = 1,\dots,r$, we have $M_{i} \lotimes_{\Lambda} C[1] = M_{j}[p_{i}]$ for 
some $p_{i} \in \ZZ$ and $j = 1,\dots, r$. 
Then we have $\tuH^{1-p_{i}}(M_{i} \lotimes_{\Lambda}C) = \tuH^{0}(M_{j}) \neq 0$. 
On the other hands, since $C$ is a $\Lambda$-module, we have $\tuH^{\geq 1} (M_{i}\lotimes_{\Lambda} C)= 0$. 
Thus we conclude that $p_{i} >0$. 
Now by  Lemma \ref{application:finite lemma}, we conclude  that $\# \ind\sfT/(-\lotimes_{\Lambda}C[1]) < \infty$.  
\end{proof}

\subsection{The case where $C = {}_{\Lambda}N \otimes_{\kk} M_{\Lambda}$ }\label{tensor bimodule case}

A straightforward way to obtain a bimodule over an algebra $\Lambda$ 
is to take a tensor product $C=N \otimes_{\kk} M$ of 
a left $\Lambda$-module $N$ and a right $\Lambda$-module $M$ both of which are finite dimensional over $\kk$.
In Section \ref{tensor bimodule case}, we investigate
the case $A= \Lambda \oplus (N \otimes_{\kk} M)$. 

\subsubsection{Results and examples}
In this Section \ref{tensor bimodule case}, 
$\Lambda$ denotes a finite dimensional algebra of finite global dimension 
and $N$ denotes a nonzero left $\Lambda$-module and $M$ denotes  a nonzero  right $\Lambda$-module. 
We set  $C : = N \otimes_{\kk} M$ and $A: = \Lambda \oplus C$. 

For simplicity we restrict ourselves to deal with this case. 
We note that  some of the statements  can be verified in the more general  case where $\Lambda$ is IG 
and $ \injdim M < \infty, \ \injdim N < \infty$.

 \begin{theorem}\label{tensor bimodule case theorem}
 Let $\Lambda$ be a finite dimensional algebra of finite global dimension, 
 
\begin{enumerate}[(1)]
\item  $\gldim A< \infty $ if and only if $M \lotimes_{\Lambda} N = 0$. 
 
 \item $A$ is IG and $\gldim A = \infty$ 
 if and only if $M$ is exceptional (i.e., $\RHom_{\Lambda}(M,M) \cong \kk$) and $\RHom_{\Lambda}(M,\Lambda) = N[-p]$ for some $p \in \NN$.

 \noindent
 \hspace{-20pt}
 \item 
Assume that $A$ is IG and $\gldim A= \infty$. 
 Then the following assertions hold.

 \begin{enumerate}[(3-a)]
 \item Let $p$ be the integer in (2). 
Then $p = \pd_{\Lambda} M = \pd_{\Lambda^{\op}} N$.

 \item  There are equivalences 
$ \stabgrCM A \cong \sfD^{\mrb}(\mod \kk)$ 
under which the graded degree shift functors $(1)$ corresponds to the complex degree shift functor $[p+1]$.

%\item $\dim \End_{\grSing A}(\varpi\Lambda) = (\dim_{\kk} M)^{2}$. 

\item 
$\stabCM A \simeq \sfD^{\mrb}(\mod \kk)/ [p+1] \simeq (\mod \kk)^{\oplus p +1}$   
 and the syzygy functor $\Omega$ is $(p+1)$-periodic.

\item $\ind \stabCM A = \{ M , \Omega M, \cdots, \Omega^{p} M \}$.

\end{enumerate} 

\end{enumerate}
 \end{theorem}

Combining Theorem \ref{tensor bimodule case theorem} and a result by X-W. Chen, 
we can provide an example of a non-IG finite dimensional algebra $A$ 
such that the singularity category $\Sing A$ is $\Hom$-finite. 

\begin{example}
Let $\Lambda$ be a basic finite dimensional algebra of finite global dimension 
and $e,f \in \Lambda$ idempotent elements. 
Then the algebra $A = \Lambda \oplus (\Lambda e \otimes_{\kk} f \Lambda)$ 
is of finite global dimension if and only if $f \Lambda e = 0$. 
The algebra $A$ is an IG-algebra of  infinite global dimension 
if and only if $e = f$ and $\dim e \Lambda e  = 1$. 

On the other hands, X-W. Chen \cite{XWChen} showed that 
$\Sing A$ is $\Hom$-finite if and only if $\dim f \Lambda e \leq 1$. 
Thus we conclude that there are finite dimensional algebras $A$
 which is not IG but whose  singularity category $\Sing A$ is $\Hom$-finite. 
\end{example}

\begin{example}
Let $\Lambda$ be a $d$-representation infinite algebra and $e \in \Lambda $ an idempotent 
such that $\dim e \Lambda e = 1$. 
Set $\theta = \Ext^{d}_{\Lambda}(D(\Lambda), \Lambda)$. 

Then $A = \Lambda \oplus D(e \Lambda) \otimes_{\kk} e\theta$ is an IG-algebra 
such that 
$\stabCM A \cong (\mod \kk)^{\oplus d +1}, \stabCM A^{\op} \cong (\mod \kk)^{\oplus d+1}$. 
\end{example}

\subsubsection{Proof of Theorem \ref{tensor bimodule case theorem}.}

First observe that $C^{a} = N \otimes_{\kk} (M \lotimes_{\Lambda} N)^{\otimes_{\kk} a -1} \otimes_{\kk} M$ for $a \geq 1$.  
On the other hand, under the assumption that $\gldim \Lambda < \infty$, we have 
 $\gldim A< \infty$ if and only if $C^{a} = 0$ for some $a \geq 0$ 
by \cite[Corollary 4.15]{adasore}.
Hence we deduce the following lemma.

\begin{lemma}\label{tbc lemma 1} 
We have $M\lotimes_{\Lambda} N = 0$ 
if and only if $\gldim A < \infty$. 
\end{lemma}

Since 
$\injdim C_{\Lambda} = \injdim M < \infty, \ 
\injdim {}_{\Lambda}C = \injdim N < \infty$, 
the trivial extension algebra $A$ is IG 
if and only if 
$C$ satisfies the right and left asid conditions  
by Theorem \ref{Adaching theorem 0}.

We concentrate on the right asid condition. 
By Lemma \ref{isomorphisms proposition}, 
the right asid  condition  is satisfied 
if and only if 
the morphism $\scrT_{C^{a}, \Lambda}: \RHom_{\Lambda}(C^{a}, \Lambda) \to \RHom_{\Lambda}(C^{a + 1} ,C)$ 
is an isomorphism. 
Moreover, if this is the case we have 
$\alpha_{r} = \min\{ a \geq 0 \mid \scrT_{C^{a}, \Lambda} \textup{ is an isomorphism.} \}$.

We give a description of the morphism $\scrT_{C^{a}, \Lambda}$ 
in terms of more concrete morphisms, which are introduced now.

By $\mathsf{lm}$ we denote the left multiplication map 
$\mathsf{lm}:  \Lambda  \to \Hom_{\kk}( N , N)$, 
i.e., 
$\mathsf{lm}(r)(n) =rn$ for $ r\in \Lambda$ and $ n \in N$.   

We define a morphism $\mathsf{h}: \kk \to \RHom_{\Lambda}(M,M)$ 
in the following way. 
By $\tilde{\mathsf{h}}$,  we denote 
the homothety map $\tilde{\mathsf{h}} : \kk \to \Hom_{\Lambda}(M, M)$, 
i.e., 
$\tilde{\mathsf{h}}(a)(m) := am$ 
for $a \in \kk$ and $m \in M$. 
Then, we set $\mathsf{h} := \mathsf{can} \circ \tilde{\mathsf{h}}$ 
where $\mathsf{can}: \Hom_{\Lambda}(M,M) \to \RHom_{\Lambda}(M,M)$ 
is the canonical morphism.

By $\scrS$, we denote the morphism 
$\scrS: \RHom_{\Lambda}(M, \Lambda) \to \Hom_{\kk}(M \lotimes_{\Lambda} N, N)$ 
in $\sfD^{\mrb}(\mod \Lambda^{\op})$  
induced from the functor $- \lotimes_{\Lambda} N : \sfD^{\mrb}(\mod \Lambda) \to \sfD^{\mrb}(\mod \kk)$. 
\[
\scrS: \RHom_{\Lambda}(M, \Lambda) \to \RHom_{\kk}(M \lotimes_{\Lambda} N, N)\cong \Hom_{\kk}(M \lotimes_{\Lambda} N, N)
\]

We leave the verification of the following lemma to the readers.

\begin{lemma}\label{tbc lemma 2} 
\begin{enumerate}[(1)]
\item 
Under the isomorphisms  
\[
\begin{split} 
&\RHom_{\Lambda} (\Lambda, \Lambda) \cong \Lambda \otimes_{\kk} \kk 
\textup{ and } \\ 
&\RHom_{\Lambda}(N \otimes_{\kk} M ,  N \otimes_{\kk} M ) 
\cong \Hom_{\kk}(N,N) \otimes_{\kk} \RHom_{\Lambda} ( M ,M ), 
\end{split}
\]  
the morphism 
$\scrT_{\Lambda, \Lambda}: \RHom_{\Lambda}(\Lambda,\Lambda) \to 
\RHom_{\Lambda}(N \otimes_{\kk} M , N \otimes_{\kk} M)$ 
corresponds to $\mathsf{lm} \otimes \mathsf{h}$.  
\[
\mathsf{lm} \otimes \mathsf{h}: \Lambda \otimes_{\kk} \kk  \to \Hom_{\kk}(N,N) \otimes_{\kk} \RHom_{\Lambda} ( M ,M )
\]

\item 
Let $a \geq 1$. 
Then under the isomorphisms 
\[
\begin{split}
& \RHom_{\Lambda}(C^{a} , \Lambda) 
\cong \tuD(N \otimes_{\kk} (M\lotimes_{\Lambda} N)^{\otimes_{\kk} a-1}) 
\otimes_{\kk} \RHom_{\Lambda} (M,\Lambda) \otimes_{\kk} \kk  \ \ \ \ \textup{ and} \\
&\RHom_{\Lambda}(C^{a+1} , C) 
\cong \tuD(N \otimes_{\kk} (M\lotimes_{\Lambda} N)^{\otimes_{\kk} a-1}) \otimes_{\kk} 
\Hom_{\kk} (M\lotimes_{\Lambda}N ,N)
 \otimes_{\kk} \RHom_{\Lambda} (M,M)
\end{split}
\] 
the morphism $\scrT_{C^{a}, \Lambda}$ corresponds to $\id \otimes \scrS \otimes \mathsf{h}$. 
\[
\id_{L_{a}} \otimes \scrS \otimes \mathsf{h}: 
 L_{a} \otimes_{\kk} \RHom_{\Lambda} (M,\Lambda) \otimes_{\kk} \kk \\
 \to 
 L_{a}  \otimes_{\kk} 
\Hom_{\kk} (M\lotimes_{\Lambda}N ,N)
 \otimes_{\kk}  \RHom_{\Lambda} (M,M).
\] 
where we set $L_{a} := \tuD(N \otimes_{\kk} (M\lotimes_{\Lambda} N)^{\otimes_{\kk} a-1}) $.  
\end{enumerate}
\end{lemma}

In the next proposition, 
we study when  $C = N \otimes_{\kk} N$ satisfies the right asid condition.

\begin{proposition}\label{tbc proposition}
The right asid  condition  is satisfied 
if and only if one of the following conditions is satisfied. 

\begin{enumerate}[(1)]
\item[(0)] $\Lambda = \End_{\kk}(V), \ N = V$ and $ M = \tuD(V)$ 
for some finite dimensional vector space $V$. 

\item[(1)] The module $M$ is exceptional and $\RHom_{\Lambda}(M, \Lambda) \cong N[-p]$ for some $p \in \ZZ$. 
The condition (0) does not hold. 

\item[(2)]  $M \lotimes_{\Lambda} N = 0$ 
\end{enumerate}

Moreover, if one of conditions (a)  is satisfied for $a = 0,1,2$, then we have $\alpha_{r} = a$. 
\end{proposition}

We prepare the following lemma

\begin{lemma}\label{tbc lemma 3} 
\begin{enumerate}[(1)]
\item The morphisms  $\mathsf{lm}$ and $\mathsf{h}$ are isomorphisms 
 if and only if $\Lambda = \End_{\kk}(V), \ N = V$ and $ M = \tuD(V)$ 
for some finite dimensional vector space $V$. 

\item 
The morphism  $\scrS$ and $\mathsf{h}$ are isomorphisms 
if and only if 
 $M$ is exceptional and $\RHom_{\Lambda}(M, \Lambda) \cong N[-p]$ for some $p \in \ZZ$. 
\end{enumerate}

Moreover, in either cases, we have $M \lotimes_{\Lambda} N \neq 0$. 

\end{lemma}

\begin{proof}
(1) 
 Since the map $\mathsf{lm}$ is an algebra homomorphism, 
it is an isomorphism if and only if 
the algebra $\Lambda$ is the endomorphism $\End_{\kk}(V)$ 
of the $\kk$-vector space $V := {N}_{\kk}$. 
If this is the case, then $\Lambda$ is a full matrix algebra. 
Thus it is easy to see that the homothety map $\mathsf{h}$ is an isomorphism 
if and only if  $M = \tuD(V)$ as left modules and $N \otimes_{\kk} M \cong \End_{\kk}(V)$ as bimodules. 
In this case it is clear that  $M \lotimes_{\Lambda} N \neq 0$.

(2) 
First assume that 
the morphism  $\scrS$ and $\mathsf{h}$ are isomorphisms. 
Applying $M\lotimes_{\Lambda} -$ to the isomorphism $\scrS$, 
we obtain the isomorphism 
$M \lotimes \scrS : \RHom_{\Lambda}(M,M) \xrightarrow{\cong} \Hom_{\kk}(M\lotimes_{\Lambda} N, M \lotimes_{\Lambda} N)$. 
Since $\RHom_{\Lambda}(M,M) \cong \kk$, 
we must have 
$M \lotimes_{\Lambda} N \cong \kk[p]$ for some integer $p \geq 0$. 
In particular, we have $M \lotimes_{\Lambda} N \neq 0$. 
Now  $\scrS$ become  an  an isomorphism $\RHom_{\Lambda}(M, \Lambda) \cong N[-p]$. 

The converse implication can easily be proved and left to the readers. 
\end{proof}

\begin{proof}[Proof of Proposition \ref{tbc proposition}]
We assume that the right asid  condition  is satisfied. 
By Lemma \ref{tbc lemma 2}, 
if $\alpha_{r} = 0$, then $\mathsf{lm}$ and $\mathsf{h}$ are isomorphisms. 
Therefore by Lemma \ref{tbc lemma 3} the condition (0) holds. 
In the case  $\alpha_{r} =1$,  
$\scrS$ and $\mathsf{h}$ are isomorphisms by Lemma \ref{tbc lemma 2}. 
Therefore by Lemma \ref{tbc lemma 3} the condition (1) holds. 
Finally,  
we assume that $\alpha_{r} \geq 2$. 
If $M \lotimes_{\Lambda} N \neq 0$, then we have $L_{a} \neq 0$. 
Therefore,  since $\scrT_{C^{a}, \Lambda}$ is an isomorphism for some $a \geq 1$, 
$\scrS$ and $\mathsf{h}$ are isomorphisms by Lemma \ref{tbc lemma 2}.  
Hence, again by Lemma \ref{tbc lemma 2} we have $\alpha_{r} \leq 1$.  
This contradicts to the assumption. 
Thus, if $\alpha_{r} \geq 2$, 
we must have $M \lotimes_{\Lambda} N = 0$. 
Moreover, by Lemma \ref{tbc lemma 2} the last condition implies that $\alpha_{r} = 2$.

By Lemma \ref{tbc lemma 2} and  Lemma \ref{tbc lemma 3}, 
it is easy to see that 
if one of the conditions (0), (1), (2) is  satisfied, 
 then $C$ satisfies the right asid  condition  and $\alpha_{r}$ has desired value.  
\end{proof}

We proceed a proof of Theorem \ref{tensor bimodule case theorem}. 

\begin{proof}[Proof of Theorem \ref{tensor bimodule case theorem}]
Combining Lemma \ref{tbc lemma 1}, Proposition \ref{tbc proposition},  
we deduce (1) and (2). 

(3-a) 
Since $\pd M < \infty$, 
the isomorphism $\RHom_{\Lambda}(M, \Lambda) \cong N[-p]$ implies that 
$p = \pd M$. 
Since derived the $\Lambda$-dual of the above isomorphism 
yields an isomorphism $\RHom_{\Lambda^{\op}}(N, \Lambda) \cong M[-p]$, 
we conclude that $p = \pd N$. 

(3-b) 
By Lemma \ref{varpi shift lemma}  and Theorem \ref{Adaching theorem 2},   
we have an equivalence $\stabgrCM A \simeq \sfT$ under which 
the autoequivalence $(1)$ of $\stabgrCM A$ corresponds to 
$- \lotimes_{\Lambda}C[1]$.

Observe that $\sfT = \thick M$. 
Since $M$ is exceptional,  
we have  an equivalence $F: \thick M  \simeq \sfD^{\mrb}(\mod \kk)$ 
which sends $M$ to $\kk$. 
Using the isomorphism $N \cong \RHom_{\Lambda}(M,\Lambda)[p]$
we obtain  the isomorphisms 
\[
M \lotimes_{\Lambda}C \cong M \lotimes_{\Lambda} N \otimes_{\kk} M \cong \RHom_{\Lambda}(M,M)\otimes_{\kk} M[p] \cong M[p].  
\]
Therefore, under above equivalence $F$, the autoequivalence $-\lotimes_{\Lambda}C[1]$ 
corresponds to $[p+1]$. 
Hence we obtain the desired equivalence.

(3-c) is  obtained by combining (3-b) and  Theorem \ref{MYY}.

(3-d) By (3-c),  
it is enough to show that $M$ is an indecomposable CM-$A$-module. 
Since $M$ is exceptional, it is indecomposable. 
To show that $M$ is CM, we construct a complete resolution  of $M$  as $A$-module. 
Let $0\to Q^{-p} \xrightarrow{ \partial^{-p}}  \cdots \xrightarrow{ \partial^{-1}} 
Q^{0} \rightarrow  M \to 0$ be a minimal projective resolution as $\Lambda$-modules. 
 Since $M \lotimes_{\Lambda} C \cong M \lotimes_{\Lambda}N \otimes_{\kk} M \cong M[p]$, 
applying $- \otimes_{\Lambda}C$ to the minimal projective resolution, we obtain the following exact sequence of $\Lambda$-modules 
\[
0\to  M \to Q^{-p} \otimes_{\Lambda} C 
\xrightarrow{ \partial^{-p} \otimes C} \cdots \xrightarrow{ \partial^{-1} \otimes C} Q^{0} \otimes_{\Lambda} C\to 0. 
\]
Therefore applying $- \otimes_{\Lambda}A$ to the minimal projective resolution, 
we obtain the following exact sequence of $A$-modules 
\[
0\to  M \to Q^{-p} \otimes_{\Lambda} A 
\xrightarrow{ \partial^{-p} \otimes A} \cdots \xrightarrow{ \partial^{-1} \otimes A} Q^{0} \otimes_{\Lambda} A
\to  M \to 0. 
\]
Splicing the copies of this sequence we obtain a complete resolution $P$ of $M$ as an $A$-module as desired.
\[
P: \ \ \ \cdots \to
Q^{0} \otimes_{\Lambda} A \xrightarrow{\ \ } Q^{-p} \otimes_{\Lambda} A 
\xrightarrow{ \partial^{-p} \otimes A} \cdots \xrightarrow{ \partial^{-1} \otimes A} Q^{0} \otimes_{\Lambda} A
\xrightarrow{ \ \ } 
Q^{-p} \otimes_{\Lambda} A \xrightarrow{ \partial^{-p} \otimes A} \cdots
\]
\end{proof}

\begin{remark}
We remark that we can prove the 
Gorenstein symmetric conjecture for 
a trivial extension algebra $A$ of the form $A = \Lambda \oplus (N \otimes_{\kk} M)$. 
Namely, $A$ is IG if and only if $\injdim A_{A} < \infty$ 
if and only if $\injdim {}_{A}A < \infty$. 
\end{remark}

\section{Application II: classification of asid bimodules}\label{Classification}

Theorem \ref{Adaching theorem 0} gives us a way to classify asid bimodules over an IG-algebra $\Lambda$ 
as follows: 

\bigskip
\noindent
\underline{Step 1.} 
Classify admissible subcategories $\sfT$ of $\sfK^{\mrb}(\proj \Lambda)$. 

\bigskip
\noindent
\underline{Step 2.} 
For an admissible subcategory $\sfT$, 
classify bimodules $C$ such that 
the functor $- \lotimes_{\Lambda} C$ acts $\sfT$ as an equivalence 
and nilpotently acts on $\sfT^{\perp}$.

\bigskip

In Section \ref{A2},  
we deal with  the path algebra $\Lambda=\kk Q$ of $A_{2}$ quiver $Q= 1 \xleftarrow{\alpha} 2$.
We demonstrate how to use the strategy and give the complete list of the asid bimodules over $\Lambda$. 

In Section \ref{A3}, 
we deal with  the path algebra $\Lambda=\kk Q$ of $A_{3}$ quiver $Q= 1 \xleftarrow{\alpha} 2 \xrightarrow{\beta} 3$. 
We give the complete  list of the asid bimodules $C$  such that $\Lambda \oplus C$ has infinite global dimension. 
The restriction is put for the sake of space. 

\begin{remark}
In this section, 
suffixes are used in a different way from other sections. 
For a $\Lambda$-module $M$ and $n\in \NN$,  we denote by $M^{n}$ the direct sum  $M^{\oplus n}$. 
\end{remark}

\subsection{Classification of asid bimodules over the path algebra of $A_{2}$ quiver}\label{A2}

Let $\Lambda = \kk[1 \xleftarrow{\alpha} 2]$ be the path algebra, 
$e_{1}, e_{2}$ the idempotent elements corresponding to the vertex $1,2$. 
We denote by 
$S_{1}^{\ell}$  the simple left $\Lambda$-module 
which corresponds to the vertex $1$. 
We denote by $S_{2}^{r}$ the simple right $\Lambda$-module 
which corresponds to the vertex $2$.

In the list below, we give an admissible subcategory $\sfT$ of $\sfK^{\mrb}(\proj \Lambda ) = \sfD^{\mrb}(\mod \Lambda)$ 
and an asid bimodule $C$ such that $\sfT = \thick C^{\alpha}$.

\begin{classification}\label{classification theorem A2}
 
\begin{enumerate}
\item[\exnumber{\textup{1}}] 
$\sfT_{1} = \mathsf{D}^{\mathrm{b}}(\textup{mod}  \Lambda)$.   
$C = \Lambda, D(\Lambda)$. 

\item[\exnumber{\textup{2}}] 

$\sfT_{2} =\thick  e_{1} \Lambda$.  
$C = \Lambda e_{1} \otimes_{\kk} e_{1} \Lambda$.

 \item[\exnumber{\textup{3}}]
$\sfT_{3}=\thick e_{2}\Lambda$.  
$C=\Lambda e_{2} \otimes_{\kk} e_{2} \Lambda$.

\item[\exnumber{\textup{4}}]
$\sfT_{4}=\thick S_{2}^{r}$.   
$C=S_{1}^{\ell} \otimes_{\kk} S_{2}^{r}$.

\item[\exnumber{\textup{5}}]
$\sfT_{5}=0$. 
$C=(\Lambda e_{2} \otimes_{\kk} e_{1} \Lambda)^{n}, (S_{1}^{\ell} \otimes_{\kk} e_{2} \Lambda)^{ n}, 
(\Lambda e_{1} \otimes_{\kk} S_{2}^{r})^{n}$ for $n \in \NN$. 
\end{enumerate}
\bigskip 

We note that 
the case \exnumber{\textup{1}} is presicely the case $\alpha = 0$ 
and that 
the case \exnumber{\textup{5}} is precisely the case $\gldim A < \infty$.
\end{classification}

Before starting the classification, 
we recall basic facts about $\Lambda = \kk[1 \leftarrow 2]$ (for the details we refer  \cite{Happel book}). 
The algebra $\Lambda$ has three indecomposable modules up to isomorphisms, 
whose quiver representations are given by
\[
P_1= e_{1} \Lambda =
( \xymatrix@!=10pt{ \kk  & 0 \ar@{.>}[l]}),
\hspace{5mm} 
P_2= e_{2} \Lambda  =
( \xymatrix@!=10pt{ \kk & \kk  \ar[l]}),
\hspace{5mm} 
I_2= S_{2}^{r} =
( \xymatrix@!=10pt{ 0  & \kk \ar@{.>}[l] }),
\]
where solid arrow is identity map and dotted arrows are zero maps.
%The Auslander-Reiten quiver of $\mod\Lambda$ is given by
%\[\xymatrix@!=5pt{ & P_2 \ar[rd] & \\P_1 \ar[ru] & & I_2 \ar@{.>}[ll] }\]
%Since $\Lambda$ has finite global dimension, 
%the homotopy category $\sfK^{\mrb}(\proj\Lambda)$ is canonically equivalent to the derived category $\sfD^{\mrb}(\mod\Lambda)$.
%Therefore we deal with $\sfD^{\mrb}(\mod\Lambda)$ instead of $\sfK^{\mrb}(\proj\Lambda)$.
%Since $\Lambda$ is hereditary, 
The complete set of pairwise non-isomorphic indecomposable objects of $\sfD^{\mrb}(\mod\Lambda)$ is given by
\[
\{ X[i] \ | \ X: \mbox{indecomposable $\Lambda$-module}, \ i \in \ZZ \}
\]
and the Auslander-Reiten quiver of $\sfD^{\mrb}(\mod\Lambda)$ is the following.
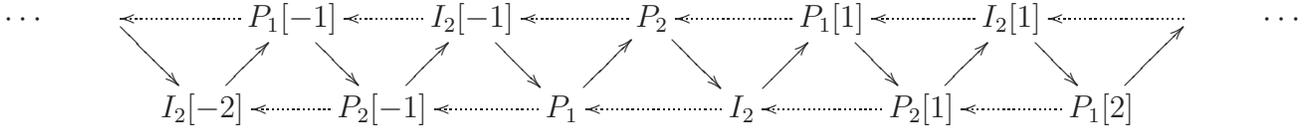
\[ 
\xymatrix@!=10pt{
\cdots & \ar[rd]   & & P_1[-1] \ar[rd]\ar@{.>}[ll] & & I_2[-1] \ar[rd]  \ar@{.>}[ll]  &  & P_2 \ar@{.>}[ll] \ar[rd] & & P_1[1] \ar[rd]  \ar@{.>}[ll]  & & I_2[1] \ar@{.>}[ll] \ar[rd] & &\ar@{.>}[ll]&\cdots \\
 & & I_2[-2]\ar[ru] & & P_2[-1] \ar@{.>}[ll] \ar[ru] & & P_1 \ar@{.>}[ll] \ar[ru] & & I_2 \ar@{.>}[ll] \ar[ru] & & P_2[1] \ar@{.>}[ll] \ar[ru] & & P_1[2] \ar[ru]\ar@{.>}[ll] & &
}
\]

Now we give sketches of Step 1 and Step 2 for the algebra $\Lambda$. 

\bigskip

\noindent
\underline{Step 1.} 
We classify admissible subcategories of $\sfK^{\mrb}(\proj\Lambda)=\sfD^{\mrb}(\mod\Lambda)$.
Since $\Lambda$ is a path algebra of a Dynkin quiver, 
we can classify thick subcategories of $\sfD^{\mrb}(\mod\Lambda)$ by applying \cite[Theorem 5.1]{Bruning} and \cite[Theorem 1.1]{IT}.
One can check that all of them are admissible subcategories.
Thus we obtain the complete list of admissible subcategories, which is shown in the table \ref{class_adm} below.
%This example has five admissible subcategories of $\sfD^{\mrb}(\mod\Lambda)$.

%%%%%%%%%%%%%%%%%%%%%%%%%%%%%%%%%%%%%%%%%
%%%%%%%%%%%%%%%%%%%%%%%%%%%%%%%%%%%%%%%%% list of admissible subcategories (add version)
%%%%%%%%%%%%%%%%%%%%%%%%%%%%%%%%%%%%%%%%%
\begin{table}[h]
\begin{center}
\begin{tabular}{lll}
%%%%%%%%%%%%%%%%%%%%%%%%
\exnumber{1} \hspace{1mm}
$\sfT_{1}=\sfD^{\mrb}(\mod\Lambda)$, 
&
%%%%%%%%%%%%%%%%%%%%%%%%
\exnumber{2} \hspace{1mm}
$\sfT_{2}=\thick P_{1}$,
&
%%%%%%%%%%%%%%%%%%%%%%%%%
\exnumber{3} \hspace{1mm}
$\sfT_{3}=\thick P_{2}$, 
\\
%%%%%%%%%%%%%%%%%%%%%%%%
\exnumber{4} \hspace{1mm}
$\sfT_{4}=\thick I_{2}$, 
&
%%%%%%%%%%%%%%%%%%%%%%%%
\exnumber{5} \hspace{1mm}
$\sfT_{5}=0$.
\end{tabular}
\end{center}
\caption{List of admissible subcategories of $\sfD^{\mrb}(\mod\Lambda)$} \label{class_adm}
\end{table}

\noindent
\underline{Step 2.} 
%For $i=1,2$, we denote by $e_{i}$ the idempotent of $\Lambda$ corresponding to the vertex $i$ of $Q$.
We represent a $\Lambda$-$\Lambda$-bimodule $C$ as a quiver representation
\[
\xymatrix@R=20pt@C=11pt{
e_1Ce_1 \ar[d]_{\alpha \, \cdot}  && e_1Ce_2 \ar[ll]_{\cdot \, \alpha} \ar[d]^{\alpha \, \cdot} \\
e_2Ce_1  && e_2Ce_2  \ar[ll]^{\cdot \, \alpha}  \\ 
}
\]
which makes the square commutative. 
This is nothing but a quiver representation of $C$ as a $\Lambda^{\mre}$-module.

By Morita Theory, we can recover a $\Lambda$-$\Lambda$-bimodule $C$ from the functor $\cT= - \lotimes_{\Lambda} C$ in the following way.
The right $\Lambda$-module structure of $C = e_{1}C \oplus e_{2} C$ can be read off from a canonical isomorphism $\cT(P_{i}) \cong e_{i}C$.
The left $\Lambda$-module structure of $C$ can be read off from the commutative diagram
\[
\xymatrix@!=10pt{
\cT(P_1) \ar[rr]^{\cT(\alpha \cdot)} \ar[d]_{\cong} & & \cT(P_2) \ar[d]^{\cong} \\
e_1C \ar[rr]_{\alpha\cdot} & & e_2C
}
\]
where the vertical maps are canonical isomorphisms.
From these datum, we may represent $C$ as a quiver representation.

Thanks to the above observation, 
we can calculate quiver representations of asid bimodules corresponding to each admissible subcategory $\sfT$ by referring possible equivalences on $\sfT$ and nilpotent endofunctors on $\sfT^{\perp}$.
In the following example, we explain how to do it for a concrete admissible subcategory.

%%%%%%%%%%%%%%%%%%%%%%%%

\begin{example}
We deal with the admissible subcategory 
$\sfT_{2}$. 
First observe that $\sfT_{2}=\add\{ P_1[i] \ | \ i \in \ZZ \}$ and $\sfT_{2}^{\perp}=\add\{ I_2[i] \ | \ i\in \ZZ \}$. 

Let $C$ be a bimodule such that $\cT = - \lotimes_{\Lambda}C$ acts on $\sfT_{2}$ as an equivalence 
 and nilpotently acts on $\sfT_{2}^{\perp}$.
Since  $\cT$  acts on $\sfT_{2}$ as an equivalence,  
$\cT(P_1) \cong P_{1}[n]$ for some $n \in \ZZ$. 
On the other hand  $\cT(P_{1}) = e_{1} C$ belongs to $\mod \Lambda$.  
Therefore we conclude that $e_{1} C = \cT(P_{1}) \cong P_1$.

We claim that $\cT(I_{2}) = 0$. 
Indeed, by the assumption $\cT(\sfT_{2}^{\perp} ) \subset \sfT_{2}^{\perp}$,  
we have $\cT(I_2) \cong I_2^{m} \oplus I_2[1]^{ n}$ for some non-negative integers $m,n$. 
Since $\cT^{a}(I_{2}) = 0 $ for $a \gg 0$,  $m$ and $n$ must be zero, namely  $\cT(I_2) = 0$.

By applying $\cT$ to the triangle \eqref{tri-alpha} below, 
we see that $\cT(\alpha \cdot):\cT(P_1) \to \cT(P_2)$ is an isomorphism.
\begin{equation}\label{tri-alpha}
I_2[-1] \to P_1 \xrightarrow{\alpha \cdot} P_2 \to I_1 
\end{equation} 

We conclude that 
the right module structure of $C$ is
\[
e_1C = P_1=( \xymatrix@!=10pt{ \kk  & 0 \ar@{.>}[l]}), \
e_2C = P_1=( \xymatrix@!=10pt{ \kk  & 0 \ar@{.>}[l]}),
\]
and the left module structure of $C$ is given by a map
\[
\xymatrix@!=10pt{ e_1C \ar[d]_{\alpha \cdot} & = & \kk  \ar[d]_{a \, \cdot} & 0 \ar@{.>}[l] \ar@{.>}[d] \\ e_2C & = & \kk  & 0 \ar@{.>}[l] }
\]
for some  $a \in \kk\backslash\{0\}$.
Thus the representation of  $C$ is given by
\[
\xymatrix@!=10pt{
\kk  \ar[d]_{a \, \cdot} & 0 \ar@{.>}[l] \ar@{.>}[d] \\
\kk  & 0 \ar@{.>}[l]. \\
}
\]
It is easy to check that this bimodule is isomorphic to $\Lambda e_{1} \otimes_{\kk} e_{1} \Lambda$. 

%\exnumber{2-1} in the table \ref{class_asid}.

\end{example}

%%%%%%%%%%%%%%%%%%%%%%%%
\begin{example} 
We deal with the case $\sfT_{5}= 0$. 
We note that this is precesely the case where $A = \Lambda \oplus C$ has finite global dimension.
 
Let $C$ be a bimodule. 
Then  $\cT= - \lotimes_{\Lambda} C$  always acts on $\sfT_{5} = 0$ as an equivalence.  
Thus we only have to study the condition that $\cT$  nilpotently acts on $\sfT^{\perp}= \sfD^{\mrb}(\mod \Lambda)$.
We assume that $C$ satisfies this condition. 
Observes that  at least one of $\cT(P_1)=0$, $\cT(P_2)=0$ and $\cT(I_2)=0$ holds.

We discuss the case $\cT(P_1)=0$.
Since $\cT(P_{2}) = e_{2} C$ belongs to $\mod \Lambda$, 
we have 
$\cT(P_2) \cong P_1^{n} \oplus P_2^{m} \oplus I_2^{l}$ for some $l,m,n \in \NN$. 
By applying $\cT$ to the triangle \eqref{tri-alpha}, we have an isomorphism $\cT(P_2) \cong \cT(I_2)$. 
Since $\cT^{a}(P_{2}) = 0$ for $a \gg 0$, we see that  $m$ and $l$ must be zero, namely $\cT(P_2) \cong P_1^{n}$.

We conclude that  the right module structure of $C$ is
\[
e_1C = ( \xymatrix@!=10pt{ 0  & 0 \ar@{.>}[l]}), \
e_2C =P_1^{n}=( \xymatrix@!=10pt{ \kk^n  & 0 \ar@{.>}[l]}).
\]
Thus the representation of  $C$ is given by
\[
\xymatrix@!=10pt{
0  \ar@{.>}[d] & 0 \ar@{.>}[l] \ar@{.>}[d] \\
\kk^n  & 0 \ar@{.>}[l]. \\
}
\]
It is clear that this bimodule is isomorphic to $(\Lambda e_{2}  \otimes_{\kk} e_{1}\Lambda)^{n}$. 

By the similar arguments,  
 in the case $\cT(P_2) =0$ or $\cT(I_2)=0$,  
we can show that $C$ is isomorphic to  one of  the bimodules 
listed in \exnumber{5}. 
%to \exnumber{5-2} (resp. \exnumber{5-3}).
Consequently we have three asid bimodules corresponding to $\sfT_{5}=0$.
\end{example}

\subsection{Classification of asid bimodules over the path algebra of $A_{3}$ quiver}\label{A3}

Let $\Lambda=\kk Q$ be the path algebra of $A_{3}$ quiver $Q$. 
\[
Q=\xymatrix{
1 & 2 \ar[l]_{\alpha} \ar[r]^{\beta} & 3.
}
\]
Below, 
we give  the list of asid bimodules $C$ over $\Lambda$ 
such that $A$ has infinite global dimension. 
The restriction $\gldim A = \infty$ is put only for the sake of space.

We use the quiver representation as below to exhibit  a bimodule $C$.  
\[
\xymatrix@R=20pt@C=8pt{
e_1Ce_1 \ar[d]_{\alpha \, \cdot} && e_1Ce_2 \ar[ll]_{\cdot \, \alpha} \ar[rr]^{\cdot \, \beta}  \ar[d]_{\alpha \, \cdot} && e_1Ce_3  \ar[d]^{\alpha \, \cdot}  \\
e_2Ce_1 && e_2Ce_2  \ar[ll]_{\cdot \, \alpha} \ar[rr]^{\cdot \, \beta} && e_2Ce_3 \\
e_3Ce_1 \ar[u]^{\beta \, \cdot} && e_3Ce_2  \ar[ll]^{\cdot \, \alpha} \ar[rr]_{\cdot \, \beta} \ar[u]^{\beta \, \cdot}  && e_3Ce_3 \ar[u]_{\beta \, \cdot}  
}
\]
%of $\kk$-vector spaces $e_iCe_j$ and $\kk$-linear maps between them such that each square is commutative. 
%This is nothing but a quiver representation of $C$ as a $\Lambda^{\mre}$-module.
%quiver representation as $A^e$-module
%\[
%\xymatrix@!=14pt{
%(1,1) \ar[d]_{\alpha \otimes e_1} && (1,2) \ar[ll]_{e_1 \otimes \alpha } \ar[rr]^{e_1 \otimes \beta}  \ar[d]_{\alpha \otimes e_2} && (1,3)  \ar[d]^{\alpha \otimes e_3}  \\
%(2,1) &&(2,2)  \ar[ll]_{e_2 \otimes \alpha } \ar[rr]^{e_2 \otimes \beta} &&(2,3) \\
%(3,1) \ar[u]^{\beta \otimes e_1} &&(3,2)  \ar[ll]^{e_3 \otimes \alpha } \ar[rr]_{e_3 \otimes \beta} \ar[u]^{\beta \otimes e_2}  &&(3,3) \ar[u]_{\beta \otimes e_3}  
%}
%\]

In the list, the asid bimodules exhibited in \exnumber{i-s}
and \exnumber{i-t} have the same admissible subcategory 
as asid subcategories.

%\bigskip

%%%%%%%%%%%%%%%%%%%%%%%%%%%%%%%%%%%%%%%%%
%%%%%%%%%%%%%%%%%%%%%%%%%%%%%%%%%%%%%%%%% list of asid modules
%%%%%%%%%%%%%%%%%%%%%%%%%%%%%%%%%%%%%%%%%
\begin{longtable}[l]{llll}%
\caption{List of asid bimodules over $1 \leftarrow 2 \rightarrow 3$}\label{class_asid A3} \\ 
%%%%%%%%%%%%%%%%%%%%%%%%
\exnumber{1-1} \hspace{1mm}
$\xymatrix@!=7pt{
\kk \ar[d] & 0 \ar@{.>}[l] \ar@{.>}[r] \ar@{.>}[d] & 0  \ar@{.>}[d] \\
\kk & \kk \ar[l] \ar[r] & \kk \\
0 \ar@{.>}[u] & 0 \ar@{.>}[l] \ar@{.>}[r] \ar@{.>}[u] & \kk \ar[u]
}$
&
%%%%%%%%%%%%%%%%%%%%%%%%
\exnumber{1-2} \hspace{1mm}
$\xymatrix@!=7pt{
0 \ar@{.>}[d] & 0 \ar@{.>}[l] \ar@{.>}[r] \ar@{.>}[d] & \kk  \ar[d] \\
\kk & \kk \ar[l] \ar[r] & \kk \\
\kk \ar[u] & 0 \ar@{.>}[l] \ar@{.>}[r] \ar@{.>}[u] & 0 \ar@{.>}[u]
}$
&
%%%%%%%%%%%%%%%%%%%%%%%%
\exnumber{1-3} \hspace{1mm}
$\xymatrix@!=7pt{
0 \ar@{.>}[d] & \kk \ar@{.>}[l] \ar[r] \ar[d] & \kk  \ar@{.>}[d] \\
0& \kk \ar@{.>}[l] \ar@{.>}[r] & 0 \\
\kk \ar@{.>}[u] &\kk \ar[l] \ar@{.>}[r] \ar[u] &0 \ar@{.>}[u]
}$ 
&
%%%%%%%%%%%%%%%%%%%%%%%%
\exnumber{1-4} \hspace{1mm}
$\xymatrix@!=7pt{
\kk \ar@{.>}[d] &\kk \ar[l] \ar@{.>}[r] \ar[d] & 0  \ar@{.>}[d] \\
0&\kk \ar@{.>}[l] \ar@{.>}[r] & 0 \\
0 \ar@{.>}[u] &\kk \ar@{.>}[l] \ar[r] \ar[u] & \kk \ar@{.>}[u]
}$ 
\\
& & & \\
%\end{longtable}
%\begin{longtable}[l]{llll}
%%%%%%%%%%%%%%%%%%%%%%%%
\exnumber{2-1} \hspace{1mm}
$\xymatrix@!=7pt{
\kk \ar[d] & \kk \ar[l] \ar@{.>}[r] \ar[d] & 0  \ar@{.>}[d] \\
\kk & \kk \ar[l] \ar@{.>}[r] & 0 \\
\kk \ar[u] & 0 \ar@{.>}[l] \ar@{.>}[r] \ar@{.>}[u] & 0 \ar@{.>}[u]
}$ 
&
%%%%%%%%%%%%%%%%%%%%%%%%
\exnumber{2-2} \hspace{1mm}
$\xymatrix@!=7pt{
0 \ar@{.>}[d] & \kk \ar@{.>}[l] \ar@{.>}[r] \ar[d] & 0 \ar@{.>}[d] \\
0 & \kk \ar@{.>}[l] \ar@{.>}[r] & 0 \\
\kk \ar@{.>}[u] & \kk \ar[l] \ar@{.>}[r] \ar[u] & 0 \ar@{.>}[u]
}$
&
%%%%%%%%%%%%%%%%%%%%%%%%
\exnumber{3-1} \hspace{1mm}
$\xymatrix@!=7pt{
0 \ar@{.>}[d] & 0 \ar@{.>}[l] \ar@{.>}[r] \ar@{.>}[d] & \kk  \ar[d] \\
0&\kk \ar@{.>}[l] \ar[r] & \kk \\
0 \ar@{.>}[u] &\kk \ar@{.>}[l] \ar[r] \ar[u] & \kk \ar[u]
}$ 
&
%%%%%%%%%%%%%%%%%%%%%%%%
\exnumber{3-2} \hspace{1mm}
$\xymatrix@!=7pt{
0 \ar@{.>}[d] &\kk \ar@{.>}[l] \ar[r] \ar[d] & \kk \ar@{.>}[d] \\
0&\kk \ar@{.>}[l] \ar@{.>}[r] & 0 \\
0 \ar@{.>}[u] &\kk \ar@{.>}[l] \ar@{.>}[r] \ar[u] & 0 \ar@{.>}[u]
}$
\\
& & & \\
%%%%%%%%%%%%%%%%%%%%%%%%
\exnumber{4-1} \hspace{1mm}
$\xymatrix@!=7pt{
\kk \ar[d] & 0 \ar@{.>}[l] \ar@{.>}[r] \ar@{.>}[d] & 0  \ar@{.>}[d] \\
\kk & \kk \ar[l] \ar[r] & \kk \\
0 \ar@{.>}[u] & 0 \ar@{.>}[l] \ar@{.>}[r] \ar@{.>}[u] & 0 \ar@{.>}[u]
}$ 
&
%%%%%%%%%%%%%%%%%%%%%%%%
\exnumber{4-2} \hspace{1mm}
$\xymatrix@!=7pt{
\kk \ar@{.>}[d] & \kk \ar[l] \ar[r] \ar[d] & \kk \ar[d] \\
0 & \kk \ar@{.>}[l] \ar[r] & \kk \\
0 \ar@{.>}[u] & 0 \ar@{.>}[l] \ar@{.>}[r] \ar@{.>}[u] & 0 \ar@{.>}[u]
}$
&
%%%%%%%%%%%%%%%%%%%%%%%%
\exnumber{5-1} \hspace{1mm}
$\xymatrix@!=7pt{
0 \ar@{.>}[d] & 0 \ar@{.>}[l] \ar@{.>}[r] \ar@{.>}[d] & 0  \ar@{.>}[d] \\
\kk&\kk \ar[l] \ar[r] & \kk \\
0 \ar@{.>}[u] &0 \ar@{.>}[l] \ar@{.>}[r] \ar@{.>}[u] & \kk \ar[u]
}$ 
&
%%%%%%%%%%%%%%%%%%%%%%%%
\exnumber{5-2} \hspace{1mm}
$\xymatrix@!=7pt{
0 \ar@{.>}[d] &0 \ar@{.>}[l] \ar@{.>}[r] \ar@{.>}[d] & 0  \ar@{.>}[d] \\
\kk&\kk\ar[l] \ar@{.>}[r] & 0 \\
\kk \ar[u] &\kk \ar[l] \ar[r] \ar[u] &\kk \ar@{.>}[u]
}$ 
\\
& & & \\
%%%%%%%%%%%%%%%%%%%%%%%%
\exnumber{6-1} \hspace{1mm}
$\xymatrix@!=7pt{
\kk \ar[d] & 0 \ar@{.>}[l] \ar@{.>}[r] \ar@{.>}[d] & 0  \ar@{.>}[d] \\
\kk & 0 \ar@{.>}[l] \ar@{.>}[r] & \kk \\
0 \ar@{.>}[u] & 0 \ar@{.>}[l] \ar@{.>}[r] \ar@{.>}[u] & \kk \ar[u]
}$ 
&
%%%%%%%%%%%%%%%%%%%%%%%%
\exnumber{6-2} \hspace{1mm}
$\xymatrix@!=7pt{
0 \ar@{.>}[d] & 0 \ar@{.>}[l] \ar@{.>}[r] \ar@{.>}[d] & \kk \ar[d] \\
\kk & 0 \ar@{.>}[l] \ar@{.>}[r] & \kk \\
\kk \ar[u] & 0 \ar@{.>}[l] \ar@{.>}[r] \ar@{.>}[u] & 0 \ar@{.>}[u]
}$
&
%%%%%%%%%%%%%%%%%%%%%%%%
\exnumber{7-1} \hspace{1mm}
$\xymatrix@!=7pt{
0 \ar@{.>}[d] & \kk \ar@{.>}[l] \ar[r] \ar@{.>}[d] & \kk  \ar@{.>}[d] \\
0 & 0 \ar@{.>}[l] \ar@{.>}[r] & 0 \\
\kk \ar@{.>}[u] & \kk \ar[l] \ar@{.>}[r] \ar@{.>}[u] & 0 \ar@{.>}[u]
}$ 
&
%%%%%%%%%%%%%%%%%%%%%%%%
\exnumber{7-2} \hspace{1mm}
$\xymatrix@!=7pt{
\kk \ar@{.>}[d] & \kk \ar[l] \ar@{.>}[r] \ar@{.>}[d] & 0  \ar@{.>}[d] \\
0 & 0 \ar@{.>}[l] \ar@{.>}[r] & 0 \\
0 \ar@{.>}[u] &\kk \ar@{.>}[l] \ar[r] \ar@{.>}[u] &\kk\ar@{.>}[u]
}$ 
\\
& & & \\
%%%%%%%%%%%%%%%%%%%%%%%%
\exnumber{8-1} \hspace{1mm}
$\xymatrix@!=7pt{
0 \ar@{.>}[d] & 0 \ar@{.>}[l] \ar@{.>}[r] \ar@{.>}[d] & 0  \ar@{.>}[d] \\
\kk^n & 0 \ar@{.>}[l] \ar@{.>}[r] & \kk \\
0 \ar@{.>}[u] & 0 \ar@{.>}[l] \ar@{.>}[r] \ar@{.>}[u] & \kk \ar[u]
}$ 
&
%%%%%%%%%%%%%%%%%%%%%%%%
\exnumber{8-2} \hspace{1mm}
$\xymatrix@!=7pt{
\kk^n \ar@{.>}[d] & \kk^n \ar[l] \ar@{.>}[r] \ar@{.>}[d] & 0  \ar@{.>}[d] \\
0 & 0 \ar@{.>}[l] \ar@{.>}[r] & \kk \\
0 \ar@{.>}[u] & 0 \ar@{.>}[l] \ar@{.>}[r] \ar@{.>}[u] & \kk \ar[u]
}$
&
%%%%%%%%%%%%%%%%%%%%%%%%
\exnumber{8-3} \hspace{1mm}
$\xymatrix@!=7pt{
0 \ar@{.>}[d] &\kk^n \ar@{.>}[l] \ar@{.>}[r] \ar[d] & 0  \ar@{.>}[d] \\
0&\kk^n \ar@{.>}[l] \ar@{.>}[r] & \kk \\
0 \ar@{.>}[u] &0 \ar@{.>}[l] \ar@{.>}[r] \ar@{.>}[u] &\kk \ar[u]
}$ 
&
%%%%%%%%%%%%%%%%%%%%%%%%
\exnumber{9-1} \hspace{1mm}
$\xymatrix@!=7pt{
\kk \ar[d] &0 \ar@{.>}[l] \ar@{.>}[r] \ar@{.>}[d] & 0  \ar@{.>}[d] \\
\kk&0 \ar@{.>}[l] \ar@{.>}[r] & \kk^n \\
0 \ar@{.>}[u] &0 \ar@{.>}[l] \ar@{.>}[r] \ar@{.>}[u] & 0 \ar@{.>}[u]
}$ 
\\
& & & \\
%%%%%%%%%%%%%%%%%%%%%%%%
\exnumber{9-2} \hspace{1mm}
$\xymatrix@!=7pt{
\kk \ar[d] & 0 \ar@{.>}[l] \ar@{.>}[r] \ar@{.>}[d] & 0  \ar@{.>}[d] \\
\kk & 0 \ar@{.>}[l] \ar@{.>}[r] & 0 \\
0 \ar@{.>}[u] & \kk^n \ar@{.>}[l] \ar[r] \ar@{.>}[u] & \kk^n \ar@{.>}[u]
}$ 
&
%%%%%%%%%%%%%%%%%%%%%%%%
\exnumber{9-3} \hspace{1mm}
$\xymatrix@!=7pt{
\kk \ar[d] & 0 \ar@{.>}[l] \ar@{.>}[r] \ar@{.>}[d] & 0  \ar@{.>}[d] \\
\kk & \kk^n \ar@{.>}[l] \ar@{.>}[r] & 0 \\
0 \ar@{.>}[u] & \kk^n \ar@{.>}[l] \ar@{.>}[r] \ar[u] & 0 \ar@{.>}[u]
}$
& & \\
\end{longtable}

\begin{longtable}[l]{ll}
%%%%%%%%%%%%%%%%%%%%%%%%
\exnumber{10-1} \hspace{1mm}
$\xymatrix@!=7pt{
0 \ar@{.>}[d] & 0 \ar@{.>}[l] \ar@{.>}[r] \ar@{.>}[d] & 0 \ar@{.>}[d] \\
\kk^n & 0 \ar@{.>}[l] \ar@{.>}[r] & 0 \\
\kk^n \ar[u] & 0 \ar@{.>}[l] \ar@{.>}[r] \ar@{.>}[u] & 0 \ar@{.>}[u]
}
\ \raisebox{-2.6eM}{\(\oplus\)} \
\xymatrix@!=7pt{
0 \ar@{.>}[d] &0 \ar@{.>}[l] \ar@{.>}[r] \ar@{.>}[d] & 0  \ar@{.>}[d] \\
0&0 \ar@{.>}[l] \ar@{.>}[r] & 0 \\
\kk \ar@{.>}[u] &\kk \ar[l] \ar@{.>}[r] \ar@{.>}[u] &0 \ar@{.>}[u]
}$ 
\hspace{10mm}
&
%%%%%%%%%%%%%%%%%%%%%%%%
\exnumber{10-2} \hspace{1mm}
$\xymatrix@!=7pt{
0 \ar@{.>}[d] &\kk^n \ar@{.>}[l] \ar[r] \ar[d] & \kk^n \ar[d] \\
0&\kk^n \ar@{.>}[l] \ar[r] & \kk^n \\
0 \ar@{.>}[u] &\kk^n \ar@{.>}[l] \ar[r] \ar[u] & \kk^n \ar[u]
}
\ \raisebox{-2.6eM}{\(\oplus\)} \
\xymatrix@!=7pt{
0 \ar@{.>}[d] &0 \ar@{.>}[l] \ar@{.>}[r] \ar@{.>}[d] & 0  \ar@{.>}[d] \\
0&0 \ar@{.>}[l] \ar@{.>}[r] & 0 \\
\kk \ar@{.>}[u] &\kk \ar[l] \ar@{.>}[r] \ar@{.>}[u] &0 \ar@{.>}[u]
}$
\\
& \\
%%%%%%%%%%%%%%%%%%%%%%%%
\exnumber{10-3} \hspace{1mm}
$\xymatrix@!=7pt{
\kk^n \ar@{.>}[d] & \kk^n \ar[l] \ar[r] \ar@{.>}[d] & \kk^n  \ar@{.>}[d] \\
0 & 0 \ar@{.>}[l] \ar@{.>}[r] & 0 \\
\kk \ar@{.>}[u] & \kk \ar[l] \ar@{.>}[r] \ar@{.>}[u] & 0 \ar@{.>}[u]
}$
&
%%%%%%%%%%%%%%%%%%%%%%%%
\exnumber{11-1} \hspace{1mm}
$\xymatrix@!=7pt{
0 \ar@{.>}[d] & 0 \ar@{.>}[l] \ar@{.>}[r] \ar@{.>}[d] & \kk^n \ar[d] \\
0 & 0 \ar@{.>}[l] \ar@{.>}[r] & \kk^n \\
0 \ar@{.>}[u] & 0 \ar@{.>}[l] \ar@{.>}[r] \ar@{.>}[u] & 0 \ar@{.>}[u]
}
\ \raisebox{-2.6eM}{\(\oplus\)} \
\xymatrix@!=7pt{
0 \ar@{.>}[d] &\kk \ar@{.>}[l] \ar[r] \ar@{.>}[d] & \kk \ar@{.>}[d] \\
0&0 \ar@{.>}[l] \ar@{.>}[r] & 0 \\
0 \ar@{.>}[u] &0 \ar@{.>}[l] \ar@{.>}[r] \ar@{.>}[u] &0 \ar@{.>}[u]
}$
\\
& \\ 
%%%%%%%%%%%%%%%%%%%%%%%%
\exnumber{11-2} \hspace{1mm}
$\xymatrix@!=7pt{
\kk^n \ar[d] & \kk^n \ar[l] \ar@{.>}[r] \ar[d] & 0 \ar@{.>}[d] \\
\kk^n & \kk^n \ar[l] \ar@{.>}[r] & 0 \\
\kk^n \ar[u] & \kk^n \ar[l] \ar@{.>}[r] \ar[u] & 0 \ar@{.>}[u]
}
\ \raisebox{-2.6eM}{\(\oplus\)} \
\xymatrix@!=7pt{
0 \ar@{.>}[d] & \kk \ar@{.>}[l] \ar[r] \ar@{.>}[d] &\kk  \ar@{.>}[d] \\
0&0 \ar@{.>}[l] \ar@{.>}[r] & 0 \\
0 \ar@{.>}[u] &0 \ar@{.>}[l] \ar@{.>}[r] \ar@{.>}[u] &0 \ar@{.>}[u]
}$
&
%%%%%%%%%%%%%%%%%%%%%%%%
\exnumber{11-3} \hspace{1mm}
$\xymatrix@!=7pt{
0 \ar@{.>}[d] & \kk \ar@{.>}[l] \ar[r] \ar@{.>}[d] & \kk  \ar@{.>}[d] \\
0 & 0 \ar@{.>}[l] \ar@{.>}[r] & 0 \\
\kk^n \ar@{.>}[u] & \kk^n \ar[l] \ar[r] \ar@{.>}[u] & \kk^n \ar@{.>}[u]
}$
\\
& \\
%%%%%%%%%%%%%%%%%%%%%%%%
\exnumber{12-1} \hspace{1mm}
$\xymatrix@!=7pt{
0 \ar@{.>}[d] & 0 \ar@{.>}[l] \ar@{.>}[r] \ar@{.>}[d] & 0  \ar@{.>}[d] \\
\kk & \kk \ar[l] \ar[r] & \kk \\
\kk^n \ar@{.>}[u] & 0 \ar@{.>}[l] \ar@{.>}[r] \ar@{.>}[u] & 0 \ar@{.>}[u]
}$
& 
%%%%%%%%%%%%%%%%%%%%%%%%
\exnumber{12-2} \hspace{1mm}
$\xymatrix@!=7pt{
0 \ar@{.>}[d] & 0 \ar@{.>}[l] \ar@{.>}[r] \ar@{.>}[d] & 0 \ar@{.>}[d] \\
0 & 0 \ar@{.>}[l] \ar@{.>}[r] & 0 \\
\kk^n \ar@{.>}[u] & 0 \ar@{.>}[l] \ar@{.>}[r] \ar@{.>}[u] & 0 \ar@{.>}[u]
}
\ \raisebox{-2.6eM}{\(\oplus\)} \
\xymatrix@!=7pt{
0 \ar@{.>}[d] & 0 \ar@{.>}[l] \ar@{.>}[r] \ar@{.>}[d] & 0 \ar@{.>}[d] \\
\kk &\kk \ar[l] \ar[r] & \kk \\
\kk \ar[u] &0 \ar@{.>}[l] \ar@{.>}[r] \ar@{.>}[u] &0 \ar@{.>}[u]
}$
\\
& \\
%%%%%%%%%%%%%%%%%%%%%%%%
\exnumber{12-3} \hspace{1mm}
$\xymatrix@!=7pt{
0 \ar@{.>}[d] & 0 \ar@{.>}[l] \ar@{.>}[r] \ar@{.>}[d] & \kk^n  \ar@{.>}[d] \\
\kk & \kk \ar[l] \ar[r] & \kk \\
0 \ar@{.>}[u] & 0 \ar@{.>}[l] \ar@{.>}[r] \ar@{.>}[u] & 0 \ar@{.>}[u]
}$
&
%%%%%%%%%%%%%%%%%%%%%%%%
\exnumber{12-4} \hspace{1mm}
$\xymatrix@!=7pt{
0 \ar@{.>}[d] & 0 \ar@{.>}[l] \ar@{.>}[r] \ar@{.>}[d] & \kk^n \ar@{.>}[d] \\
0 & 0 \ar@{.>}[l] \ar@{.>}[r] & 0 \\
0 \ar@{.>}[u] & 0 \ar@{.>}[l] \ar@{.>}[r] \ar@{.>}[u] & 0 \ar@{.>}[u]
}
\ \raisebox{-2.6eM}{\(\oplus\)} \
\xymatrix@!=7pt{
0 \ar@{.>}[d] & 0 \ar@{.>}[l] \ar@{.>}[r] \ar@{.>}[d] & \kk \ar[d] \\
\kk &\kk \ar[l] \ar[r] & \kk \\
0 \ar@{.>}[u] &0 \ar@{.>}[l] \ar@{.>}[r] \ar@{.>}[u] &0 \ar@{.>}[u]
}$
\\
& \\
%%%%%%%%%%%%%%%%%%%%%%%%
\exnumber{13-1} \hspace{1mm}
$\xymatrix@!=7pt{
0 \ar@{.>}[d] & \kk^n \ar@{.>}[l] \ar[r] \ar[d] & \kk^n \ar[d] \\
0 & \kk^n \ar@{.>}[l] \ar[r] & \kk^n \\
0 \ar@{.>}[u] & 0 \ar@{.>}[l] \ar@{.>}[r] \ar@{.>}[u] & 0 \ar@{.>}[u]
}
\ \raisebox{-2.6eM}{\(\oplus\)} \
\xymatrix@!=7pt{
0 \ar@{.>}[d] & \kk \ar@{.>}[l] \ar@{.>}[r] \ar[d] & 0 \ar@{.>}[d] \\
0 &\kk \ar@{.>}[l] \ar@{.>}[r] & 0 \\
0 \ar@{.>}[u] & \kk \ar@{.>}[l] \ar@{.>}[r] \ar[u] &0 \ar@{.>}[u]
}$
&
%%%%%%%%%%%%%%%%%%%%%%%%
\exnumber{13-2} \hspace{1mm}
$\xymatrix@!=7pt{
0 \ar@{.>}[d] & 0 \ar@{.>}[l] \ar@{.>}[r] \ar@{.>}[d] & 0 \ar@{.>}[d] \\
\kk^n & \kk^n \ar[l] \ar@{.>}[r] & 0 \\
\kk^n \ar[u] & \kk^n \ar[l] \ar@{.>}[r] \ar[u] & 0 \ar@{.>}[u]
}
\ \raisebox{-2.6eM}{\(\oplus\)} \
\xymatrix@!=7pt{
0 \ar@{.>}[d] & \kk \ar@{.>}[l] \ar@{.>}[r] \ar[d] & 0 \ar@{.>}[d] \\
0 &\kk \ar@{.>}[l] \ar@{.>}[r] & 0 \\
0 \ar@{.>}[u] & \kk \ar@{.>}[l] \ar@{.>}[r] \ar[u] &0 \ar@{.>}[u]
}$
\end{longtable}

\begin{remark}
In these examples, 
for every admissible subcategory $\sfT$, 
there exsits  at least one  asid bimodule $C$ 
having $\sfT$ as the asid subcategory.
However this is not the case for general algebras. There exists an algebra $\Lambda$ such that
$\sfK^{\mrb}(\proj \Lambda)$ has an admissible subcategory $\sfT$ which is not the asid subcategory 
for any asid bimodule.

Let $\Lambda$ be the path algebra $\kk[ 1 \xleftarrow{\alpha} 2 \xleftarrow{\beta} 3]$ 
with the relation $\beta \alpha$. 
We denote by $P_{1}$ the indecomposable projective $\Lambda$-module 
corresponds to the vertex $1$ and
 by $S_{3}$ the simple $\Lambda$-module corresponds to the vertex $3$. 
Then, the admissible subcategory $\thick (P_{1} \oplus S_{3}) \subset \sfD^{\mrb}(\mod \Lambda)$ 
 is not the asid subcategory for any asid bimodule.
\end{remark}

\end{document}